\documentclass[12pt]{amsart}
\usepackage{epsf}
\usepackage{psfrag}
\usepackage{fullpage}
\usepackage{float}
\usepackage{mathrsfs}
\usepackage{amsfonts}
\usepackage[centertags]{amsmath}
\usepackage{amssymb}
\usepackage{amsthm}
\usepackage{graphicx}
\usepackage{float}
\usepackage[all]{xy}
\usepackage{tikz-cd}
\usepackage{comment}
\usepackage{tikz}
\usetikzlibrary[shapes]
\usepackage{multirow}
\usepackage{caption}
\usepackage{xparse}
\usepackage{lscape}
\usepackage{dsfont}
\usepackage{inputenc}
\usepackage{subcaption} 
\usepackage{enumerate,enumitem}
\usepackage[margin=1in]{geometry} 
\usetikzlibrary{matrix,arrows,decorations.pathmorphing}
\usepackage{listing} 
\usepackage{hyperref}
\hypersetup{
bookmarksnumbered,
pdfstartview={FitH},
breaklinks=true,
linkcolor=blue,
urlcolor=blue,
citecolor=blue,
bookmarksdepth=2
}
\usepackage{cleveref}
\usepackage{ytableau}
\usepackage{adjustbox}
\usepackage{xurl}

\theoremstyle{plain}
   \newtheorem{theorem}{Theorem}[section]
   \newtheorem{proposition}[theorem]{Proposition}
   \newtheorem{lemma}[theorem]{Lemma}
   \newtheorem{corollary}[theorem]{Corollary}

\theoremstyle{definition}
   \newtheorem{definition}[theorem]{Definition} 
   \newtheorem{example}[theorem]{Example} 
   \newtheorem{question}[theorem]{Question}
   
   \newtheorem{remark}[theorem]{Remark}

\numberwithin{equation}{section}

\newcommand{\ZZ}{{\mathbb {Z}}}

\newcommand{\ch}{{\operatorname{ch}}}

\newcommand{\SSYT}{{\operatorname{SSYT}}}
\DeclareMathOperator{\flag}{\mathcal{F}\hspace{-1.6pt}\ell}

\newcommand{\CC}{{\mathbb{C}}}

\newcommand{\kk}{\Bbbk}

\DeclareMathOperator{\Gr}{Gr}

\begin{document}

\title{Cluster structures on spinor helicity and momentum twistor varieties}

\author{Lara Bossinger, Jian-Rong Li}
\address{Instituto de Matem\'{a}ticas Unidad Oaxaca, Universidad Nacional Aut\'{o}noma de M\'{e}xico, Le\'{o}n 2, altos, Centro Hist\'{o}rico, 68000 Oaxaca, Mexico}
\email{lara@im.unam.mx}
\address{Jian-Rong Li, Faculty of Mathematics, University of Vienna, Oskar-Morgenstern-Platz 1, 1090 Vienna, Austria}
\email{jianrong.li@univie.ac.at}
\date{}

%MSC 13F60

\begin{abstract}
We study the homogeneous coordinate rings of partial flag varieties and Grassmannians in their Plücker embeddings and exhibit an embedding of the former into the latter.
Both rings are cluster algebras and the embedding respects the cluster algebra structures in the sense that there exists a seed for the Grassmannian that restricts to a seed for the partial flag variety (i.e. it is obtained by freezing and deleting some cluster variables).

The motivation for this project stems from the application of cluster algebras in scattering amplitudes: spinor helicity and momentum twistor varieties describe massless scattering without assuming dual conformal symmetry. Both may be obtained from Grassmanninas which model the dual conformal case. They are instances of partial flag varieties and their cluster structures reveal information for the scattering amplitudes.
As an application of our main result we exhibit the relation between these cluster algebras.
\end{abstract}

% We study cluster structures on spinor helicity and momentum twistor varieties which describe the kinematic spaces of massless scattering (with or without dual conformal symmetry).
% Both are instances of partial flag varieties. \textcolor{red}{We exhibit embeddings of the cluster algebra of any partial flag variety into the cluster algebra of a sufficiently large Grassmannian and show how the former is obtained from the latter by freezing certain cluster variables and deleting some cluster variables, that is by means of a restricted seed. In particular, we obtain cluster structures on spinor helicity and momentum twistor varieties from cluster structures on Grassmannians.}
%We exhibit embeddings of the corresponding cluster algebras into cluster algebras of sufficiently large Grassmannians and show how the former is obtained from the latter by freezing certain cluster variables and deleting some cluster variables.

\maketitle

%\tableofcontents

\section{Introduction}

\subsection{Motivation from particle physics}
In the study of scattering amplitudes in quantum field theories for massless scattering the kinematic space modeling particle interactions is represented by two complex matrices $\lambda$ and $\tilde \lambda$ of size $2\times n$ with property that $\lambda\tilde\lambda^T$ is zero. 
This is known as the \emph{spinor-helicity formalism} \cite[Section 1.8]{Book:ScatAmplQFT}.
In this situation the $2\times 2$-minors of $\lambda$ are typically denoted by $\langle ij\rangle$ and the $2\times 2$-minors of $\tilde\lambda$ by $[ij]$ for $1\le i<j\le n$ (notice that $\langle ij\rangle=-\langle ji\rangle$ and similarly for $[ij]$).
Each class of minors satisfy \emph{quadratic Plücker relations} also known as \emph{Schouten identities}, namely for all $1\le i<j<k<l\le n$
\[
[ij][kl]-[ik][jl]+[il][jk]=0 \quad \text{and} \quad \langle ij\rangle\langle kl\rangle-\langle ik\rangle\langle jl\rangle+\langle il\rangle\langle jk\rangle=0
\]
and additionally \emph{momentum conservation} for all $1\le i,j\le n$
\[
\sum_{s=1}^n \langle is\rangle [sj]=0.
\]
Adding the genericity assumptions that $\lambda$ and $\tilde \lambda$ are of full rank the space determined by the above equations is known as the \emph{\bf spinor helicity variety} \cite{Bernd_spinor-helicity} and we denote it by $\mathcal{SH}_{n}$ (the notation $\mathcal{SH}_{n}$ corresponds to $SH(2,n,0)$ in \cite{Bernd_spinor-helicity}, the spaces of interest to the authors in \emph{loc.cit.} are generalizations).
It is shown in \cite[Proposition 2.5]{Bernd_spinor-helicity} that $\mathcal{SH}_n$ is isomorphic to the \emph{partial flag variety} $\flag_{2,n-2;n}$ (see \eqref{eq:def flag} for the definition) with isomorphism given by
\begin{eqnarray}\label{eq:SH to flag2,n-2,n}
    \langle i j \rangle \mapsto P_{ij},\quad \text{and} \quad [ij]\mapsto (-1)^{i+j-1}P_{[n]-\{i,j\}},
\end{eqnarray}
here $P_{ij}$ and $P_{[n]-\{i,j\}}$ denote \emph{Plücker coordinates} of $\flag_{2,n-2;n}$, see \eqref{eq:Plücker embed} below.
From the physics point of view the scattering process may also be described in terms of \emph{momentum twistors} $Z_1,\dots,Z_n\in \mathbb CP^3$ representing the particles in Minkowski space. 
In the planar model a fixed order among the $Z_i$ is assumed which yields a cyclic symmetry.
Interpreting the coordinates of $Z_i$ as column vectors we obtain a $4\times n$ matrix up to column rescaling.
If the system has \emph{dual conformal symmetry} a parametrization is given by maximal minors of this matrix 
\[
\langle ijkl\rangle:=\det(Z_iZ_jZ_kZ_l)
\]
representing a Plücker coordinate $P_{ijkl}$ in the Grassmannian $\Gr_{4;n}$ (up to sign).
If the system does not have dual conformal symmetry, i.e. it is \emph{non-dual conformal invariant} (or NDCI for short) we add an \emph{infinity twistor} in form of a line $\ell_{\infty}$. 
The line $\ell_{\infty}$ may be understood as the line spanned by two additional points $Z_{n+1},Z_{n+2}\in\mathbb CP^3$. 
In particular, we may now interpret the coordinates of $Z_1,\dots,Z_{n+2}$ as a $4\times (n+2)$ matrix.
However the kinematics do not depend on each of the two points $Z_{n+1},Z_{n+2}$ separately, but rather only on the line $\ell_{\infty}$: moving either point along the line (as long as it does not collide with the other) does not change the configuration.
This fact is reflected in the parametrization as follows:
consider only those minors $\langle ijkl\rangle $ that satisfy
\[
 n+1,n+2\in\{i,j,k,l\} \quad\text{or}\quad n+1,n+2\not\in\{i,j,k,l\}.
\]
The former are denoted by $\langle ij\rangle$ omitting the indices $n+1,n+2$ from the set.
More precisely, we define a map of polynomial rings 
\[
\phi_n:\kk[\langle ij \rangle,\langle ijkl\rangle:i,j,k,l\in [n]] \to \kk[z_{a,b}:a\in [4],b\in [n+2]]
\]
by $\langle ij\rangle \mapsto \det((z_{a,b})_{a\in[4],b\in \{i,j,n+1,n+2\}})$ and $\langle ijkl\rangle \mapsto \det((z_{a,b})_{a\in[4],b\in \{i,j,k,l\}})$. Let $J_n=\ker(\phi_n)$. The vanishing set of the ideal $J_n$ is naturally defined in a product of two projective spaces $\mathbb P^{\binom{n}{2}-1}\times \mathbb P^{\binom{n}{4}-1}$ determined by the bigrading on $\kk[\langle ij \rangle,\langle ijkl\rangle:i,j,k,l\in [n]]$.
We define the {\bf momentum twistor variety $\mathcal{MT}_n$} as $V(J_n)\subset \mathbb P^{\binom{n}{2}-1}\times \mathbb P^{\binom{n}{4}-1}$.
Identifying 
\begin{equation}\label{eq:MH to flag2,4,n}
\langle ij\rangle\mapsto P_{ij}, \quad \text{and}\quad \langle ijkl \rangle \mapsto P_{ijkl}, 
\end{equation}
yields an isomorphism between $\mathcal{MT}_n$ and the partial flag variety $\flag_{2,4;n}$ by Proposition \ref{proposition:embed}.

The change of parametrization between spinor helicity and momentum twistor variables is known, however highly nontrivial, see \emph{e.g.} \cite[Equation 5.49]{EH15}. 
Part of the map is given by
$\langle ij\rangle \mapsto \langle ij\rangle$ and 
\begin{eqnarray}\label{eq:MT to SH}
& \langle i-1,i, j, j+1\rangle\mapsto [ij]\langle ij\rangle\langle i-1,i\rangle\langle j, j+1\rangle,
\end{eqnarray}
see \cite[Equation 5.52]{EH15}.
In the case of $n=5$ there is a further isomorphism between the affine cone of $\mathcal{SH}_5$ and the affine cone of the Grassmannain $\Gr_{3;6}$, see e.g. \cite[Equation (6.12)]{Bossinger:2022_scattering_Grobner}.
To summarize, we have the following spaces modelling massless particle configurations and maps between them:
\begin{center}
\begin{tikzcd}
    \mathcal{SH}_n  \ar[d,"\eqref{eq:MT to SH}"] & \flag_{2,n-2;n} \ar[d,"\eqref{eq:MT to SH}"]\arrow{l}{\eqref{eq:SH to flag2,n-2,n}}[swap]{\sim}\ar[d]  \\
    \mathcal{MT}_n  & \flag_{2,4;n} \arrow{l}{\eqref{eq:MH to flag2,4,n}}[swap]{\sim} 
\end{tikzcd}
\end{center}

In recent years the computation of scattering amplitudes has been simplified by refraining from actually computing the Feynman integrals which are its summand, but rather computing the amplitude directly (\emph{e.g.} as the canonical form of the \emph{amplituhedron} \cite{Arkani-Hamed:2013jha, DFLP19, ELPTSW23, ELPTSW24}).
This has lead to a \emph{bootstrap}: the amplitude is expected to be a polylogarithmic function on the kinematic space with a certain singularity structure.
Predicticting the singularities eventually leads to a unique solution, see \emph{e.g.} \cite{BSS77, BDS05, CDMH16, CDFHM19, CDDHMP19, DDHMP17, DDH11, DDH12, DDHP13,  Drummond:2017ssj_cluster-adjacency, Drummond:2018dfd_beyondMHV, Drummond:2018caf_4loop, DDDP14, DH14, DHM16, DPS15, Golden:2013xva, GSO77, tHooft93}.
In the planar dual conformal case corresponding to $\Gr_{4;n}$ cluster algebras have been extremely useful as in many cases cluster variables are singularities of the amplitude. 
For non-planar non-dual cases however so far few is known, see \cite[Section 6]{Bossinger:2022_scattering_Grobner} for $n=5$. 

In this paper we explore the cluster structures of kinematic spaces in form of the spinor helicity variety $\mathcal{SH}_{n}\cong \flag_{2,n-2;n}$ and the momentum twistor variety $\mathcal{MT}_n \cong \flag_{2,4;n}$ and show how they can be obtained from the cluster structure on Grassmannians. Our results work for general partial flag varieties. 
To state our results, we first give some mathematical background.  

% \begin{theorem}\label{thm:MT cluster}
%     The momentum twistor variety has a cluster structure induced by the cluster structure on $\flag_{2,4;n}$. It may be obtained from the initial cluster of the Grassmannian $\Gr_{4;n+2}$ by two mutations followed by one freezing and three deletions.
%     %specializing three cluster variables to zero.
% \end{theorem}

% \begin{theorem}\label{thm:SH cluster}
%     The spinor helicity variety has a cluster structure, namely the one on $\flag_{2,n-2;n}$. 
%     It may be obtained from the initial cluster of the Grassmannian $\Gr_{n-2;2n-4}$ by $\frac{n(n-4)(n-5)}{2}$ mutations followed by freezing $n-5$ cluster variables and deleting $(n-3)(n-5)$ vertices.
% \end{theorem}
% These results are proven in \S\ref{sec:Cluster structures on MTn and SHn from Grassmannians}.
%Mathematically speaking these spaces are the flag varieties $\flag_{2,n-2;n}$ and $\flag_{2,4;n}$, respectively.
%In particular, the embedding $\flag_{2,4;n-2}\subset\Gr_{4;n}$ is the mathematical starting point of this project.

\subsection{Mathematical setup}
We refer to \cite[\S5]{Lakshmibai_GrassmannianV} and \cite{fulton_young} for more details on the material in this subsection.

For a positive integer $i$ write $[i]:=\{1,\dots,i\}$ and for $j>i$ write $[i,j]:=\{i,i+1,\dots,j\}$.
Let $n$ be a positive integer and let $1\le d_1<d_2<\dots<d_k< n$. To this data we associate the variety of partial flags of subspaces in $\kk^n$ ($\kk$ a field of characteristic zero):
\begin{eqnarray}\label{eq:def flag}
    \flag_{d_1,\dots,d_k;n}(\kk):=\left\{ 0\in V_1\subset V_2\subset \dots \subset V_k\subset \kk^n: \dim V_i=d_i\right\}.
\end{eqnarray}
When there is no confusion regarding the base field we drop it from the notation and simply write $\flag_{d_1,\dots,d_k;n}$. Moreover, it will be convenient to set $d_0:=0$ and $d_{k+1}:=n$.
If $k=1$ the associated partial flag variety is a \emph{Grassmannian} 
\[
\flag_{d;n}=\Gr_{d;n}:=\{0\in V\subset \kk^n: \dim V=d\}.
\]
%Let $N_i:=n+d_i-d_1$ and $I(d_i):=[N_i+1,N]$.
In this paper we are interested in a map between the (multi-)homogeneous coordinate rings of a partial flag variety in its Plücker embedding (recalled below), denoted $\kk[\flag_{d_1,\dots,d_k;n}]$ and the homogeneous coordinate ring of a Grassmannian in its Plücker embedding, denoted $\kk[\Gr_{d_k;N}]$ with $N=n+d_k-d_1$. 
%\textcolor{blue}{In  $\kk[\Gr_{d_k;N}]$ w_ define $J_{d_1,\dots,d_k;n}$ as the ideal generated by all Plücker coordinates $P_I$ for $I\subset [N]$ of cardinality $d_k$ satisfying
%\[
%[N_i+1,N]\not=(I\cap [N_i+1,N])\not=\varnothing.
%\]
We define
\begin{eqnarray}\label{eq:embed}
  \varphi^*: \kk[\flag_{d_1,\dots,d_k;n}]\to\kk[\Gr_{d_k;N}]%/J_{d_1,\dots,d_k;n},  
\end{eqnarray} 
given by the images of the Plücker coordinates with $I\subset [n]$ of cardinality $d_i, i\in[k]$ by
\[
%P_{I}\mapsto P_{I\cup[N_i+1,N]},
{P_{I}\mapsto P_{I\cup[n+1,n+d_k-d_i]}}.
\]
%with convention $[N+1,N]=\varnothing$. 
In Lemma~\ref{proposition:embed} we verify that the map is well defined.
We are slightly abusing notation by denoting Plücker coordinates of $\flag_{d_1,\dots,d_k;n}$ and $\Gr_{d_k,N}$ by the same symbol.
Both rings $\kk[\Gr_{d_k;N}]$ and $\kk[\flag_{d_1,\dots,d_k;n}]$ carry the structure of a cluster algebra \cite{FZ02}, see \cite{Sco06} for $\kk[\Gr_{d_k;N}]$ and \cite{GLS_partial-flags} for $\kk[\flag_{d_1,\dots,d_k;n}]$. Our main result is the following.

\begin{theorem}\label{thm:general flag}
    There exists a seed $s'$ for $\kk[\Gr_{d_k;N}]$ such that the initial seed of $\kk[\flag_{d_1,\dots,d_k;n}]$ is obtained from $s'$ by freezing and deleting.
\end{theorem}
 
In technical terms the initial seed of $\kk[\flag_{d_1,\dots,d_k;n}]$ is a \emph{restricted seed} of $s'$ for $\kk[\Gr_{d_k;N}]$, see \S\ref{sec:mutation sequence} and \cite{FWZ_ch4-5}. 
Theorem \ref{thm:general flag} is be proved in \S\ref{sec:mutation sequence} by exhibiting an explicit mutation sequence.

The momentum twistor variety has a cluster structure induced by the cluster structure on $\flag_{2,4;n}$ and the spinor helicity variety has a cluster structure, namely the one on $\flag_{2,n-2;n}$. 
Theorem~\ref{thm:general flag} the following results regarding the momentum twistor variety $\mathcal{MT}_n$ and the spinor helicity variety $\mathcal{SH}_n$ (more  details in \S\ref{sec:Cluster structures on MTn and SHn from Grassmannians}):

\begin{theorem}\label{thm:MT cluster}
     The momentum twistor variety has a cluster structure induced by the cluster structure on $\flag_{2,4;n}$. It may be obtained from the initial cluster of the Grassmannian $\Gr_{4;n+2}$ by three mutations followed by one freezing and three deletions.    %specializing three cluster variables to zero.
\end{theorem}

\begin{theorem}\label{thm:SH cluster}
     The spinor helicity variety has a cluster structure, induced by the one on $\flag_{2,n-2;n}$. 
     It may be obtained from the initial cluster of the Grassmannian $\Gr_{n-2;2n-4}$ by $\frac{1}{2}\left(n - 5\right)\, \left(n^2 - 10\, n + 30\right)$ mutations followed by freezing $n-5$ cluster variables and deleting $(n-3)(n-5)$ vertices.
 \end{theorem}

Cluster algebras of Grassmannians and flag varieties have been studied from several points of view including monoidal categorification \cite{HL10}. 
This has lead to a good combinatorial understanding of certain elements.
More precisely, there is a bijection between \emph{semistandard Young tableaux} \cite{fulton_young} (of suitable shape and up to equivalence of tableaux, see \S\ref{sec:initial tableaux}) and elements of the \emph{dual canonical basis} \cite{L90} (compare to \emph{crystal basis}
%independently \textcolor{red}{(I had a paper where I wrote ``independent'' but Lusztig sent an email to me and he said it is not independent. Maybe we need to remove the word ``independent'')} discovered (\textcolor{red}{maybe also remove ``discovered''?}) 
in \cite{Kashiwara_crystal}) sending $T$ to the basis element $\ch(T)$; 
this bijection is due to \cite{CDFL,Li20} based on the representation theory of quantum affine algebras \cite{CP91} and their relation to cluster algebras \cite{HL10}. 
An explicit expression of the dual canonical basis element denoted by $\ch(T)$ in terms of monomials in Plücker coordinates is given by the \emph{character formula}, see \eqref{eq:formula of chT for kU after simplification} and \eqref{eq:formula of ch(T) Grassmannian}. 
More precisely, the character of a tableau is an expression in \emph{standard monomials}, \emph{i.e.} monomials in Plücker coordinates whose index sets form the columns of a semistandard Young tableau \cite{LR_SMT}.
The character formula amounts to the change of basis from standard monomial basis to dual canonical basis.
In particular, cluster variables and more generally \emph{cluster monomials} (monomials in cluster variables that are all in one seed) are elements of the dual canonical basis \cite{FanQin17,KKOP_quantumaffine, KKOP_monoidal}.%\textcolor{red}{add citation Kashiwara-K.-Oh-P.}

The tableaux for the Grassmannian $\Gr_{d;n}$ are of rectangular shape with $d$ rows and entries in $[n]$, denoted $\SSYT_{d;n}$.
The tableaux for $\flag_{d_1,\dots,d_k;n}$ may have columns with $d_1,\dots,d_k$ many rows and entries in $[n]$, denoted $\SSYT_{d_1,\dots,d_k;n}$.
In Corollary~\ref{cor:initial tableau} we describe the tableaux of the initial cluster variables of $\flag_{d_1,\dots,d_k;n}$ explicitly, called the \emph{initial tableaux}.

The map \eqref{eq:embed} naturally induces a map 
\begin{eqnarray}\label{eq:map of tableaux}
    \phi:\SSYT_{d_1,\dots,d_k;n} \to \SSYT_{d_k;N}
\end{eqnarray}
by filling up all columns with $d_i<d_k$ many rows to columns with $d_k$ many rows by adding the entries $n+1,n+2,\dots,n+d_k-d_i$.
The images of the initial tableaux of $\flag_{d_1,\dots,d_k;n}$ are given in Corollary~\ref{cor:image of initial tableaux}.
The following result is a consequence of Theorem~\ref{thm:general flag} and we prove it in \S\ref{sec:initial tableaux}.

\begin{corollary}\label{cor:character compatinility initials}
   Let $1\le d_1<\cdots<d_k<n$ and consider the flag variety $\flag_{d_1,\ldots,d_k;n}$. Let $T$ be the tableaux of an initial cluster variable. Then
    \begin{eqnarray}\label{eq:compatibility of characters}
    \ch(\phi(T)))=\varphi^*(\ch(T)).
    \end{eqnarray}
\end{corollary}
% If Conjecture~\ref{conj:general flag} holds then it implies the more general statement:
% \begin{conjecture}\label{conj:character compatinility initials}
%     For an arbitrary partial flag variety $\flag_{d_1,\dots,d_k;n}$ and an arbitrary initial cluster variable $\ch(T)$ Equation~\eqref{eq:compatibility of characters} holds.
% \end{conjecture}

It would be insightful to see if this is true more generally for arbitrary dual canonical basis elements, that is:
\begin{question}\label{question}
    Given a tableaux $T$ for $\flag_{d_1,\dots,d_k;n}$ and the corresponding element of the dual canonical basis $\ch(T)$. Then $\phi(T)$ is a tableaux for $\Gr_{d_k;N}$ with dual canonical basis element $\ch(\phi(T))$. Is it true that Equation~\eqref{eq:compatibility of characters} holds in this generality?
\end{question}

%and arbitrary partial flag varieties.

\subsection{Structure} In \S\ref{sec:preliminary} we recall the cluster algebra structure on partial flag varieties due to Geiss, Leclerc and Schr\"{o}er \cite{GLS_partial-flags}. We extend their results by introducing the combinatorial gadget of \emph{pseudoline arrangements} in this context and slightly adapt their constructions to our needs. 
Further, in \S\ref{sec:tableaux} we recall the notion of tableaux that parametrize dual canonical basis elements \cite{CDFL}.
In this section we investigate the characters and tableaux of initial minors and how they behave under the embeddings into Grassmannians. The proof of Corollary~\ref{cor:character compatinility initials} is included in this section.
In \S\ref{sec:mutation sequence} we prove Theorem~\ref{thm:general flag} by exhibiting explicit mutation sequences in the corresponding Grassmannians. In \S\ref{sec:Cluster structures on MTn and SHn from Grassmannians}, we describe the mutation sequences for momentum twist varieties and spinor helicity varieties as an application of Theorem~\ref{thm:general flag}. 
  
\subsection*{Acknowledgements}
This project started at the Simons Center for Geometry and Physics workshop on \emph{Mathematical Aspects of N=4 Super-Yang-Mills Theory}, we are grateful to the organizers.
Moreover, we would like to thank James Drummond, \"{O}mer G\"{u}rdo\u{g}an, June Huh and Lauren Williams for their comments and fruitful discussions.
LB is grateful to Leonid Monin and Alfredo Nájera Chávez for discussions on the cluster structure of partial flag varieties.
JRL would like to thank the Galileo Galilei Institute for Theoretical Physics for the hospitality and the INFN for partial support during the completion of this work. 
LB is supported by the UNAM DGAPA PAPIIT project IA100724 and the CONAHCyT project CF-2023-G-106.
JRL is supported by the Austrian Science Fund (FWF): P-34602, Grant DOI: 10.55776/P34602, and PAT 9039323, Grant-DOI 10.55776/PAT9039323. 

\section{Cluster structure on partial flag varieties}\label{sec:preliminary}

Denote by $SL_n$ the special linear group of $n\times n$ matrices with entries in a field $\kk$. We fix the Borel subgroup of upper triangular matrices $B\subset SL_n$ and the subgroup of unipotent matrices $U\subset B$; analogously for the lower triangular case we denote $U^-\subset B^-\subset SL_n$. Additionally, let $T\subset B$ be the torus of diagonal matrices.
Given a parabolic subgroup $P^-\supset B^-$ of $SL_n$ the quotient $P^-\backslash SL_n$ is a partial flag variety $\flag_{d_1,\dots,d_k;n}$ for some $d_1,\dots,d_k$, so we write $P^-=P_{d_1,\dots,d_k;n}^-$.
We denote the unipotent radical of $P_{d_1,\dots,d_k;n}:=(P_{d_1,\dots,d_k;n}^-)^T$ by $U_{d_1,\dots,d_k;n}$. 
Then $U_{d_1,\dots,d_k;n}\subset U$ is an affine variety of the same dimension as $\flag_{d_1,\dots,d_k;n}$.
In fact we have an embedding $\iota:U_{d_1,\dots,d_k;n}\hookrightarrow \flag_{d_1,\dots,d_k;n}$ given by generating a flag from the rows of a matrix in $U_{d_1,\dots,d_k;n}$.
\begin{example}
    Consider $\flag_{2,4;5}$ and let $B\subset SL_5$ be upper triangular matrices. Then
    \[
    P^-_{2,4;5}=\left(\begin{matrix}
        *&*&0&0&0\\ *&*&0&0&0\\ *&*&*&*&0\\ *&*&*&*&0\\ *&*&*&*&*
        %*&*&*&*&* \\ *&*&*&*&* \\ 0&0&*&*&* \\ 0&0&*&*&* \\ 0&0&0&0&*
    \end{matrix}\right) \quad \text{and} \quad U_{2,4;5}=\left(\begin{matrix}
        1&0&*&*&* \\ 0&1&*&*&* \\ 0&0&1&0&* \\ 0&0&0&1&* \\ 0&0&0&0&1
    \end{matrix}\right).
    \]
    The embedding $U_{2,4;5}\hookrightarrow \flag_{2,4;5}$ sends a matrix $M\in U_{2,4;5}$ with rows $m_1,\dots,m_5$ to the flag
    \[
    \{0\}\subset \langle m_1,m_2\rangle \subset \langle m_1,\dots,m_4\rangle \subset \kk^5.
    \]
\end{example}

For the corresponding Lie algebras we use the notation: $\text{Lie}(SL_n)=\mathfrak{sl}_n$, $\text{Lie}(B)=\mathfrak b$ and $\text{Lie}(U)=\mathfrak n, \text{Lie}(T)=\mathfrak h$. The simple roots of $\mathfrak h$ are denoted by $\alpha_1,\dots,\alpha_{n-1}$ and the fundamental weights are $\omega_1,\dots,\omega_{n-1}$.

The coordinate ring $\kk[U]$ has a cluster algebra structure due to Berenstein, Fomin and Zelevisnky \cite{BFZ96, BFZ05}.
Their work was generalized to the case of $\kk[U_{d_1,\dots,d_k;n}]$ by Geiss, Leclerc and Schröer \cite{GLS_partial-flags}. 
%that is of main interest to us.
The cluster algebra structure on $\kk[U_{d_1,\dots,d_k;n}]$ lifts to a cluster algebra structure on $\kk[\flag_{d_1,\dots,d_k;n}]$ via the pullback of the embedding 
$U_{d_1,\dots,d_k;n}\hookrightarrow \flag_{d_1,\dots,d_k;n}$.
We recall the details in the following subsections.

\subsection{Plücker relations}
We start by recalling the Pl\"ucker embedding of a Grassmannian. Let $(e_i)_{1\le i\le n}$ denote the standard basis of $\kk^n$, so that 
\[
\{e_{i_1}\wedge\dots\wedge e_{i_d}\mid 1\leq i_1<i_2<\ldots <i_d\leq n\}
\]
is a basis of $\wedge^d\kk^n$. Let $(\wedge^d\kk^n)^*$ be the dual vector space, then the Pl\"ucker coordinate $P_{i_1,\dots,i_d}\in(\wedge^d\kk^n)^*$ for $1\leq i_1<i_2<\ldots <i_d\leq n$ is defined to be the basis element dual to $e_{i_1}\wedge\ldots \wedge e_{i_d}$.
For  $i_1, \ldots ,i_d\in [n]$ pairwise distinct, but not necessarily increasing, the Pl\"ucker coordinate $P_{i_1, \ldots, i_d}$ has the following property
\[
P_{\sigma(i_1),\dots,\sigma(i_d)}=(-1)^{\ell(\sigma)}P_{i_1,\dots,i_k} \quad  \hbox{ for all }\sigma \in S_n,
\]
where $S_n$ denotes the symmetric group. 
Let $\mathcal{S}(n,d)$ be the set of tuples $I=(i_1,\dots,i_d)$ with entries in $n$ and 
denote by $P_I=P_{i_1,\dots,i_d}$ the Pl\"ucker coordinate. Notice that $P_I=0$ if not all entries in $I$ are distinct.
%corresponding to a sequence $I=(i_1,\dots,i_k)\in\mathcal{S}(n,k)$. 
To simplify notation we index some Pl\"ucker coordinates by a set instead of a sequence which is interpreted as being indexed by the sequence obtained by arranging the elements of the set in an increasing order.
In this way we obtain the \emph{Pl\"ucker embedding}
\begin{equation}\label{eq:plückerembed}
\Gr_{d;n}\hookrightarrow \mathbb P(\wedge^d\mathbb \kk^n)\cong \mathbb P^{\binom{n}{d}-1}.
\end{equation}
The flag variety is embedded in the product of Grassmannians
\begin{eqnarray*}
\flag_{d_1,\dots,d_k;n} \hookrightarrow \Gr_{d_1;n}\times \Gr_{d_2;n}\times\dots\times\Gr_{d_k;n}.
\end{eqnarray*}
By composing the latter embedding with the embedding \eqref{eq:plückerembed} for each Grassmannian in the product, we get the \emph{Plücker embedding}
\begin{eqnarray}\label{eq:Plücker embed}
\flag_{d_1,\dots,d_k;n} \hookrightarrow \mathbb P^{\binom{n}{d_1}-1}\times \dots \mathbb P^{\binom{n}{d_k}-1}.
\end{eqnarray}
Denote by ${I}_{d_1,\dots,d_k;n}$ the (homogeneous) ideal of $\flag_{d_1,\dots,d_k;n}$ in $\kk[P_{I}: I\in\mathcal S(n,d_i),i\in[k]]$ with respect to this embedding. 
The ideal ${I}_{d_1,\dots,d_k;n}$ is generated by \emph{Plücker relations}, which are defined as follows: for each choice of $p\le q$ with $p,q\in \{d_1,\dots,d_k\}$, a number $s\in[p]$ and a pair of tuples $J\in \mathcal S(n,d_p),L\in\mathcal S(n,d_q)$
we define
\begin{equation}\label{eq:def Plücker rel}
R^s_{J,L}=P_JP_L-\sum_{1\le r_1<\dots<r_s\le d_q} P_{J'}P_{L'},
\end{equation}
where 
$L'=(L\setminus (l_{r_1},\dots,l_{r_s}))\cup (j_1,\dots, j_s)$ and
$J'= (J\setminus ( j_1,\dots,j_s))\cup(l_{r_1},\dots,l_{r_s})$ \cite[p.135, Equation (3)]{fulton_young}, \cite[Equation (1.1)]{Fe12}. 
Notice that this is a restricted set of Plücker relations in the sense that the exchange taking place among the index sets only involves the elements $j_1,\dots,j_s\in J$.

Before proceeding to verify that the map \eqref{eq:embed} is well defined, 
notice that $\kk[\flag_{d_1,\dots,d_k;n}]$ is $\mathbb Z_{\ge 0}^k$-multigraded. Denote by $\epsilon_1,\dots,\epsilon_k$ the standard basis of $\mathbb Z^k$. 
We have $\deg(P_{I})=\epsilon_{|I|}$.
Recall that $\varphi^*$ defined in \eqref{eq:embed} is induced by a map on the polynomial rings (which we also denote as $\varphi^*$ by abuse of notation)
\[
\varphi^*: \kk[P_I:I\in\mathcal S(n,d_i),i\in[k]]\to \kk[P_I:I\in\mathcal S(N,d_k)]
\]
given by $P_I\mapsto P_{I\cup[n+1,n+d_k-d_i]}$ for $I\in\mathcal{S}(n,d_i), i\in [k]$.

\begin{proposition}\label{proposition:embed}
We have the identity 
\begin{eqnarray}\label{eq:Plücker ideals}
I_{d_k;N}\cap \kk[P_{I\cup[n+1,n+d_k-d_i]}:I\in\mathcal{S}(n;d_i),i\in[k]] = \varphi^*(I_{d_1,\dots,d_k;n}),   
\end{eqnarray} 
and the map $\varphi^*:\kk[\flag_{d_1,\dots,d_k;n}]\to \kk[\Gr_{d_k;N}]$ is a homomphism of graded algebras and an embedding.
\end{proposition}

\begin{proof}
By definition, the images of Plücker coordinates in the flag are Plücker coordinates in the Grassmannian, in consequence the map is graded.
The algebra $\kk[\flag_{d_1,\dots,d_k;n}]$ is a quotient of the polynomial ring $\kk[P_I:I\in\mathcal S(n,d_i),i\in[k]]$ by the ideal $I_{d_1,\dots,d_k;n}$.
Similarly, $\kk[\Gr_{d_k;N}]$ is the quotient of the polynomial ring $\kk[P_I:I\in\mathcal S(N,d_k)]$ by the ideal $I_{d_k;N}$.
The map $\varphi^*$ on the polynomial rings is a well defined embedding.
We first show that the image of $I_{d_1,\dots,d_k;n}$ is contained in $I_{d_k;N}$ by proving it for each of the generators: let $p\le q$, $p,q\in \{d_1,\dots,d_k\}$,  $s\in[p]$ and $J\in \mathcal S(n,d_p),L\in\mathcal S(n,d_q)$ and compute
\[
\varphi^*(R^s_{J,L})=P_{J\cup[n+1,n+d_k-d_p]}P_{L\cup[n+1,n+d_k-d_q]}- \sum_{1\le r_1<\dots<r_s\le d_q} P_{J'\cup[n+1,n+d_k-d_p]}P_{L'\cup [n+1,n+d_k-d_q]} 
\]
We claim that the image coincides with the generator for $s,J\cup[n+1,n+d_k-d_p],L\cup[n+1,n+d_k-d_q]$ in $I_{d_k;N}$:
\begin{align}\label{eq:special Plücker}
\begin{split}
R^s_{J\cup[n+1,n+d_k-d_p],L\cup[n+1,n+d_k-d_q]} = \ & P_{J\cup[n+1,n+d_k-d_p]}P_{L\cup[n+1,n+d_k-d_q]}  \\ & - \sum_{1\le r_1<\dots<r_s\le d_k} P_{(J\cup[n+1,n+d_k-d_p])'}P_{(L\cup [n+1,n+d_k-d_q])'}.
\end{split}
\end{align}
The sum of the right hand side splits into two parts: either $1\le r_1<\dots<r_s\le d_q$ in which case 
\[
(J\cup[n+1,n+d_k-d_p])'=J'\cup[n+1,n+d_k-d_p]\text{ and }(L\cup [n+1,n+d_k-d_q])'=L'\cup [n+1,n+d_k-d_q];
\]
or there exists $r_j>d_q$ so that $\ell_{r_j}\in [n+1,n+d_k-d_q]$, hence $\ell_{r_j}\in (J\setminus (j_1,\dots,j_s))\cup [n+1,n+d_k-d_p]$ and therefore $P_{(J\cup[n+1,n+d_k-d_p])'}=0$. We deduce
\[
\varphi^*(R^s_{J,L})= R^s_{J\cup[n+1,n+d_k-d_p],L\cup[n+1,n+d_k-d_q]},
\]
which implies $\varphi^*(I_{d_1,\dots,d_k;n})\subset I_{d_k;N}$. 
To see that $\varphi^*$ is an embedding consider $\bar f\in\kk[\flag_{d_1,\dots,d_k;n}]$ with $\varphi^*(\bar f)=0$. Then $\bar f$ is represented by a polynomial $f\in \kk[P_I:I\in\mathcal S(n,d_i),i\in[k]]$ with $\varphi^*(f)\in I_{d_k;N}$, so 
\[
\varphi^*(f)\in I_{d_k;N}\cap \kk[P_{I\cup[n+1,n+d_k-d_i]}:I\in\mathcal{S}(n;d_i),i\in[k]].
\]
The RHS is generated by elements of form \eqref{eq:special Plücker} which coincides with the image of a generator of $I_{d_1,\dots,d_k;n}$. So we have the identity (\ref{eq:Plücker ideals}) and the second claim follows.
\end{proof}

Hence, geometrically we have a map between the affine cones
\[
\widetilde{\Gr}_{d_k;N} \to \widetilde{\flag}_{d_1,\dots,d_k;n}.
\]
Note that the map does not induce a map between the projective varieties due to the different gradings.

%\subsection{Cluster structure on partial flag varieties}
\subsection{Pseudoline arrangements}
For $\flag_{d_1,\dots,d_k;n}$ the \emph{pseudoline arrangement} (sometimes also known as \emph{wiring diagram}) $\mathcal{P}_{d_1,\dots,d_k;n}$ is a pictorial presentation of the permutation $\sigma\in S_n$ whose one-line presentation is 
\[
\sigma=[d_{k}+1,d_{k}+2,\dots,n,d_{k-1}+1,d_{k-1}+2\dots, d_{k},\dots,1,\dots,d_1].
\]
Notice that $\sigma$ is the permutation corresponding to the minimal representative of the coset of the longest word in $S_n/(s_i:i\not\in\{d_1,\dots,d_k\})$.

\medskip
\noindent
{\bf Algorithm 1: Pseudoline arrangement}

We draw $\mathcal P_{d_1,\dots,d_k;n}$ inside a two-dimeinsional positive orthant.
\begin{enumerate}
\item Label the $x$- and $y$-axes by $1,\dots,n$
%and the $y$-axis by $\sigma(1)=d_k+1,\dots,\sigma(n)=d_1$.
%Draw nodes labelled by $1,\dots,n$ (left to right) on a horizontal line and $n$ nodes labelled $\sigma(1),\dots,\sigma(n)$ on a vertical line as in a positive orthant.
\item For each $i$ draw a line segment from $(i,0)$ to $(i,\sigma(i))$ and another line segment from $(0,\sigma(i))$ to $(i,\sigma(i))$. 
The union of the two line segments is called the \emph{pseudoline} (or \emph{wire}) $\ell_i$.
%$\ell_i$ with starting po connecting $i$ with $i=\sigma(j)$ (for some $j$) in the following way: the lines $\ell_{d_i+1},\ell_{d_i+2}\dots,\ell_{d_{i+1}}$ do not cross; they start vertically and then turn once they reach the height of the vertical nodes $d_i+1,\dots,d_{i+1}$. 
\end{enumerate}

An example of the resulting pseudoline arrangement $\mathcal P_{2,5;7}$ is given in Figure~\ref{fig:P and Q257}.
%Notice that every square of $Gamma_{d_1,\dots,d_k;n}$ gives rise to a bounded face of $\mathcal P_{d_1,\dots,d_k;n}$. We frequently make use of this observation.

\subsection{Quivers} We associate a quiver $Q_{d_1,\dots,d_k;n}$ to $\mathcal P_{d_1,\dots,d_k;n}$ that determines a seed in the cluster structure of the multi-homogeneous coordinate ring $\kk[\flag_{d_1,\dots,d_k;n}]$ with respect to the Pl\"ucker embedding, compare to \cite[\S9.3.2]{GLS_partial-flags}.

\medskip
\noindent
{\bf Algorithm 2: From pseudoline arrangement to quiver}
\begin{enumerate}
\item {\bf Vertices of $Q_{d_1,\dots,d_k;n}$:}
    \begin{enumerate}
    \item mutable vertices of  $Q_{d_1,\dots,d_k;n}$ correspond to bounded faces of $\mathcal P_{d_1,\dots,d_k;n}$; %(in particular, every square in $Gamma_{d_1,\dots,d_k;n}$ corresponds to a mutable vertex);
    \item there are two types of frozen vertices: $n-1$ of them correspond to the unbounded faces along the $y$-axis; additionally there are $k$ frozen vertices, we denote them by $v_{d_1},\dots,v_{d_{k}}$.      
    \end{enumerate}
\item {\bf Arrows of $Q_{d_1,\dots,d_k;n}$:}
There are four types of arrows:
\begin{enumerate}
    \item from left to right perpendicular to a vertical straight lines segment connecting adjacent faces of $\mathcal{P}_{d_1,\dots,d_k;n}$;
    \item from top to bottom perpendicular to a horizontal straight line segment connecting adjacent faces of $\mathcal{P}_{d_1,\dots,d_k;n}$;
    \item diagonally from bottom right to top left through a crossing of two straight line segments connecting faces of $\mathcal{P}_{d_1,\dots,d_k;n}$ that share a vertex;
    \item arrows to and from the extra frozen vertices $v_{d_1},\dots,v_{d_k}$: there is an arrow from the face bounded by $\ell_{d_i-1},\ell_{d_i}$ vertically and by $\ell_{d_i+1},\ell_{d_i+2}$ horizontally to the vertex $v_{d_i}$, and an arrow from $v_{d_i}$ to the face bounded by $\ell_{d_i}$ on the left, by $\ell_{d_i+1}$ on the top and right (this is where $\ell_{d_i+1}$ bends) and by $\ell_{d_{i+2}}$ on the bottom. 
\end{enumerate}
\end{enumerate}

The quiver $Q_{2,5;7}$ is depicted in {Figure~\ref{fig:P and Q257}}. The frozen vertices $v_{2},v_5$ are labelled $\omega_2,\omega_5$, respectively.
 \begin{figure}
     \centering
 \adjustbox{scale=0.55}{\begin{tikzcd}
 	7=\sigma(2) &&&& {} \\
 	\\
 	6=\sigma(1) && {} \\
 	& \boxed{\{1,2\}} && \boxed{\{2\}} && \boxed{\omega_2} \\
 	5=\sigma(5) &&&&&&&&&& {} \\
 	& \boxed{\{1,2,5\}} && {\{2,5\}} &&&&&& {\{5\}} \\
 	4=\sigma(4) &&&&&&&& {} \\
 	& \boxed{\{1,2,4,5\}} && {\{2,4,5\}} &&&& {\{4,5\}} \\
 	3=\sigma(3) &&&&&& {} \\
 	& \boxed{\{1,2,3,4,5\}} && {\{2,3,4,5\}} && {\{3,4,5\}} \\
 	2=\sigma(7) &&&&&&&&&&&&&& {} \\
 	& \boxed{\{1,2,3,4,5,7\}} && {\{2,3,4,5,7\}} && {\{3,4,5,7\}} && {\{4,5,7\}} && {\{5,7\}} && \boxed{\omega_5} \\
 	1=\sigma(6) &&&&&&&&&&&& {} \\
 	\\
 	&& 1 && \boxed{2} && 3 && 4 && \boxed{5} && 6 && 7 \\
  \arrow[teal, no head, from=1-1, to=15-5, rounded corners, to path=-| (\tikztotarget)]
   %\arrow[teal, no head, from=1-1, to=1-5]
 	%\arrow[teal,no head, from=1-5, to=15-5]
 \arrow[teal,no head, from=3-1, to=15-3,rounded corners, to path=-| (\tikztotarget)]
    %\arrow[teal,no head, from=3-1, to=3-3]
 	%\arrow[teal,no head, from=3-3, to=15-3]
 \arrow[teal,no head, from=5-1, to=15-11,rounded corners, to path=-| (\tikztotarget)]
    %\arrow[teal,no head, from=5-1, to=5-11]
 	%\arrow[teal,no head, from=5-11, to=15-11]
 \arrow[teal,no head, from=7-1, to=15-9,rounded corners, to path=-| (\tikztotarget)]
    %\arrow[teal,no head, from=7-1, to=7-9]
 	%\arrow[teal,no head, from=7-9, to=15-9]
 \arrow[teal,no head, from=9-1, to=15-7,rounded corners, to path=-| (\tikztotarget)]
    %\arrow[teal,no head, from=9-1, to=9-7]
 	%\arrow[teal,no head, from=9-7, to=15-7]
  \arrow[teal,no head, from=11-1, to=15-15,rounded corners, to path=-| (\tikztotarget)]
 	%\arrow[teal,no head, from=11-1, to=11-15]
 	%\arrow[teal,no head, from=11-15, to=15-15]
  \arrow[teal,no head, from=13-1, to=15-13,rounded corners, to path=-| (\tikztotarget)]
 	%\arrow[teal,no head, from=13-1, to=13-13]
 	%\arrow[teal,no head, from=13-13, to=15-13]
 	\arrow[from=12-6, to=10-4]
 	\arrow[from=12-2, to=12-4]
 	\arrow[from=12-4, to=12-6]
 	\arrow[from=12-6, to=12-8]
 	\arrow[from=12-8, to=12-10]
 	\arrow[from=10-4, to=12-4]
 	\arrow[from=8-4, to=10-4]
 	\arrow[from=6-4, to=8-4]
 	\arrow[from=4-4, to=6-4]
 	\arrow[from=6-2, to=6-4]
 	\arrow[from=6-4, to=4-2]
 	\arrow[from=6-4, to=6-10]
 	\arrow[from=8-2, to=8-4]
 	\arrow[from=10-4, to=8-2]
 	\arrow[from=10-2, to=10-4]
 	\arrow[from=10-4, to=10-6]
 	\arrow[from=10-6, to=8-4]
 	\arrow[from=8-4, to=6-2]
 	\arrow[from=12-8, to=10-6]
 	\arrow[from=8-4, to=8-8]
 	\arrow[from=8-8, to=12-8]
 	\arrow[from=8-8, to=6-4]
 	\arrow[from=10-6, to=12-6]
 	\arrow[from=6-10, to=4-4]
 	\arrow[from=6-10, to=12-10]
 	\arrow[from=12-10, to=8-8]
  \arrow[from=12-10, to=12-12]
  \arrow[from=8-4, to=4-6]
  \arrow[from=4-6, to=10-6]
 \end{tikzcd}}
 \caption{The pseudoline arrangement $\mathcal P_{2,5;7}$ and its quiver $Q_{2,5;7}$ together with the index sets $I_F$ of the initial minors and the additional frozen vertices corresponding to $d_1=2,d_2=5$ labelled by their weights (as is \S\ref{sec:U}), $\omega_2$ respectively $\omega_5$. }
     \label{fig:P and Q257}
 \end{figure}
Every face of the pseudoline arrangement $\mathcal P_{d_1,\dots,d_k;n}$ can be associated with a minor in $\kk[U]$ \cite{BFZ96}.
The minors are of form form $D_{I,J}$ with $I,J\subset [n]$ of the same size. 
In our case, the column index set $J$ is always of form $\{n-|I|-1,\dots,n\}$.
We associate index sets $I_F$ to faces $F$ of $\mathcal P_{d_1,\dots,d_k;n}$: 
%$I$ is given by the indices of lines passing north east of the face $F$.
%\footnote{Usually one would consider all lines passing below the given face. Our pseudoline arrangements are 135° tilted to the left, so the usual convention of lines passing "below" translated to lines passing "north east" of the given face.}.
%That is, given a face $F$ of $\mathcal P_{d_1,\dots,d_k;n}$ we set
\begin{eqnarray}\label{def:IF}
    I_F:=\{i:\ell_i \text{ passes north-east of }F\}\quad \text{and} \quad D_{I_F}:=D_{I_F,\{n-|I_F|+1,\dots,n\}}. 
\end{eqnarray}

\medskip
\noindent
\subsection{\texorpdfstring{Actions on $\kk[U]$}{Actions on kk[U]}} \label{sec:U} \cite[\S2.1]{GLS_partial-flags}
Recall that there is a left and right action of $U$ on $\kk[U]$ given by
\[
(x\cdot f)(n)=f(nx) \quad \text{and} \quad (f\cdot x)(n)=f(xn), \quad \text{for all } f\in \kk[U],x,n\in U.
\]
Differentiating these actions we obtain a left and right action of $\mathfrak n$ on $\kk[U]$.
Recall that $\mathfrak n$ has a basis given by elementary matrices $E_{i,j}$ with $i<j$. Define the \emph{Chevalley generators} $e_i:=E_{i,i+1}$ for $1\le i\le n-1$.
Set $e_i^\dagger(f):=f\cdot e_i$.
Given a minor $D_{I,\{n-d-1,\dots n\}}=:D_I$ it is not hard to see that
\begin{eqnarray}\label{eq:action of n}
\begin{matrix}
e_i^\dagger (D_I)= 0 & \text{ if }& i\not\in I, \text{ or } i,i+1 \in I,\\
e_i^\dagger (D_I)\not=0 & \text{ if } &i\in I\not\ni i+1.
\end{matrix}
\end{eqnarray}
Additionally we have $(e_i^\dagger)^2(D_I)=0$ for all $i$ and $I$.
The coordinate ring $\kk[U_{d_1,\dots,d_k;n}]$ can be identified with a subalgebra of $\kk[U]$:
\begin{equation}\label{eq:kUd in kU}
    \kk[U_{d_1,\dots,d_k;n}]\cong \{f\in \kk[U]: e_j^\dagger(f)=0 \forall j\not \in \{d_1,\dots,d_k\}\}.
\end{equation}
Hence, we also have an $\mathfrak n$-action on $\kk[U_{d_1,\dots,d_k;n}]$.
%Recall from \cite[\S2.2]{GLS_partial-flags} that $U_{d_1,\dots,d_k;n}$ is a dense open subset of $\flag_{d_1,\dots,d_k;n}$. 
%Hence, there is a projection between the corresponding rings.

The multi-homogeneous coordinate ring $\kk[\flag_{d_1,\dots,d_k;n}]$ is multigraded by a submonoid of the weight lattice of the Cartan subalgebra of diagonal matrices $\mathfrak{h}\subset \mathfrak{sl}_n$. 
More precisely, let $\omega_1,\dots,\omega_{n-1}$ denote the fundamental weights. 
Then $\kk[\flag_{d_1,\dots,d_k;n}]$ is graded by $\bigoplus_{i\in \{d_1,\dots,d_k\}}\mathbb Z_{\ge 0}\omega_i$.
In \cite[Proof of Lemma 2.4]{GLS_partial-flags} the authors construct a lift $\hat f\in \kk[\flag_{d_1,\dots,d_k;n}]$ for every $f\in \kk[U_{d_1,\dots,d_k;n}]$ uniquely determined by
\begin{enumerate}
    \item $\iota^*(\hat f)=f$, where $\iota:U_{d_1,\dots,d_k;n}\hookrightarrow \flag_{d_1,\dots,d_k;n}$;
    \item $\hat f$ is homogeneous;
    \item $\hat f$ is of minimal degree with respect to (1) and (2).
\end{enumerate}
%\textcolor{red}{(What is the relation between $\flag_{d_1,\dots,d_k;n}$ and $U_{d_1,\dots,d_k;n}$)}
As the degree of an element $\kk[\flag_{d_1,\dots,d_k;n}]$ is a weight of $\mathfrak{h}$ we use the term weight to refer to the degree.
The weight of $\hat f$ is 
\[
\lambda(\hat f)=\lambda(f):=\sum_{i\in \{d_1,\dots,d_k\}}a_i(f)\omega_i,  \quad \text{where}\quad  a_i(f):=\max\{s:(e_i^\dagger)^s(f)\not =0\}.
\]
In particular, as $(e_i^\dagger)^2(D_I)=0$ for all $i$ and $I$ we have 
\begin{eqnarray}\label{eq:wt DI}
\lambda(D_I)=\sum_{i\in \{d_1,\dots,d_k\}, i\in I\not\ni i+1} \omega_i.    
\end{eqnarray}

Therefore weights of minors ${D}_{I_F}$ corresponding to faces $F\in \mathcal P_{d_1,\dots,d_k;n}$ are encoded in the pseudoline arrangement as follows:

\begin{lemma}\label{lem:weight}
Let $F$ be a face of $\mathcal P_{d_1,\dots,d_k;n}$ and $D_{I_F}$ the corresponding minor. 
Then $a_i(D_{I_F})=1$ if and only if the line $\ell_i$ passes north-east of $F$ and the line $\ell_{i+1}$ passes south-west of $F$.
In particular, $a_i(D_{I_F})=1$ if and only if $F$ lies in between the horizontal segments of the pseudolines $\ell_{d_{i}}$ and $\ell_{d_{i+1}}$. 
%of $\mathcal P_{d_1,\dots,d_k;n}$ (recall that $d_{k+1}:=n$ and $d_0:=d_1-1$).}
\end{lemma}
\begin{proof}
Both claims follow directly from \eqref{eq:wt DI}.
In particular, the faces between between the horizontal segments of the pseudolines $\ell_{d_{i}}$ and $\ell_{d_{i+1}}$ are the only ones satisfying $d_i\in I_F$ and $d_{i}+1\not\in I_F$.
\end{proof}

Notice that the pseudolines $\ell_{i}$ and $\ell_{i+1}$ cross if and only if $i\in \{d_1,\dots,d_k\}$.
\begin{comment}
Based on Lemma~\ref{lem:weight} we define the \emph{region of weight $\omega_i$} inside $\mathcal P_{d_1,\dots,d_k;n}$ by
\begin{eqnarray}\label{eq:R omega_i}
\mathcal R_{\omega_i}:=\bigcup_{F:a_i(D_{I_F})=1} F.
\end{eqnarray}    
\end{comment}

\subsection{Cluster variables} We extend $Q_{d_1,\dots,d_k;n}$ to a seed for $\kk[\flag_{d_1,\dots,d_k;n}]$ by associating cluster variables with its vertices. 
As $\kk[\flag_{d_1,\dots,d_k;n}]$ is graded the main challenge is to verify that the mutation relations are homogeneous.
For a vertex $v_F$ corresponding to a face $F$ of $\mathcal P_{d_1,\dots,d_k;n}$ the cluster variable is the lift $\widehat{D}_{I_F}$ of the corresponding minor $D_{I_F}\in \kk[U]$, which we describe explicitly in terms of Plücker coordinates subsequently in \S\ref{sec:initial tableaux}.
To a frozen vertex $v_{d_i}$ we associate the Pl\"ucker coordinate $P_{[d_i]}$, which is of weight $\omega_{d_i}$ in $\kk[\flag_{d_1,\dots,d_k;n}]$.
Denote the resulting cluster by ${\bf x}_{d_1,\dots,d_k;n}$.

\begin{remark}
For partial flag varieties without \emph{gaps} in dimension, that is $\{d_1,\dots,d_k\}=[d_1,d_k]$, all lifts $\widehat{D}_{I_F}$ are Pl\"ucker coordinates. In particular, the case of the full flag variety and the Grassmannain are included in this class.
This is not the case for a general partial flag variety with \emph{gaps}, the lifts may be homogeneous binomials in Pl\"ucker coordinates of degree two, see Proposition~\ref{prop:expansion of minors}.
It is worth noting that the initial seed for the Grassmannian obtained from its pseudoline arrangement coincides with the ``standard" or ``rectangle" initial seed.
\end{remark}

\begin{definition}\label{def:balanced}
A vertex $v$ of $Q_{d_1,\dots,d_k;n}$ is called \emph{balanced} if
\[
\sum_{v_F\to v} \lambda(\widehat{D}_F) = \sum_{v\to v_{F'}} \lambda(\widehat{D}_{F'}).
\]
The quiver $Q_{d_1,\dots,d_k;n}$ is called \emph{balanced} if all of its mutable vertices are.
\end{definition}

The following corollary is a consequence of Lemma~\ref{lem:weight}.

\begin{figure}
    \centering
\includegraphics[width=0.5\textwidth]{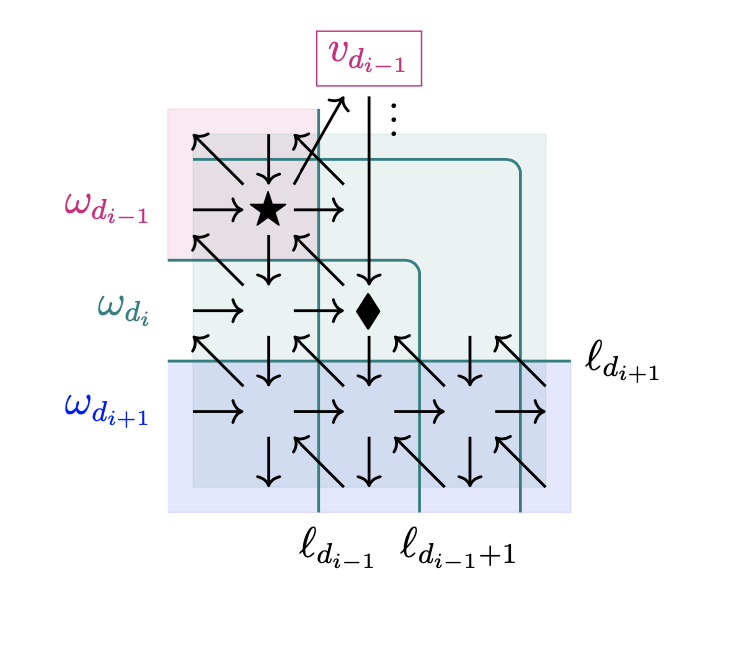} 
    \caption{The shading of the regions indicates the weight of the associated minor. Notice that the vertices $\bigstar$ and $\blacklozenge$ are non-balanced while all others are balanced. 
    }
    \label{fig:weights}
\end{figure}

\begin{corollary}\label{cor:generically balanced}
Let $v\in Q_{d_1,\dots,d_k;n}$ be a mutable vertex.
Then $v$ is balanced.
\end{corollary}

\begin{proof}
Consider Figure~\ref{fig:weights}. The top left shaded regions are of weight $\omega_{d_{i-1}}$, the bottom shaded region is of weight $\omega_{d_{i+1}}$ and all visible regions are of weight $\omega_{d_{i}}$.
Form the picture one can see that all regions except the $\bigstar$ and $\blacklozenge$ yield balanced vertices in $Q_{d_1,\dots,d_k;n}$.
For the vertex $\blacklozenge$ the weights of incoming arrows from faces of $\mathcal P_{d_1,\dots,d_k;n}$ sum to $2\omega_{d_i}+\omega_{d_{i+1}}$ and outgoing such arrows sum to $2\omega_{d_i}+\omega_{d_{i-1}}+\omega_{d_{i+1}}$; this imbalance is fixed by the incoming arrow to $v_{d_{i-1}}$ which is of weight $\omega_{d_{i-1}}$.
For the vertex $\bigstar$ the incoming arrows from faces of $\mathcal P_{d_1,\dots,d_k;n}$ contribute $3\omega_{d_i}+2\omega_{d_{i-1}}$ and outgoing such arrows $3\omega_{d_i}+\omega_{d_{i-1}}$;
this imbalance is fixed by the outgoing arrow to $v_{d_{i-1}}$.
\end{proof}

\begin{proposition}
The seed $(Q_{d_1,\dots,d_k;n},{\bf x}_{d_1,\dots,d_k;n})$ is an initial seed defining Geiss--Leclerc--Schr\"oer's cluster structure on $\kk[\flag_{d_1,\dots,d_k;n}]$. \end{proposition}

\begin{proof}
By construction the pseudoline arrangement $\mathcal P_{d_1,\dots,d_k;n}$ corresponds to a reduced expression of $\sigma\in S_n$.
%with \textcolor{red}{$ww_0^K=w_0$}.
Let $Q'$ be the full subquiver of $Q_{d_1,\dots,d_k;n}$ on all vertices except $v_{d_1},\dots,v_{d_k}$ and let ${\bf x}'$ be the set of minors $D_{I_F}$ corresponding to vertices in $Q'$.
Then by \cite[Proposition 9.4 and \S10]{GLS_partial-flags} $(Q',{\bf x}')$ is an initial seed for $\kk[U_{d_1,\dots,d_k;n}]$.
An initial seed for $\kk[\flag_{d_1,\dots,d_k;n}]$ can be obtained from $(Q',{\bf x}')$ as follows:
\begin{enumerate}
    \item extend $Q'$ by add $k$ frozen vertices and arrows following \cite[Theorem 10.2(iii)]{GLS_partial-flags};
    \item lift all elements of ${\bf x}'$ and add the $k$ Pl\"ucker coordinates $p_{[d_1]},\dots,p_{[d_k]}$ for the new frozen vertices.
\end{enumerate}
By definition the cluster ${\bf x}_{d_1,\dots,d_k;n}$ has the desired form and we only need to verify that the arrows adjacent to $v_{d_1},\dots,v_{d_k}$ agree with the ones in the reference. 
The number of arrows from $v_{d_j}$ to a vertex $v_F$ (eventually negative if there are only arrows from $v_F$ to $v_{d_j}$) is obtained by the unique lift of $D_{I_F}\mu_F(D_{I_F})=M_F+N_F$.
According to \cite[Lemma 2.5, Equations (10.1)\&(10.2)]{GLS_partial-flags}
\[
\widehat{D_{I_F}\mu_F(D_{I_F})}=\widehat{D}_{I_F} \widehat{\mu_F(D_{I_F})} = m_F\widehat{M}_F + n_F\widehat{ N}_F,
\]
where $m_F,n_F$ are corpime monomials in $P_{[d_1]},\dots,P_{[d_k]}$. By definition $m_F\widehat{M}_F + n_F\widehat{N}_F$ is homogeneous and of weight $\lambda(D_{I_F})+\lambda(\mu_F(D_{I_F}))$.
This uniquely determines $m_F$ and $n_F$, which in turn define the missing arrows as follows: 
\begin{enumerate}
    \item identify a non-balanced mutable vertex $v$ of $Q'$;
    \item compute the weights of all cluster variables corresponding to vertices adjacent to (and including) $v$;
    \item for $\sum_{v_F\to v} \lambda(D_F)=\sum_{i=1}^ka_i\omega_{d_i}$ and 
    $ \sum_{i=1}^k b_i\omega_{d_i}=\sum_{v\to v_{F'}} \lambda(D_{F'})$ add $\max\{0,b_i-a_i\}$-many arrows $v\to v_{d_i}$ and $\max\{0,a_i-b_i\}$-many arrows $v_{d_i}\to v$.
\end{enumerate}
Then $m_F=\prod_{i=1}^k P_{[d_i]}^{\max\{0,a_i-b_i\}}$ and  $n_F=\prod_{i=1}^k P_{[d_i]}^{\max\{0,b_i-a_i\}}$, see \cite[Equation (10.2)]{GLS_partial-flags}.
By Corollary~\ref{cor:generically balanced} we only have to consider vertices of type $\blacklozenge$ and $\bigstar$ in Figure~\ref{fig:weights}.
In the proof of Corollary~\ref{cor:generically balanced} we have computed the weights of minors for adjacent vertices.
For $v=\blacklozenge$ in Step 3 we add one incoming arrow $v_{d_{i-1}}\to \blacklozenge$; for $v=\bigstar$ we add one outgoing arrow $\bigstar\to v_{d_{i-1}}$.
The resulting quiver coincides with $Q_{d_1,\dots,d_k;n}$.
\end{proof}

\section{Cluster variables as tableaux}\label{sec:tableaux}
In this section we describe the initial cluster variables in more detail. We start by recalling some combinatorial notions. A \emph{Young diagram} is an arrangement of boxes that are north-west bound, that is the number of boxes in each row, respectively column, is weakly decreasing from top to bottom, respectively left to right.
A \emph{semistandard Young tableau} is a filling of a Young diagram with positive integers that are weakly increasing in rows and strictly increasing in columns. 
As we are mainly interested in Young tableaux (as apposed to Young diagrams) we refrain from drawing the boxes, but rather just depict the filling, see Example~\ref{exp:tableaux union}. 

\subsection{Tableaux}
For $k, m \in \ZZ_{\ge 1}$, denote by ${\rm SSYT}_{\le k; m}$ the set of all semistandard Young tableaux (including the empty tableau denoted by $\mathds{1}$) with less or equal to $k$ rows and with entries in $[m]$. 
To simplify language we frequently refer to semistandard Young tableaux simply as \emph{tableaux}.
For a tableau $T \in {\rm SSYT}_{\le k; m}$ with $k'\le k$ many rows,
we set the $i$th row to be empty for $i>k'$. 
%say the $i$th row of $T$ and $i'>k'$, then the $i$th row is empty. 

For $T,T' \in {\rm SSYT}_{\le k; m}$, we denote by $T \cup T'$ the row-increasing tableau whose $i$th row is the union of the $i$th rows of $T$ and $T'$ as multisets. 
The resulting tableau is semistandard \cite[Lemma 3.2]{Li20}.
%By Lemma 3.2 in \cite{Li20}, the tableau $T \cup T'$ is still semistandard. 
Therefore ${\rm SSYT}_{\le k; m}$ has the structure of a monoid with multiplication given by ``$\cup$''. 

\begin{example}\label{exp:tableaux union}
In ${\rm SSYT}_{\le 5;6}$, we have that 
\begin{align*}
\begin{matrix}
1 \\ 4 \\ 6
\end{matrix} \cup \begin{matrix} 2 \\ 3\end{matrix} = \begin{matrix} 1 & 2 \\ 3 & 4 \\ 6\end{matrix}.
\end{align*} 
\end{example}

For $S, T \in {\rm SSYT}_{\le k; m}$, we say that $S$ is a \emph{factor} of $T$, denoted $S \subset T$, if for every $i \in [k]$, the $i$th row of $S$ is contained in the $i$th row of $T$ as multisets. 
For a factor $S$ of $T$, we define $\frac{T}{S}=S^{-1}T=TS^{-1}$ as the row-increasing tableau whose elements in the $i$th row are the elements in the multiset-difference of $i$th row of $T$ and the $i$th row of $S$ for every $i \in [k]$.

We call a tableau $T \in {\rm SSYT}_{\le k;m}$ \emph{trivial} if it is a one-column tableau with entries $\{1,\ldots, p\}$ for some $p \in [k]$. 
For any $T \in {\rm SSYT}_{\le k;m}$, we denote by $T_{\text{red}} \subset T$ the \emph{reduced} semistandard tableau obtained by from $T$ by removing a maximal trivial factor. 

For $S, T \in {\rm SSYT}_{\le k; m}$ we define an equivalence relation 
\[
S\sim T \text{ if and only if } S_{\text{red}} = T_{\text{red}}.
\]
Note that if $T \sim T'$, then $T$ and $T'$ have the same number of rows. 
We denote by ${\rm SSYT}_{\le k; m,\sim}$ the set of $\sim$-equivalence classes in ${\rm SSYT}_{\le k; m}$. 
With a slight abuse of notation, we write $T \in {\rm SSYT}_{\le k; m,\sim}$ instead of $[T] \in {\rm SSYT}_{\le k; m,\sim}$.

Recall that we denote by $U \subset SL_n$ the subgroup of unipotent upper triangular matrices. 
The dual canonical basis of $\kk[U]$ is parametrized by the set ${\rm SSYT}_{\le n-1; n, \sim}$ of equivalence classes of semistandard Young tableaux \cite[Theorem 4.8]{Li20}. 
For every $T\in {\rm SSYT}_{\le n-1; n, \sim}$ denote by $\ch(T)$ the corresponding element in the dual canonical basis of $\kk[U]$. 
The dual canonical basis for $\kk[\flag_{1,\dots,n-1;n}]=\kk[U^-\backslash SL_n]$ is indexed by ${\rm SSYT}_{\le n-1;n}$ and it is obtained from the dual canonical basis of $\kk[U]$ by lifting $\ch(T)\in \kk[U]$ along $\iota^*:\kk[\flag_{1,\dots,n-1;n}]\to \kk[U]$ 
(as in \cite[Proof of Lemma 2.4]{GLS_partial-flags} discussed in \S\ref{sec:U}).
Every homogeneous element in the fibre $(\iota^*)^{-1}(\ch(T))$ obtained from the unique lift $\widehat{\ch}(T)$ by multiplying with a monomial in $P_{[i]}$, $1\le i\le n-1$ is a dual canonical basis element for $\kk[\flag_{1,\dots,n-1;n}]$.
The indexing set in this case is ${\rm SSYT}_{\le n-1;n}$ as follows from the discussion in \S\ref{subsec:from tableaux to dual canonical basis eleemnts in kU}.
By \eqref{eq:kUd in kU} $\kk[U_{d_1,\dots,d_k;n}]$ is a subalgebra in $\kk[U]$ so we refrain from treating this case separately, unless $k=1$.

We denote by ${\rm SSYT}_{k;n}$ (resp. ${\rm SSYT}_{k;n, \sim}$) the subset of ${\rm SSYT}_{\le k;n}$ (resp. ${\rm SSYT}_{\le k; n, \sim}$) consisting of all semistandard Young tableaux of rectangular shapes with $k$ rows. 
Denote by $\kk[\Gr_{k;n,\sim}]$ the quotient of $\kk[\Gr_{k;n}]$ by the inhomogeneous ideal $\langle P_{[i,i+k-1]}-1: 1 \le i \le n-k+1 \rangle$. The dual canonical basis of $\kk[\Gr_{k;n,\sim}]$ is parametrized by the set ${\rm SSYT}_{k; n, \sim}$ of equivalence classes of semistandard Young tableaux, \cite[Theorem 3.25]{CDFL}. The dual canonical basis of $\kk[\Gr_{k;n}]$ is parametrized by the set ${\rm SSYT}_{k;n}$ of semistandard Young tableaux, \cite[Theorem 3.1]{DGL}. 

\subsection{Partial order on tableaux}\label{sec:partial order} 
There is a partial order on the set of semistandard Young tableaux that is  induced from a partial order on partitions as follows \cite[Section 5.5]{Brini05}.

Let $\lambda = (\lambda_1,\dots,\lambda_\ell)$, $\mu = (\mu_1,\dots,\mu_\ell)$, with $\lambda_1 \geq \cdots \geq \lambda_\ell \geq 0$, $\mu_1 \ge \cdots \ge \mu_{\ell} \ge 0$, be partitions. 
Then the \emph{dominance order on partitions} is defined as
\[
\lambda \le \mu\quad  \text{  if and only if } \quad  \sum_{j \leq i}\lambda_j \le \sum_{j \leq i}\mu_j \text{ for all } 1\le i\le \ell.
\]
For $T\in {\rm SSYT}_{\le k; m}$ and $i \in [m]$, denote by $T[i]$ the sub-tableau obtained from $T$ by restriction to the entries in $[i]$. 
For a tableau $T$, let ${\rm sh}(T)$ denote the \emph{shape of $T$}, that is a partition $(\lambda_1,\dots,\lambda_\ell)$ where $\ell\le k$ is the number of rows of $T$ and $\lambda_i$ is the number of entries in the $i^{\text{th}}$ row. 
If $T,T' \in {\rm SSYT}_{\le k; m}$ are of the same shape, Then the \emph{dominance order on tableaux} is defined as 
\[
T \le T' \quad \text{if and only if} \quad  {\rm sh}(T[i]) \le {\rm sh}({T'}[i]) \quad \text{for all} \quad 1\le i\le m.
\]

\subsection{\texorpdfstring{From tableaux to dual canonical basis elements in $\kk[U]$}{From tableaux to dual canonical basis elements in kk[U]}} \label{subsec:from tableaux to dual canonical basis eleemnts in kU}

We recall how to obtain a dual canonical basis element $\widetilde{\ch}(T) \in \kk[U]$ from a tableau $T \in {\rm SSYT}_{\le n-1; n, \sim}$, see \cite[\S5]{Li20}. 
There is a natural bijection between one-column tableaux $T$ and Plücker coordinates with index set given by the filling of $T$.
Monomials in Plücker coordinates whose index sets form columns of a semistandard Young tableaux (up to reordering) are called \emph{standard monomials}.
For a tableau $T \in {\rm SSYT}_{\le k; n}$ we denote by $P_{T} = P_{T_1} \cdots P_{T_m}$ the standard monomial of $T$, where $T_1, \ldots, T_m$ are columns of $T$.
We define a \emph{fundamental tableau} as a one-column tableau $T_{(a,b)}$ with entries $\{1,2,\ldots, a-1, b\}$ for some $a \in [n-1]$ and $b \in [p+1, n]$.

For the rest of this subsection let us fix a tableau $T \in {\rm SSYT}_{\le n-1; n, \sim}$ and construct its associated element of the dual canonical basis.
First take the unique tableau $T'$ that satisfies
\begin{enumerate}
    \item $T\sim T'$,
    \item columns of $T'$ are fundamental tableaux of form $T_{(a_i, b_i)}$ with $1\le i\le m$, where $m$ is number of columns of $T'$.
\end{enumerate}
By definition, $a_1, \ldots, a_m \in [n-1]$ and $b_1, \ldots, b_m \in [n]$.  
We denote ${\bf p}_T := \{ (a_i, b_i): i \in [m]\}$ as a multi-set. 
Let ${\bf i}_T =(i_1, \ldots, i_m)$, respectively ${\bf j}_T = (j_1, \ldots, j_m)$, be such that 
\begin{eqnarray*}
i_1 \leq \dots \leq i_m \quad &\text{and}&  \quad \{ i_1,\dots,i_m\}=\{a_1, \ldots, a_m\},\quad  \text{respectively}\\
j_1 \leq \dots \leq j_m \quad &\text{and}& \quad \{j_1,\dots,j_m\}=\{b_1, \ldots, b_m\}.
\end{eqnarray*}
For ${\bf c}=(c_1, \ldots, c_m), {\bf d} = (d_1, \ldots, d_m) \in \ZZ^m$ we write ${\bf p}_{{\bf c}, {\bf d}} := \{ (c_i, d_i): i \in [m] \}$ as a multi-set. 
Let $S_m$ be the symmetric group on $[m]$ and $\ell(w)$ the length of $w \in S_m$; denote by $w_0 \in S_m$ be the longest permutation. 
For any $T \in {\rm SSYT}_{\le n-1; n, \sim}$, there exists $w \in S_m$ such that ${\bf p}_{T} = {\bf p}_{w \cdot {\bf i}_T, {\bf j}_T}$, and let $w_T \in S_m$ be the unique permutation with maximal length such that ${\bf p}_{T} = {\bf p}_{w_T \cdot {\bf i}_T, {\bf j}_T}$, see \cite[Proposition 2.7]{Billey_etal_18}, \cite[\S2.4-5]{Bourbaki} and \cite[Proposition 2.3]{Kob11}.
 
Let $T,T' \in {\rm SSYT}_{\le n-1;n,\sim}$ be as above.
For $u \in S_m$ we define a standard monomial $P_{u;T} \in \kk[U]$ as follows. 
If $j_a \in [i_{u(a)}, i_{u(a)}+n]$ for all $a \in [m]$, define $\alpha(u;T)$ as the tableau with columns $T_{(i_{u(a)}, j_a)}$, $a \in [m]$, and define $P_{u; T} := P_{\alpha(u;T)} \in \kk[U]$ as its standard monomial. 
If $j_a \not\in [i_{u(a)}, i_{u(a)}+n]$ for some $a \in [m]$, then $\alpha(u;T)$ is undefined and $P_{u;T} := 0$. 

Let $T,T' \in {\rm SSYT}_{\le n-1; n,\sim}$ be as above. Then by \cite[Theorem 5.3]{Li20}
\begin{align}  \label{eq:formula of chT}
\widetilde{\ch}(T) = \sum_{u \in S_m} (-1)^{\ell(uw_T)} p_{uw_0, w_Tw_0}(1) P_{u; T'} \in \kk[U],
\end{align}
where $p_{u,v}(q)$ is a Kazhdan-Lusztig polynomial \cite{KL79}. 
The expression (\ref{eq:formula of chT}) is of the form $\widetilde{\ch}(T) = \sum_{S} c_S P_S$, where $S$'s are semistandard tableaux. 
The largest tableau in the expression is $T'$ and $c_{T'} = 1$. 
After removing the trivial tableaux in $T'$, we obtain the tableau $T$ back. 
Denote by $T''$ the unique trivial tableau such that $T' = T \cup T''$. The expression (\ref{eq:formula of chT}) can be simplified to 
\begin{align} \label{eq:formula of chT for kU after simplification}
\ch(T) = \frac{1}{P_{T''}} \widetilde{\ch}(T).    
\end{align}
Since $P_{T''}$ is equal to $1$ on $\kk[U]$, $\ch(T)$ and $\widetilde{\ch}(T)$ take the same value on $\kk[U]$.  
The dual canonical basis for $\kk[\flag_{1,\dots,n-1;n}]=\kk[U^-\backslash SL_n]$ is obtained from the dual canonical basis of $\kk[U]$ via the unique lifting (see Theorem 1.1 in \cite{Kad23} and \S\ref{sec:U}) and multiplying lifts with monomials in $P_{[d]}$ where $1\le d\le n-1$.
Notice that $P_{T''}$ above is such a monomial.

\subsection{\texorpdfstring{From tableaux to dual canonical basis elements in $\kk[\Gr_{k;n}]$}{From tableaux to dual canonical basis elements in kk[Gr\_k;n]}} \label{subsec:from tableaux to dual canonical basis eleemnts in CGrkn}
The map from tableaux to dual canonical basis elements in $\kk[\Gr_{k;n}]$ \cite[\S5]{CDFL} is similar to the case of $\kk[U]$, but they differ in the notion of fundamental tableaux. 
A \emph{fundamental tableau} in $\SSYT_{k;n}$ is a one-column tableau whose entries are of the form $\{i,i+1, \ldots, \widehat{j}, \ldots, i+k\}$, where $1\le i\le n-k$, $i < j < i+k$, and $\widehat{j}$ means the number $j$ is missing. 

We recall the map from tableaux to dual canonical basis elements in $\kk[\Gr_{k;n}]$. 
Let $T \in \SSYT_{k;n}$. 
Denote by $T'$ the unique tableau whose columns are  fundamental tableaux in $\SSYT_{k;n}$ and $T\sim T'$.
Let $m$ denote the number of columns of $T'$. 
Let ${\bf i} = i_1 \leq i_2 \dots \leq i_m$ be the entries in the first row of $T'$, and let $r_1,\dots,r_m$ be the elements such that the $a$th column of $T'$ has content $[i_a,i_a+k] \setminus \{r_a\}$. Let ${\bf j} = j_1 \leq j_2 \leq \dots \leq j_m$ be the elements $r_1,\dots,r_m$ written in weakly increasing order. 

For $u \in S_m$, define $P_{u;T'} \in \kk[\Gr_{k;n}]$ as follows. Provided $j_a \in [i_{u(a)}, i_{u(a)}+k]$ for all $a \in [m]$, define $\alpha(u;T')$ as the tableau whose columns have entries $[i_{u(a)}, i_{u(a)}+k] \setminus \{j_a\}$ for $a \in [m]$. 
Then $P_{u; T'} := P_{\alpha(u;T')} \in \kk[\Gr_{k;n}]$ is the corresponding standard monomial. 
If $j_a \notin [i_{u_a}, i_{u(a)}+k]$ for some $a$, then the tableau $\alpha(u;T')$ is undefined and $P_{u ;T'} := 0$. 

There is a unique $u \in S_m$ of maximal length with the property that the sets $\{[i_{u(a)},i_{u(a)}+k] \setminus \{j_a \} \}_{a \in [m]}$ describe the columns of $T'$. This $u$ is denoted by $u = w_{T}$.
By \cite[Theorem 5.8]{CDFL}, the element $\tilde{\ch}(T)$ in the dual canonical basis of $\kk[\Gr_{k;n,\sim}]$ is given by 
\begin{align}\label{eq:formula of tilde ch(T) Grassmannian}
\tilde{\ch}(T) = \sum_{u \in S_m} (-1)^{\ell(uw_T)} p_{uw_0, w_Tw_0}(1) P_{u; T'} \in \kk[\Gr_{k;n,\sim}],
\end{align}
where $p_{u,v}(q)$ is a Kazhdan-Lusztig polynomial \cite{KL79}. 
The formula (\ref{eq:formula of tilde ch(T) Grassmannian}) is a change of basis from the standard monomial basis to the dual canonical basis. 
Define 
\begin{align}
\ch(T) = \frac{1}{P_{T''}} \tilde{\ch}(T)    \label{eq:formula of ch(T) Grassmannian}
\end{align}
where $T'' = T' T^{-1}$. The set $\{\ch(T): T \in \SSYT_{k;n}\}$ is the dual canonical basis of $\kk[\Gr_{k;n}]$, see \cite[Theorem 3.1]{DGL}. 

\begin{remark}
The concept of fundamental tableaux in this subsection is different from the concept of fundamental tableaux in \S\ref{subsec:from tableaux to dual canonical basis eleemnts in kU}. We expect that for a rectangular tableau $T \in \SSYT_{k;n}$, after simplification, $\ch(T)$ in (\ref{eq:formula of ch(T) Grassmannian}) is equal to $\ch(T)$ in (\ref{eq:formula of chT for kU after simplification}), \emph{c.f.} Question~\ref{question}.
\end{remark}

\subsection{Mutation of cluster variables in terms of tableaux}

Mutations of cluster variables in the cluster algebra $\kk[U]$ can be described in terms of tableaux \cite[\S6]{Li20}. 
Starting from an initial seed of $\kk[U]$, each time we perform a mutation at a cluster variable $\ch(T_r)$, we obtain a new cluster variable $\ch(T'_r)$ determined by
\begin{align}\label{eq:char mut}
\ch(T'_r)\ch(T_r) = \prod_{i \to r \text{ in }Q} \ch(T_i) + \prod_{r \to j \text{ in }Q} \ch(T_j),
\end{align}
where $\ch(T_i)$ is the cluster variable at the vertex $i$. The two tableaux $\cup_{i \to r} T_i$, $\cup_{r \to j} T_j$ are comparable in the dominance order $\le$ and $T'_r$ is determined by
\begin{align}\label{eq:mutation tableaux}
T'_r = T^{-1}_r \max_{\le}\{\cup_{i \to r} T_i, \cup_{r \to j} T_j \}. 
\end{align}
As $\kk[U_{d_1,\dots,d_k;n}]$ is a subalgebra of $\kk[U]$, if all elements in \eqref{eq:char mut} belong to $\kk[U_{d_1,\dots,d_k;n}]$ then the equation is an exchange relation in $\kk[U_{d_1,\dots,d_k;n}]$.

\subsection{Tableaux for initial minors}\label{sec:initial tableaux}

First we describe the initial cluster variables for the cluster algebra $\kk[\flag_{d_1,\dots,d_k;n}]$ explicitly in terms of Plücker coordinates. 
This will allow us to easily compute their image with respect to the algebra embedding \eqref{eq:embed} and also express them in terms of tableaux.
To simplify notation we denote the lift of a minor $D_I$ from now on by $\Delta_I:=\widehat{D}_I$. 
All computations take place in $\kk[\flag_{d_1,\dots,d_k;n}]$.

\begin{lemma}\label{lem:initial minors}
The index sets $I_F$ associated with faces $F$ of $\mathcal P_{d_1,\dots,d_k;n}$ are

\noindent
{\bf frozens} (faces unbounded to the left): $\{d_1\},[d_1-1,d_1],\dots,[d_1],[d_1]\cup\{d_2\}, [d_1]\cup[d_2-1,d_2],\dots[d_2], [d_2]\cup \{d_3\},\dots [d_k],[d_k]\cup\{n\},\dots,[n]\setminus{d_k+1}$.

\noindent
{\bf mutable} (bounded faces):
\begin{enumerate}
        \item for all $i>1$ we have $d_i,[d_i-1,d_i],\dots,[2,d_i]$ (notice that $[d_i]$ appears as frozen);
        \item for all $1\le i<j\le k+1$ we get 
        \[
        \begin{smallmatrix}
            [d_{i-1}+1,d_{j-1}]\cup \{d_j\}, & [d_{i-1}+2,d_{j-1}]\cup \{d_j\},& \dots & [d_i-1,d_{j-1}]\cup \{d_j\} &[d_i,d_{j-1}]\cup \{d_j\}, \\
             [d_{i-1}+1,d_{j-1}]\cup [d_j-1,d_j],& [d_{i-1}+2,d_{j-1}]\cup [d_j-1,d_j] & \dots & [d_i-1,d_{j-1}]\cup [d_j-1,d_j]  &[d_i,d_{j-1}]\cup [d_j-1,d_j] ,\\
             \vdots &&& \\
            [d_{i-1}+1,d_{j-1}]\cup[d_{j-1}+2,d_j],  & [d_{i-1}+2,d_{j-1}]\cup[d_{j-1}+2,d_j],& \dots & [d_{i}-1,d_{j-1}]\cup[d_{j-1}+2,d_j],& [d_{i},d_{j-1}]\cup[d_{j-1}+2,d_j].
        \end{smallmatrix}
        \]
\end{enumerate}
\end{lemma}

In particular, the index sets of initial minors take the form of one or two intervals:
\begin{itemize}
    \item[(a)] $[i_j,d_j]$ for $1\le j\le k$ and $1\le i_j\le d_j$, or
    \item[(b)] $[i_j,d_j]\cup[i_{j+1},d_{j+1}]$ for $1\le i_j\le d_j< i_{j+1}\le d_{j+1}\le n$ and $0\le j<k$
\end{itemize}
(by convention $d_{k+1}=n$ and $d_0=0$). 
In case (a) observe that for any columnset $J$ of appropriate size satisfying $J\cap [1,i_{j}-1]=\varnothing$ we have
\begin{eqnarray} 
\Delta_{[i_j,d_j]}=\Delta_{[i_j,d_j],J}=\Delta_{[1,d_j],[1,i_j-1]\cup J}=P_{[1,i_j-1]\cup J}.    
\end{eqnarray}
This yields an expression of all cluster variables with index set of form (a) in terms of Plücker coordinates:
\begin{eqnarray}\label{eq:Plücker minor}
    \Delta_{[i_j,d_j]} = P_{[1,i_j-1]\cup [n-d_j+i_j,n]}
\end{eqnarray}
For index sets of type (b) we have the following Proposition:

\begin{proposition} \label{prop:expansion of minors}
    Consider an arbitrary flag variety $\flag_{d_1,\dots,d_k;n}$ and an arbitrary initial minor $\Delta_{[i_j,d_{j}]\cup [i_{j+1},d_{j+1}]}$ with $1\le i_j\le d_j<i_{j+1}\le d_{j+1}\le n$ and $0\le j<k$ (recall, that $d_0:=0,d_{k+1}:=n$).  Set $\ell=n-d_j-d_{j+1}+i_j+i_{j+1}-1$.
    Then
    \begin{eqnarray}\label{eq:laplace initial minor}
    \Delta_{[i_j,d_{j}]\cup [i_{j+1},d_{j+1}]}=\sum_{J\in\binom{[\ell,n]}{d_j-i_j+1}, \ J'=[\ell,n]\setminus J} (-1)^{\Sigma(i_j,d_j,J)} P_{[i_j-1]\cup J}P_{[i_{j+1}-1]\cup J'}
    \end{eqnarray}
    where $\Sigma(i_j,d_j,J):=\sum_{q=i_j}^{d_j} q+\sum_{j\in J}j$.
\end{proposition}
%Notice that \eqref{eq:Plücker minor} is a special case of \eqref{eq:laplace initial minor} with one of the two intervals being $\varnothing$.

\begin{proof}
The minor $\Delta_{[i_j,d_{j}]\cup [i_{j+1},d_{j+1}]}$ is the determinant of a matrix of size $n-\ell+1=d_j+d_{j+1}-i_j-i_{j+1}+2$.
To be precise
\[
\Delta_{[i_j,d_{j}]\cup [i_{j+1},d_{j+1}]} = \Delta_{[i_j,d_{j}]\cup [i_{j+1},d_{j+1}],[\ell,n]}.
\]
Observe that for any set $J\subset [\ell, n]$ of size $d_{j}-i_j+1$ we have 
\begin{eqnarray*}
\Delta_{[i_j,d_j],J}=\Delta_{[1,d_j],[1,i_j-1]\cup J}=P_{[1,i_j-1]\cup J}.    
\end{eqnarray*}
%The case $i_j=1$ is included by setting $[1,0]=\varnothing$.
Equation~\eqref{eq:laplace initial minor} is a Laplace expansion:
\begin{eqnarray*}
    \Delta_{[i_j,d_{j}]\cup [i_{j+1},d_{j+1}],[\ell,n]}&\overset{\text{Laplace}}{=}& \sum_{J\in\binom{[\ell,n]}{d_j-i_j+1}, \ J'=[\ell,n]\setminus J} (-1)^{\Sigma(i_j,d_j,J)} \Delta_{[i_j,d_j], J}\Delta_{[i_{j+1},d_{j+1}],J'}\\
    &\overset{\eqref{eq:Plücker minor}}{=}&\sum_{J\in\binom{[\ell,n]}{d_j-i_j+1}, \ J'=[\ell,n]\setminus J} (-1)^{\Sigma(i_j,d_j,J)} P_{[i_j-1]\cup J}P_{[i_{j+1}-1]\cup J'}
\end{eqnarray*}
\end{proof}

We can now deduce the tableaux corresponding to the initial cluster variables.

\begin{corollary}\label{cor:initial tableau}
The tableau corresponding to the initial cluster variable $\Delta_{[i_j,d_j]\cup[i_{j+1},d_{j+1}]}\in \kk[\flag_{d_1,\dots,d_k;n}]$ in Proposition \ref{prop:expansion of minors} is 
\begin{align}\label{eq:initial tableau}
T_{\Delta_{[i_j,d_j]\cup[i_{j+1},d_{j+1}]}}:=\begin{matrix}
1 & 1 \\
\vdots & \vdots \\
\vdots & i_{j}-1 \\
i_{j+1}-1 & n-d_{j+1}-d_j + i_{j+1} + i_j -1 \\
n-d_{j+1}+i_{j+1} & \vdots \\
\vdots & n-d_{j+1}+i_{j+1} -1 \\
n 
\end{matrix}.
\end{align}

\end{corollary}
We call the tableaux of form $T_{\Delta_{[i_j,d_j]\cup[i_{j+1},d_{j+1}]}}$ \emph{initial tableaux}.

\begin{proof}
Notice that the two columns correspond to one of the monomials in the Plücker expression of $\Delta_{[i,d_i]\cup[i_{j+1},d_{j+1}]}$ in \eqref{eq:laplace initial minor}. 
Moreover, $\Delta_{[i,d_i]\cup[i_{j+1},d_{j+1}]}$ is an element of the dual canonical basis.
Hence, by \eqref{eq:formula of chT} it can be expressed as $P_T + \sum_{T' < T} c_{T'} P_{T'}$, where $P_T'$ are standard monomials. Note that in the right hand side of (\ref{eq:laplace initial minor}), each term $P_{[i_j-1]\cup J}P_{[i_{j+1}-1]\cup J'}$ is equal to $P_{T'}$, where the first column of $T'$ has entries $[i_{j+1}-1]\cup J'$ and the second column of $T'$ has entries $[i_j-1]\cup J$, and $T'$ is semistandard. The tableau $T$ corresponding to the dual canonical basis element is the largest tableau appearing in the expression (\ref{eq:laplace initial minor}). Further, the tableau $T_{\Delta_{[i_j,d_j]\cup[i_{j+1},d_{j+1}]}}$ of the claim is indeed the largest with respect to the partial order defined in \S\ref{sec:partial order} as $T_{\Delta_{[i_j,d_j]\cup[i_{j+1},d_{j+1}]}} = P_{[i_j-1]\cup J}P_{[i_{j+1}-1]\cup J'}$, where $J'=[n-d_{j+1}+i_{j+1},n]$ is the largest possible value among all choices of $J'$.  
%Therefore we only need to verify that $T_{\Delta_{[i_j,d_j]\cup[i_{j+1},d_{j+1}]}}$ is the largest tableau in  \eqref{eq:laplace initial minor}.
    
%First we note that all tableaux representing terms in the Pl\"{u}cker expression in Proposition \ref{prop:expansion of minors} are already semistandard. Further, the tableau of the claim is indeed the largest with respect to the partial order defined in \S\ref{sec:partial order} as it contains the largest possible entries among all $J$ and $J'$.
\end{proof}

Now consider the embedding of algebras defined in Equation \eqref{eq:embed}. 
Recall that Plücker coordinates $P_I$ correspond to single column tableaux of length (\emph{i.e.} number of rows) equal to the cardinality of the index set $I$.
The embedding $\varphi^*:\kk[\flag_{d_1,\dots,d_k;n}]\hookrightarrow \kk[\Gr_{d_k;N}]$ from \eqref{eq:embed} \emph{fills up} the index set $I$ until it is of size $d_k$.
In particular, the image of an initial cluster variables $\Delta_{[i_j,d_j]\cup[i_{j+1},d_{j+1}]}$ is determined by filling up the index sets of the Plücker coordinates of its Laplace expansion.
From the proof of Corollary~\ref{cor:initial tableau} we may therefore deduce the tableaux of $\varphi^*(\Delta_{[i_j,d_j]\cup[i_{j+1},d_{j+1}]})$: it is precisely given by filling up the tableaux $T_{[i_j,d_j]\cup[i_{j+1},d_{j+1}]}$, that is $\phi(T_{[i_j,d_j]\cup[i_{j+1},d_{j+1}]})$ for $\phi$ defined in \eqref{eq:map of tableaux}.

\begin{corollary}\label{cor:image of initial tableaux}

The tableau of the image $\varphi^*(\Delta_{[i_j,d_j]\cup[i_{j+1},d_{j+1}]})\in \kk[\Gr_{d_k;N}]$ of an initial cluster variable $\Delta_{[i_j,d_j]\cup[i_{j+1},d_{j+1}]}\in \kk[\flag_{d_1,\dots,d_k;n}]$ is
\[
\begin{matrix}
1 & 1 \\
\vdots & \vdots \\
\vdots & i_{j}-1 \\
i_{j+1}-1 & n-d_{j+1}-d_j + i_{j+1} + i_j -1 \\
n-d_{j+1}+i_{j+1} & \vdots \\
\vdots & n-d_{j+1}+i_{j+1} -1 \\
\vdots & {n+1}\\
n  & \vdots \\
n+1 & \vdots\\
\vdots & \vdots \\
n+d_k-d_{i_{j+1}} &  n+d_k-d_{i_{j}}
\end{matrix}
=\phi(T_{\Delta_{[i_j,d_j]\cup[i_{j+1},d_{j+1}]}}).
\]

\end{corollary}

As mentioned in the introduction, by exhibiting explicit mutation sequences in \S\ref{sec:mutation sequence} we show that the images with respect to $\varphi^*$ of the initial minors for $\flag_{d_1,\ldots,d_k;n}$ are cluster variables in $\kk[\Gr_{d_k;N}]$.
From this we deduce the proof of Corollary~\ref{cor:character compatinility initials} stated in the introduction:

\begin{proof}[Proof of Corollary~\ref{cor:character compatinility initials}]
Let $[i_j,d_j]\cup[i_{j+1},d_{j+1}]$ be the index set of an initial minor in the respective cases according to Lemma~\ref{lem:initial minors} with associated tableau $T_{\Delta_{[i_j,d_j]\cup[i_{j+1},d_{j+1}]}}$ as in Corollary~\ref{cor:initial tableau}.
Corollary~\ref{cor:image of initial tableaux} shows that $\phi(T_{\Delta_{[i_j,d_j]\cup[i_{j+1},d_{j+1}]}})$ is the tableau of $\varphi^*(\ch(T_{\Delta_{[i_j,d_j]\cup[i_{j+1},d_{j+1}]}}))$.
Moreover, we know from \S\ref{sec:mutation sequence} that $\varphi^*(\ch(T_{\Delta_{[i_j,d_j]\cup[i_{j+1},d_{j+1}]}}))$ is a cluster variables and hence an element of the dual canonical basis for $\kk[\Gr_{d_k;N}]$.
Its expression in terms of Plücker coordinates is obtained by applying $\varphi^*$ to \eqref{eq:laplace initial minor}.
    
Observe that $\varphi^*$ maps standard monomials in Plücker coordinates in $\kk[\flag_{d_1,\ldots,d_k;n}]$ to standard monomials in Plücker coordinates in $\kk[\Gr_{d_k;N}]$.
This happens precisely because it amounts to applying the map of semistandard Young tableaux $\phi$ to the tableaux associated with a standard monomial.
Every dual canonical basis element admits a unique expression in the basis of standard monomials which is its character. Hence, the obtained expression is $\ch(\phi(T_{\Delta_{[i_j,d_j]\cup[i_{j+1},d_{j+1}]}}))$.
\end{proof}

As stated in the introduction, it would be interesting to see if $\varphi^*$ more generally maps the character of a tableau $T\in\SSYT_{d_1,\dots,d_k;n}$ to the character of its image the tableau $\phi(T)\in \SSYT_{d_k;N}$, \emph{i.e.} if Question~\ref{question} has an affirmative answer.

\begin{figure}
    \centering
\adjustbox{scale=0.45}{\begin{tikzcd}
	5 &&& {} \\
	 \\
	4 & {} &&& \\
	\boxed{\begin{matrix} 4 \\ 5 \end{matrix}}&& \boxed{\begin{matrix} 1 \\ 5 \end{matrix}} && \boxed{\begin{matrix} 1 \\ 2 \end{matrix}} \\
	3 &&&&&&& { } \\
	\boxed{\begin{matrix} 1 & 3  \\ 2 & 4 \\ 3 \\ 5 \end{matrix}} && {\begin{matrix} 1 & 1 \\ 2 & 4 \\ 3 \\ 5 \end{matrix}} &&&& {\begin{matrix} 1 \\ 2 \\ 3 \\ 5 \end{matrix}} && \\
	2 &&&&& { } \\
	\boxed{\begin{matrix} 2 \\ 3 \\ 4 \\ 5 \end{matrix}} && {\begin{matrix} 1 \\ 3 \\ 4 \\ 5 \end{matrix}} && {\begin{matrix} 1 \\ 2 \\ 4 \\ 5 \end{matrix}} &&&& \boxed{\begin{matrix} 1 \\ 2 \\3 \\4 \end{matrix}} \\
	1 &&&&&&&&& { } \\
	& 1 && 2 && 3 && 4 && 5
	%\arrow[no head, from=1-1, to=1-4]
	%\arrow[no head, from=1-4, to=10-4]
    \arrow[teal, no head, from=1-1, to=10-4, rounded corners, to path=-| (\tikztotarget)]
	%\arrow[no head, from=3-1, to=3-2]
	%\arrow[no head, from=3-2, to=10-2]
    \arrow[teal, no head, from=3-1, to=10-2, rounded corners, to path=-| (\tikztotarget)]
	%\arrow[no head, from=5-1, to=5-8]
	%\arrow[no head, from=5-8, to=10-8]
    \arrow[teal, no head, from=5-1, to=10-8, rounded corners, to path=-| (\tikztotarget)]
	%\arrow[no head, from=7-1, to=7-6]
	%\arrow[no head, from=7-6, to=10-6]
    \arrow[teal, no head, from=7-1, to=10-6, rounded corners, to path=-| (\tikztotarget)]
	%\arrow[no head, from=9-1, to=9-10]
	%\arrow[no head, from=9-10, to=10-10]
    \arrow[teal, no head, from=9-1, to=10-10, rounded corners, to path=-| (\tikztotarget)]
	\arrow[from=6-7, to=4-3]%
	\arrow[from=6-3, to=4-1]
	\arrow[from=8-3, to=6-1]
	\arrow[from=6-1, to=6-3]
	\arrow[from=4-3, to=6-3]%
	\arrow[from=6-3, to=8-3]
	\arrow[from=6-3, to=6-7]
	\arrow[from=8-5, to=6-3]
	\arrow[from=8-1, to=8-3]
	\arrow[from=8-3, to=8-5]
	\arrow[from=6-3, to=4-5]
	\arrow[from=4-5, to=8-5]
	\arrow[from=6-7, to=8-9]
\end{tikzcd}$\quad \quad \quad$
\begin{tikzcd}
	{6} &&& {} \\
	&& {} \\
	{5} & {} \\
	\boxed{\begin{matrix} 5 \\ 6 \end{matrix}} && \boxed{\begin{matrix} 1 \\ 6 \end{matrix}} && \boxed{\begin{matrix} 1 \\ 2 \end{matrix}} \\
	{4} &&& {} &&&& {} \\
	\boxed{\begin{matrix} 1 & 4 \\ 2 & 5 \\ 3 \\ 6 \end{matrix}} & {} & {\begin{matrix} 1 & 1 \\ 2 & 5 \\ 3 \\ 6 \end{matrix}} && {} && {\begin{matrix} 1 \\ 2 \\ 3 \\6 \end{matrix}} \\
	{3} &&&&& {} \\
	\boxed{\begin{matrix} 3 \\ 4 \\ 5 \\ 6 \end{matrix}} & {} & {\begin{matrix} 1 \\ 4 \\ 5 \\6 \end{matrix}} & {} & {\begin{matrix} 1 \\ 2 \\ 5 \\6 \end{matrix}} & {} \\
	{2} &&&&&&&&&&& {} \\
	\boxed{\begin{matrix} 2 \\ 3 \\ 4 \\5 \end{matrix}} & {} & {\begin{matrix} 1 \\ 3 \\ 4 \\5 \end{matrix}} & {} & {\begin{matrix} 1 \\ 2 \\ 4 \\5 \end{matrix}} && {\begin{matrix} 1 \\ 2 \\ 3 \\5 \end{matrix}} && \boxed{\begin{matrix} 1 \\ 2 \\ 3 \\4 \end{matrix}} \\
	{1} &&&&&&&&& {} \\
	& 1 && { 2} && 3 && 4 && 5 && 6
	%\arrow[no head, from=11-1, to=11-10]
	%\arrow[no head, from=11-10, to=12-10]
   \arrow[teal, no head, from=11-1, to=12-10, rounded corners, to path=-| (\tikztotarget)]
	%\arrow[no head, from=9-1, to=9-12]
	%\arrow[no head, from=9-12, to=12-12]
   \arrow[teal, no head, from=9-1, to=12-12, rounded corners, to path=-| (\tikztotarget)]
	%\arrow[no head, from=7-1, to=7-6]
	%\arrow[from=7-6, to=12-6]
   \arrow[teal, no head, from=7-1, to=12-6, rounded corners, to path=-| (\tikztotarget)]
	%\arrow[no head, from=5-1, to=5-8]
	%\arrow[no head, from=5-8, to=12-8]
   \arrow[teal, no head, from=5-1, to=12-8, rounded corners, to path=-| (\tikztotarget)]
	%\arrow[no head, from=3-1, to=3-2]
	%\arrow[no head, from=3-2, to=12-2]
   \arrow[teal, no head, from=3-1, to=12-2, rounded corners, to path=-| (\tikztotarget)]
	%\arrow[no head, from=1-1, to=1-4]
	%\arrow[no head, from=1-4, to=12-4]
   \arrow[teal, no head, from=1-1, to=12-4, rounded corners, to path=-| (\tikztotarget)]
	\arrow[from=6-3, to=4-1]
	\arrow[from=10-1, to=10-3]
	\arrow[from=10-3, to=10-5]
	\arrow[from=10-5, to=10-7]
	\arrow[from=8-1, to=8-3]
	\arrow[from=8-3, to=8-5]
	\arrow[from=10-7, to=8-5]
	\arrow[from=10-5, to=8-3]
	\arrow[from=10-3, to=8-1]
	\arrow[from=8-5, to=10-5]
	\arrow[from=8-3, to=10-3]
	\arrow[from=6-3, to=8-3]
	\arrow[from=8-3, to=6-1]
	\arrow[from=6-1, to=6-3]
	\arrow[from=8-5, to=6-3]
	\arrow[from=6-3, to=6-7]
	\arrow[from=6-7, to=4-3]
	\arrow[from=6-7, to=10-7]
    \arrow[from=10-7, to=10-9]
    \arrow[from=4-3, to=6-3]
    \arrow[from=4-5, to=8-5]
    \arrow[from=6-3, to=4-5]
\end{tikzcd} 
}
    \caption{The initial seeds for the partial flag varieties $\flag_{2,4;5}$ (on the left) and $\flag_{2,4;6}$ (on the right).}
    \label{fig:initial seed for Fl245 and Fl246}
\end{figure}

\subsection{Examples} \label{subsec:examples}

We exhibit the results of the previous subsection in some examples.

\begin{example}
In the case of $\flag_{2,4;5}$, there are two cluster variables, one mutable one frozen, with two-column tableaux. More precisely we have 
\begin{align*}
\ch\left( { \begin{smallmatrix}
1 & 3 \\ 2 & 4 \\ 3 \\ 5
\end{smallmatrix}}\right) = P_{1235}P_{34} - P_{1234}P_{35}, \quad\text{and} \quad \ch\left( \begin{smallmatrix} 
1 & 1 \\ 2 & 4 \\ 3 \\ 5
\end{smallmatrix}\right) = P_{1235}P_{14} - P_{1234}P_{15}.
\end{align*} 
The initial seed is shown in Figure \ref{fig:initial seed for Fl245 and Fl246} on the left.
\end{example}

\begin{example}
In the case of $\flag_{2,4;6}$, there are two cluster variables, one mutable one frozen, with two-column tableaux. More precisely we have 
\begin{align*}
& \ch\left( { \begin{smallmatrix} 
1 & 4 \\ 2 & 5 \\ 3 & \\ 6& 
\end{smallmatrix}}\right) = P_{1234}P_{56}-P_{1235}P_{46}+P_{1236}P_{45}, \quad\text{and}\quad
& \ch\left( { \begin{smallmatrix} 
1 & 1 \\ 2 & 5 \\ 3 \\ 6
\end{smallmatrix}}\right) = P_{1236}P_{15} - P_{1235}P_{16}.
\end{align*} 
The initial seed is shown in Figure~\ref{fig:initial seed for Fl245 and Fl246} on the right.
\end{example}

\begin{example}
In Figure~\ref{fig:initial seed for Fl257} we display the initial seed for $\flag_{2,5;7}$ in terms of tableaux instead of initial minors for direct comparison with Figure~\ref{fig:P and Q257} above.
\end{example}

 \begin{figure}
     \centering
 \adjustbox{scale=0.5}{\begin{tikzcd}
 	7 &&&& {} \\
 	\\
 	6 && {} \\
 	& \boxed{\begin{matrix} 6 \\ 7 \end{matrix}} && \boxed{\begin{matrix} 1 \\ 7 \end{matrix}} && \boxed{{\begin{matrix} 1 \\ 2 \end{matrix}}} \\
 	5 &&&&&&&&&& {} \\
 	& \boxed{\begin{matrix} 1&5\\2&6\\3\\4 \\ 7 \end{matrix}} && {\begin{matrix} 1&1\\2&6\\3\\4 \\ 7 \end{matrix}} &&&&&& {\begin{matrix} 1 \\ 2 \\ 3 \\ 4 \\ 7 \end{matrix}} \\
 	4 &&&&&&&& {} \\
 	& \boxed{\begin{matrix} 1&4\\2&5\\3\\6 \\ 7 \end{matrix}} && {\begin{matrix} 1&1\\2&5\\3\\6 \\ 7 \end{matrix}} &&&& {\begin{matrix} 1\\2\\3\\6 \\ 7 \end{matrix}} \\
 	3 &&&&&& {} \\
 	& \boxed{\begin{matrix} 3\\4\\5 \\ 6\\7 \end{matrix}} && {\begin{matrix} 1\\4\\5\\6 \\ 7 \end{matrix}} && {\begin{matrix} 1\\2\\5\\6 \\ 7 \end{matrix}} \\
 	2 &&&&&&&&&&&&&& {} \\
 	& \boxed{\begin{matrix} 2\\3\\4\\5 \\ 6 \end{matrix}} && {\begin{matrix} 1\\3\\4\\5 \\ 6 \end{matrix}} && {\begin{matrix} 1\\2\\4\\5 \\ 6 \end{matrix}} && {\begin{matrix} 1\\2\\3\\5 \\ 6 \end{matrix}} && {\begin{matrix} 1\\2\\3\\4 \\ 6 \end{matrix}} && \boxed{\begin{matrix} 1\\2\\3\\4 \\ 5 \end{matrix}} \\
 	1 &&&&&&&&&&&& {} \\
 	\\
 	&& 1 && 2 && 3 && 4 && 5 && 6 && 7 \\
  \arrow[teal, no head, from=1-1, to=15-5, rounded corners, to path=-| (\tikztotarget)]
   %\arrow[teal, no head, from=1-1, to=1-5]
 	%\arrow[teal,no head, from=1-5, to=15-5]
 \arrow[teal,no head, from=3-1, to=15-3,rounded corners, to path=-| (\tikztotarget)]
    %\arrow[teal,no head, from=3-1, to=3-3]
 	%\arrow[teal,no head, from=3-3, to=15-3]
 \arrow[teal,no head, from=5-1, to=15-11,rounded corners, to path=-| (\tikztotarget)]
    %\arrow[teal,no head, from=5-1, to=5-11]
 	%\arrow[teal,no head, from=5-11, to=15-11]
 \arrow[teal,no head, from=7-1, to=15-9,rounded corners, to path=-| (\tikztotarget)]
    %\arrow[teal,no head, from=7-1, to=7-9]
 	%\arrow[teal,no head, from=7-9, to=15-9]
 \arrow[teal,no head, from=9-1, to=15-7,rounded corners, to path=-| (\tikztotarget)]
    %\arrow[teal,no head, from=9-1, to=9-7]
 	%\arrow[teal,no head, from=9-7, to=15-7]
  \arrow[teal,no head, from=11-1, to=15-15,rounded corners, to path=-| (\tikztotarget)]
 	%\arrow[teal,no head, from=11-1, to=11-15]
 	%\arrow[teal,no head, from=11-15, to=15-15]
  \arrow[teal,no head, from=13-1, to=15-13,rounded corners, to path=-| (\tikztotarget)]
 	%\arrow[teal,no head, from=13-1, to=13-13]
 	%\arrow[teal,no head, from=13-13, to=15-13]
 	\arrow[from=12-6, to=10-4]
 	\arrow[from=12-2, to=12-4]
 	\arrow[from=12-4, to=12-6]
 	\arrow[from=12-6, to=12-8]
 	\arrow[from=12-8, to=12-10]
 	\arrow[from=10-4, to=12-4]
 	\arrow[from=8-4, to=10-4]
 	\arrow[from=6-4, to=8-4]
 	\arrow[from=4-4, to=6-4]
 	\arrow[from=6-2, to=6-4]
 	\arrow[from=6-4, to=4-2]
 	\arrow[from=6-4, to=6-10]
 	\arrow[from=8-2, to=8-4]
 	\arrow[from=10-4, to=8-2]
 	\arrow[from=10-2, to=10-4]
 	\arrow[from=10-4, to=10-6]
 	\arrow[from=10-6, to=8-4]
 	\arrow[from=8-4, to=6-2]
 	\arrow[from=12-8, to=10-6]
 	\arrow[from=8-4, to=8-8]
 	\arrow[from=8-8, to=12-8]
 	\arrow[from=8-8, to=6-4]
 	\arrow[from=10-6, to=12-6]
 	\arrow[from=6-10, to=4-4]
 	\arrow[from=6-10, to=12-10]
 	\arrow[from=12-10, to=8-8]
  \arrow[from=12-10, to=12-12]
  \arrow[from=8-4, to=4-6]
  \arrow[from=4-6, to=10-6]
 \end{tikzcd}}
 \caption{The initial seed for the partial flag variety $\flag_{2,5;7}$.}
     \label{fig:initial seed for Fl257}
 \end{figure}

\section{Mutation sequences from Grassmannians to partial flag varieties} \label{sec:mutation sequence}
%\textcolor{red}{How about changing the name of the section to something that contains "mutation sequences"? So far in fact all the above discussion has been in general.}

In this section, we prove Theorem~\ref{thm:general flag}: we construct a mutation sequence starting from an initial seed of Grassmannian cluster algebra, which produce an initial seed for a partial flag variety after we freeze some cluster variables and delete some cluster variables. 
We provide a mutation sequence from the initial seed for $\Gr_{d_k;N}$ to a seed $s'$ with the properties
\begin{enumerate}
\item the images under the algebra embedding \eqref{eq:embed} of the initial cluster variables for  $\flag_{d_1,\ldots,d_k;n}$ are cluster variables in $s'$;
\item the initial quiver for $\flag_{d_1,\ldots,d_k;n}$ is a full subquiver of the quiver of $s'$, after freezing.
\end{enumerate}

\subsection{Restricted seeds} \label{subsec:restricted seeds}
Before proceeding we recall the notion of \emph{restricted seed} from \cite[\S4.2]{FWZ_ch4-5}.

\begin{definition}\cite[Definition 4.2.3]{FWZ_ch4-5}\label{def:restricted seed}
Let $(Q,(x_i:i\in Q_0))$ be a seed, and let $I \cup J$ be a partition of the vertex set $Q_0$ of $Q$
such that there are no arrows between mutable vertices in $I$ and vertices in $J$. 
Let $Q'$ be the quiver obtained from $Q$ by deleting all vertices in $J$ (\emph{i.e.} the vertex set of $Q'$ is $I$).
Then the seed $(Q',(x_i:i\in I))$ is called a {\bf restricted seed} of $(Q,(x_i:i\in Q_0))$.
\end{definition}
By \cite[Lemmata 4.2.2 and 4.2.5]{FWZ_ch4-5} passing to a restricted seed commutes with mutation and there yields a \emph{seed subpattern} and induces a \emph{cluster subalgebra} (see \cite[Definition 4.2.6]{FWZ_ch4-5}).

Recall the initial seed for a partial flag variety $(Q_{d_1,\dots,d_k;n},{\bf x}_{d_1,\dots,d_k;n})$.
Denote by $\varphi^*({\bf x}_{d_1,\dots,d_k;n})$ the tuple of images of cluster variables under the ring homomorphism $\varphi^*$ defined in \eqref{eq:embed}.
In this section we exhibit an explicit mutation sequence from the initial seed of the Grassmannian $\Gr_{d_k;N}$ to a seed which contains $(Q_{d_1,\dots,d_k;n}, \varphi^*({\bf x}_{d_1,\dots,d_k;n}))$ as a restricted seed. 
\medskip
 
Recall the initial seed for $\kk[\Gr_{k;n}]$ as constructed in \S\ref{sec:preliminary}, it is shown in Figure~\ref{fig:initial seed for Grkn}.

\begin{figure}[h]
\centering
\adjustbox{scale=0.6}{
\begin{tikzcd}
	{[n-k+1,n]} & {\{1\} \cup [n-k+2,n]} & \cdots & {[1,k-2]\cup [n-1,n]} & {[1,k-1]\cup \{n\}} \\
	{[n-k,n-1]} & {\{1\} \cup [n-k+1,n-1]} & \cdots & {[1,k-2]\cup [n-2,n-1]} & {[1,k-1]\cup \{n-1\}} \\
	\vdots & \vdots & \cdots & \vdots & \vdots \\
	{[3,k+2]} & {\{1\} \cup [4,k+2]} & \cdots & {[1,k-2]\cup [k+1,k+2]} & {[1,k-1]\cup \{k+2\}} \\
	{[2,k+1]} & {\{1\} \cup [3,k+1]} & \cdots & {[1,k-2]\cup [k,k+1]} & {[1,k-1]\cup \{k+1\}} & {[1,k]}
	\arrow[from=1-2, to=2-2]
	\arrow[from=1-3, to=2-3]
	\arrow[from=1-4, to=2-4]
	\arrow[from=1-5, to=2-5]
	\arrow[from=2-1, to=2-2]
	\arrow[from=2-2, to=1-1]
	\arrow[from=2-2, to=2-3]
	\arrow[from=2-2, to=3-2]
	\arrow[from=2-3, to=1-2]
	\arrow[from=2-3, to=2-4]
	\arrow[from=2-3, to=3-3]
	\arrow[from=2-4, to=1-3]
	\arrow[from=2-4, to=2-5]
	\arrow[from=2-4, to=3-4]
	\arrow[from=2-5, to=1-4]
	\arrow[from=2-5, to=3-5]
	\arrow[from=3-1, to=3-2]
	\arrow[from=3-2, to=2-1]
	\arrow[from=3-2, to=3-3]
	\arrow[from=3-2, to=4-2]
	\arrow[from=3-3, to=2-2]
	\arrow[from=3-3, to=3-4]
	\arrow[from=3-3, to=4-3]
	\arrow[from=3-4, to=2-3]
	\arrow[from=3-4, to=3-5]
	\arrow[from=3-4, to=4-4]
	\arrow[from=3-5, to=2-4]
	\arrow[from=3-5, to=4-5]
	\arrow[from=4-1, to=4-2]
	\arrow[from=4-2, to=3-1]
	\arrow[from=4-2, to=4-3]
	\arrow[from=4-2, to=5-2]
	\arrow[from=4-3, to=3-2]
	\arrow[from=4-3, to=4-4]
	\arrow[from=4-3, to=5-3]
	\arrow[from=4-4, to=3-3]
	\arrow[from=4-4, to=4-5]
	\arrow[from=4-4, to=5-4]
	\arrow[from=4-5, to=3-4]
	\arrow[from=4-5, to=5-5]
	\arrow[from=5-1, to=5-2]
	\arrow[from=5-2, to=4-1]
	\arrow[from=5-2, to=5-3]
	\arrow[from=5-3, to=4-2]
	\arrow[from=5-3, to=5-4]
	\arrow[from=5-4, to=4-3]
	\arrow[from=5-4, to=5-5]
	\arrow[from=5-5, to=4-4]
	\arrow[from=5-5, to=5-6]
\end{tikzcd}
}
\caption{The initial seed for $\CC[\Gr_{k;n}]$ we used.}
\label{fig:initial seed for Grkn}
\end{figure}

\subsection{Mutation sequence for two-step partial flag varieties} \label{subsec:two-step partial flag varieties}

We start by describing the case of two-step partial flag varieties which serves as a building block for the general case.
The mutations for the case of two-step partial flag varieties $\flag_{d_1,d_2;n}$ is to perform a mutation sequence in the $(d_2-d_1-1) \times (d_2-2)$ rectangle with left corner in the second row and second column in the initial seed. The region where the cluster variables in partial flag variety appears is a rectangle of size $(d_2-d_1-1) \times d_1$ with left corner in the second row and second column. 

Denote $a = d_2-d_1-1$, $b=d_2-2$, $c=d_1$. We denote the coordinates of the region of the $(d_2-d_1-1) \times (d_2-2)$ rectangle with left corner in the second row and second column by $(i,j)$, $i \in [a]$, $j \in [b]$, where the left upper corner has coordinate $(a,1)$, the right upper corner has coordintate $(a,b)$, the left lower corner has coordinate $(1,1)$, the lower right corner has coordinate $(1,b)$, see Figure \ref{fig:mutation region for 2 step partial flag varieties}.
{The mutation sequence} is similar to a maximal green sequence for the Grassmannian \cite[\S11]{Marsh-Scott} and we use similar notation: a \emph{page} in the sequence refers to a subsequence of consecutive mutations so that each vertex in the mutation grid is mutated at most once. The mutation sequence is as follows.
\begin{itemize}
\item The first page consists of the following mutations 
\begin{align*}
& (a,1), (a,2), \ldots, (a,b), \\
& (a-1,1), (a-1,2), \ldots, (a-1,b), \\
& \vdots \\
& (2,1), (2,2), \ldots, (2,b), \\
& (1,1), (1,2), \ldots, (1,c).
\end{align*}
After performing these mutations the cluster variable with tableaux $\begin{matrix}
1 & d_1+2 \\ \vdots & \vdots \\ d_1+1 & 2d_1+1 \\ n-d_2+d_1+2 \\ \vdots \\ n
\end{matrix}$ appears in position $(1,c)$. 

\item The first page consists of the following mutations
\begin{align*}
& (a,1), (a,2), \ldots, (a,b), \\
& (a-1,1), (a-1,2), \ldots, (a-1,b), \\
& \vdots \\
& (3,1), (3,2), \ldots, (3,b), \\
& (2,1), (2,2), \ldots, (2,c), \\
& (1,1), (1,2), \ldots, (1,c-1).
\end{align*}
The two cluster variables $\begin{matrix}
1 & d_1+3 \\ \vdots & \vdots \\ d_1+2 & 2d_1+2 \\ n-d_2+d_1+3 \\ \vdots \\ n
\end{matrix}$ and $\begin{matrix}
1 & 1 \\ 2 & d_1+3 \\ \vdots & \vdots \\ d_1+1 & 2d_1+1 \\ n-d_2+d_1+2 \\ \vdots \\ n
\end{matrix}$ appear in positions $(2,c)$ and $(1,c-1)$ respectively.

\item Continue this procedure, that is, the $i$th ($i \in [1,a+c-1]$) page consists of mutations at
\begin{align*}
& (a,1), (a,2), \ldots, (a,b), \\
& (a-1,1), (a-1,2), \ldots, (a-1,b), \\
& \vdots \\
& (i+1,1), (i+1,2), \ldots, (i+1,b), \\
& (i,1), (i,2), \ldots, (i,c), \\
& (i-1,1), (i-1,2), \ldots, (i-1,c-1), \\
& \vdots \\
& (1,1), (1,2), \ldots, (1,c-i+1),
\end{align*}
where the last few lines can be empty if $i>c$. Subsequently the cluster variables 
$$
\begin{matrix}
1 & 1 \\ \vdots & \vdots \\ \vdots & j_2 \\ d_1+j_1+1 & d_1+j_1+j_2+2 \\ n-d_2+d_1+j_1+2 & \vdots \\ \vdots & 2d_1+j_1+1 \\  n
\end{matrix}
$$ 
appear in positions $(j_1+1, c-j_2)$, $j_1+j_2=i-1$, $j_1,j_2 \in \ZZ_{\ge 0}$, $j_2 \le c-1$, $i \in [1, a+c-1]$, respectively.
\end{itemize}

%\begin{minipage}{\textwidth}
By direct counting, we have the following.
\begin{lemma} \label{lem:number of mutation steps for 2 step partial flag varieties}
The number of mutations to obtain the initial seed for $\CC[\flag_{d_1,d_2;n}]$ starting from the initial seed for $\CC[\Gr_{d_2;n+d_2-d_1}]$ is $\frac{ac(c+1)}{2}+\frac{ab(a-1)}{2}$, where $a=d_2-d_1-1$, $b=d_2-2$, $c=d_1$.  
\end{lemma}    
%\end{minipage}

\begin{proof}
The number is equal to 
\begin{align*}
& a^2c + ac(c-a) - \frac{(a-1)a(c-a)}{2} - \frac{(c-a+1)(c-a)a}{2} + \frac{a(a-1)(b-c)}{2} \\
& = \frac{ac(c+1)}{2}+\frac{ab(a-1)}{2}. 
\end{align*} 
\end{proof}

\begin{figure}
\centering
\adjustbox{scale=0.6}{
\begin{tikzcd}
	\bullet & \bullet & \bullet & \bullet & \bullet & \bullet & \bullet & \bullet & \bullet \\
	\bullet & {(1,1)} & {(1,2)} & \cdots & {(1,c)} & {(1,c+1)} & \cdots & {(1,b)} & \bullet \\
	\bullet & {(2,1)} & {(2,2)} & \cdots & {(2,c)} & {(2,c+1)} & \cdots & {(2,b)} & \bullet \\
	\bullet & \vdots & \vdots & \cdots & \vdots & \vdots & \cdots & \vdots & \bullet \\
	\bullet & {(a,1)} & {(a,2)} & \cdots & {(a,c)} & {(a,c+1)} & \cdots & {(a,b)} & \bullet \\
	\bullet & \bullet & \bullet & \cdots & \bullet & \bullet & \cdots & \bullet & \bullet \\
	\vdots & \vdots & \vdots & \cdots & \vdots & \vdots & \cdots & \vdots & \vdots \\
	\bullet & \bullet & \bullet & \cdots & \bullet & \bullet & \cdots & \bullet & \bullet & \bullet
\end{tikzcd}
}
\caption{The mutation region for a two-step partial flag vareity, where bullets have other cluster variables and frozen variables in the initial seed which we do not mutate.}
\label{fig:mutation region for 2 step partial flag varieties}
\end{figure}

\begin{figure}
\centering
\adjustbox{scale=0.6}{
\begin{tikzcd}
	\bullet & \bullet & \bullet & \cdots & \bullet & \bullet & \cdots & \bullet & \bullet \\
	\bullet & \circ & \circ & \cdots & \circ & \circ & \cdots & \circ & \bullet \\
	\vdots & \vdots & \vdots & \cdots & \vdots & \vdots & \cdots & \vdots & \vdots \\
	\bullet & \circ & \circ & \cdots & \circ & \circ & \cdots & \circ & \bullet \\
	\bullet & \bullet & \bullet & \cdots & \bullet & \bullet & \cdots & \bullet & \bullet \\
	\bullet & \circ & \circ & \cdots & \circ & \circ & \cdots & \circ & \bullet \\
	\vdots & \vdots & \vdots & \cdots & \vdots & \vdots & \cdots & \vdots & \vdots \\
	\bullet & \circ & \circ & \cdots & \circ & \circ & \cdots & \circ & \bullet \\
	\bullet & \bullet & \bullet & \cdots & \bullet & \bullet & \cdots & \bullet & \bullet \\
	\vdots & \vdots & \vdots & \cdots & \vdots & \vdots & \cdots & \vdots & \vdots \\
	\bullet & \bullet & \bullet & \cdots & \bullet & \bullet & \cdots & \bullet & \bullet & \bullet
\end{tikzcd}
}
\caption{The mutation region for a general partial flag vareity, where the rectangular regions of circles are the places we perform the mutations, and bullets have other cluster variables and frozen variables in the initial seed which we do not mutate.}
\label{fig:mutation region for general partial flag varieties}
\end{figure}

\subsection{General partial flag varieties}
We describe the mutation sequence for $\flag_{d_1,\ldots,d_k;n}$. In each region: $(d_j-d_{j-1}-1) \times (d_k-2)$ rectangle with left upper corner at the $(d_{j-1}-d_1+2)$nd row, the second column in the initial seed, $j \in [2,k]$, perform the same mutation sequence as in the two-step partial flag varieties in Section \ref{subsec:two-step partial flag varieties}.
%\textcolor{red}{I think here it would be good to have the notion of mutation grid already introduced, we can adapt it to the general case now, and also refer to Figure 7}. 
The new cluster variables which belong to the partial flag variety appears at rectangles $(d_j-d_{j-1}-1) \times d_{j-1}$, with the left upper corner at the $(d_{j-1}-d_1+2)^{\text{th}}$ row, the second column in the initial seed, $j \in [2,k]$. 

Denote the seed obtained as $s'$ with quiver $Q'$. The above exposition shows
\begin{lemma}\label{lem:embed variables}
The images of the initial cluster variables for $\flag_{d_1,\ldots,d_k;n}$ are contained in the set of cluster variables of $s'$. 
\end{lemma}

{\bf The final quiver.} Recall that the desired tableaux are sitting at the vertices of each grid of size $(d_j-d_{j-1}-1) \times d_{j-1}$ with left corner at the $(d_{j-1}-d_1+2)^{\text{th}}$ row and second column. We freeze the cluster variables at positions $(i, d_{j-1}+1)$ in $Q'$, $i \in [d_{j-1}-d_{1}+2, d_j-d_1]$, $j \in [2,k]$. Moreover, we freeze the cluster variables 
\begin{align*}
P_{[1,d_j] \cup [n+1, n+d_k-d_j]}, \quad j \in [1,k-1]. 
\end{align*}
They correspond to frozen variables $P_{[1,d_j]}$ for $\flag_{d_1,\ldots,d_k;n}$. The obtained quiver splits into two mutable subquivers that are connected only through the newly frozen vertices, see for example, Figure~\ref{fig:seed of Gr917 after mutation sequence with labels}, where we have rotated the grid with positions 39, 47, 55, 63 and the grid with positions 20, 21, 28, 29, 36, 37, 44, 45, 52, 53, 60, 61, by $180^{\circ}$ around their centers.

\begin{lemma}
The seed $(Q_{d_1,\ldots,d_k;n},\varphi^*({\bf x}_{d_1,\ldots,d_k;n}))$ is a restricted seed of $s'$. In particular, $\varphi^*$ is an embedding of algebras.
\end{lemma}

\begin{proof}
After freezing the additional vertices the quiver $Q'$ splits into two mutable subquivers: one is of the same type as $\flag_{d_1,\ldots,d_k;n}$ (including frozens). The two subquivers are connected only through the newly frozen vertices.

By Lemma~\ref{lem:embed variables} the images of initial cluster variables in $\flag_{d_1,\ldots,d_k;n}$ are contained in the seed $s'$. Hence, by \cite[\S4.2]{FWZ_ch4-5} the cluster algebra generated by the restricted seed is a subalgebra in $\kk[\Gr_{d_k;N}]$.
\end{proof}

% We freeze and delete the corresponding cluster variables as the two-step partial flag variety case. Moreover, we freeze the cluster variables 
% \begin{align*}
% [1,d_j] \cup [n+1, n+d_k-d_j], \quad j \in [1,k-1]. 
% \end{align*}
% They correspond to frozen variables $[1,d_j]$ for $\flag_{d_1,\ldots,d_k;n}$. Then we obtain the initial seed of the partial flag variety $\flag_{d_1,\ldots,d_k;n}$.

\subsection{\texorpdfstring{Example: $\flag_{4,6,9;12}$}{Example: flag4,6,9;12}} \label{subsec:example flag 4 6 9 12}

In this subsection, we describe the mutation sequence from the initial seed of $\Gr_{9;17}$ to the initial seed of $\flag_{4,6,9;12}$ explicitly. Computations of mutations can be found in: \url{https://github.com/lijr07/Grassmannian-cluster-algebra-and-pathial-flag-varieties}.

In this case, we have two mutation regions. We use $t_{(i,j)}$ (resp. $s_{(i,j)}$) to denote the positions of the vertices of the first (resp. second) mutation region. 

The initial seed for $\Gr_{9;17}$ we choose is shown in Figure \ref{fig:initial seed for Gr917}.
We label the vertices where we will perform mutations by $t_{(i,j)}$ and $s_{(i,j)}$. The mutation sequence is 
\begin{align*}
& t_{(1,1)}, t_{(1,2)}, t_{(1,3)}, t_{(1,4)}, \quad t_{(1,1)}, t_{(1,2)}, t_{(1,3)}, \quad t_{(1,1)}, t_{(1,2)}, \quad t_{(1,1)}, \\
& s_{(2,1)}, s_{(2,2)}, \ldots, s_{(2,7)}, \quad
 s_{(1,1)}, s_{(1,2)}, \ldots, s_{(1,6)}, \quad
 s_{(2,1)}, s_{(2,2)}, \ldots, s_{(2,6)}, \quad
 s_{(1,1)}, s_{(1,2)}, \ldots, s_{(1,5)}, \\ 
& s_{(2,1)}, s_{(2,2)}, \ldots, s_{(2,5)}, \quad
 s_{(1,1)}, s_{(1,2)}, \ldots, s_{(1,4)}, \quad
 s_{(2,1)}, s_{(2,2)}, s_{(2,3)}, s_{(2,4)}, \quad
 s_{(1,1)}, s_{(1,2)}, s_{(1,3)}, \\
& s_{(2,1)}, s_{(2,2)}, s_{(2,3)}, \quad
 s_{(1,1)}, s_{(1,2)}, \quad 
 s_{(2,1)}, s_{(2,2)}, \quad
 s_{(1,1)}, \quad
 s_{(2,1)}. 
\end{align*}
After mutations, we freeze the cluster variables at positions $t_{(1,4)}$, $\ldots$, $t_{(1,8)}$, $s_{(1,6)}$, $s_{(1,7)}$, $s_{(1,8)}$, $s_{(2,6)}$, $s_{(2,7)}$, $s_{(2,8)}$. We also freeze $[1,6]\cup [13,15]$ (the cluster variable $P_{[1,6]\cup [13,15]}$). Then we obtain the initial seed for $\flag_{4,6,9;12}$ in Figure \ref{fig:initial seed for Fl4_6_9_12} after removing all irrelevant frozen variables. 

To see the initial seed of $\flag_{4,6,9;12}$ as a restricted seed of a seed for $\Gr_{9;17}$, we label the initial seed of $\Gr_{9;17}$ as shown in Figure \ref{fig:initial seed of Gr917 with labels}. After performing the above mutations and freezing cluster variables, we obtain the seed in Figure \ref{fig:seed of Gr917 after mutation sequence with labels}. We see that the initial seed of $\flag_{4,6,9;12}$ is a restricted seed for the seed of $\Gr_{9;17}$ in Figure \ref{fig:initial seed of Gr917 with labels}. 

\begin{figure}
    \centering
\includegraphics[width=0.6\textwidth]{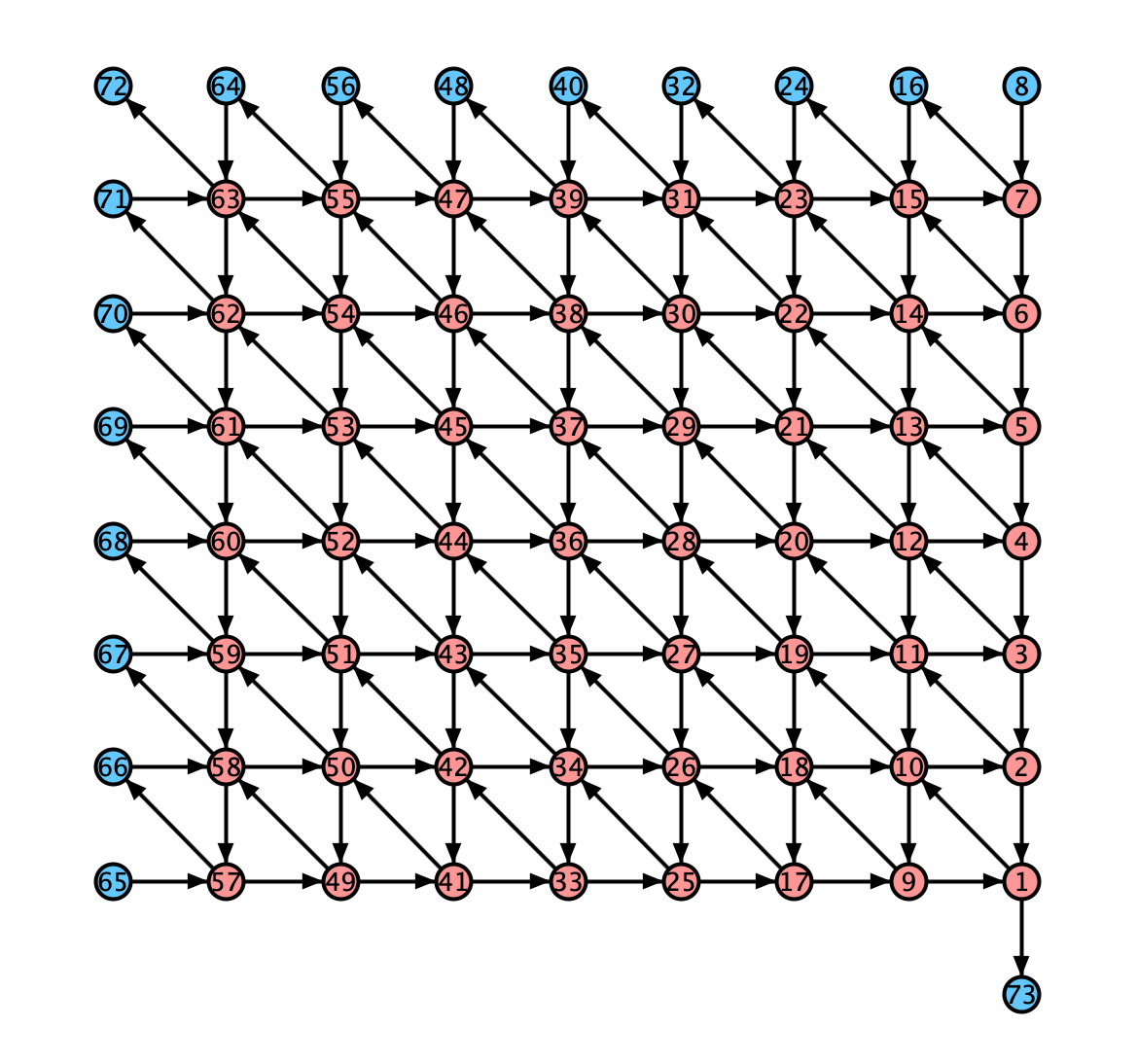} 
    \caption{The initial seed of $\Gr_{9;17}$ with labels.}
    \label{fig:initial seed of Gr917 with labels}
\end{figure}

\begin{figure}
    \centering
\includegraphics[width=0.7\textwidth]{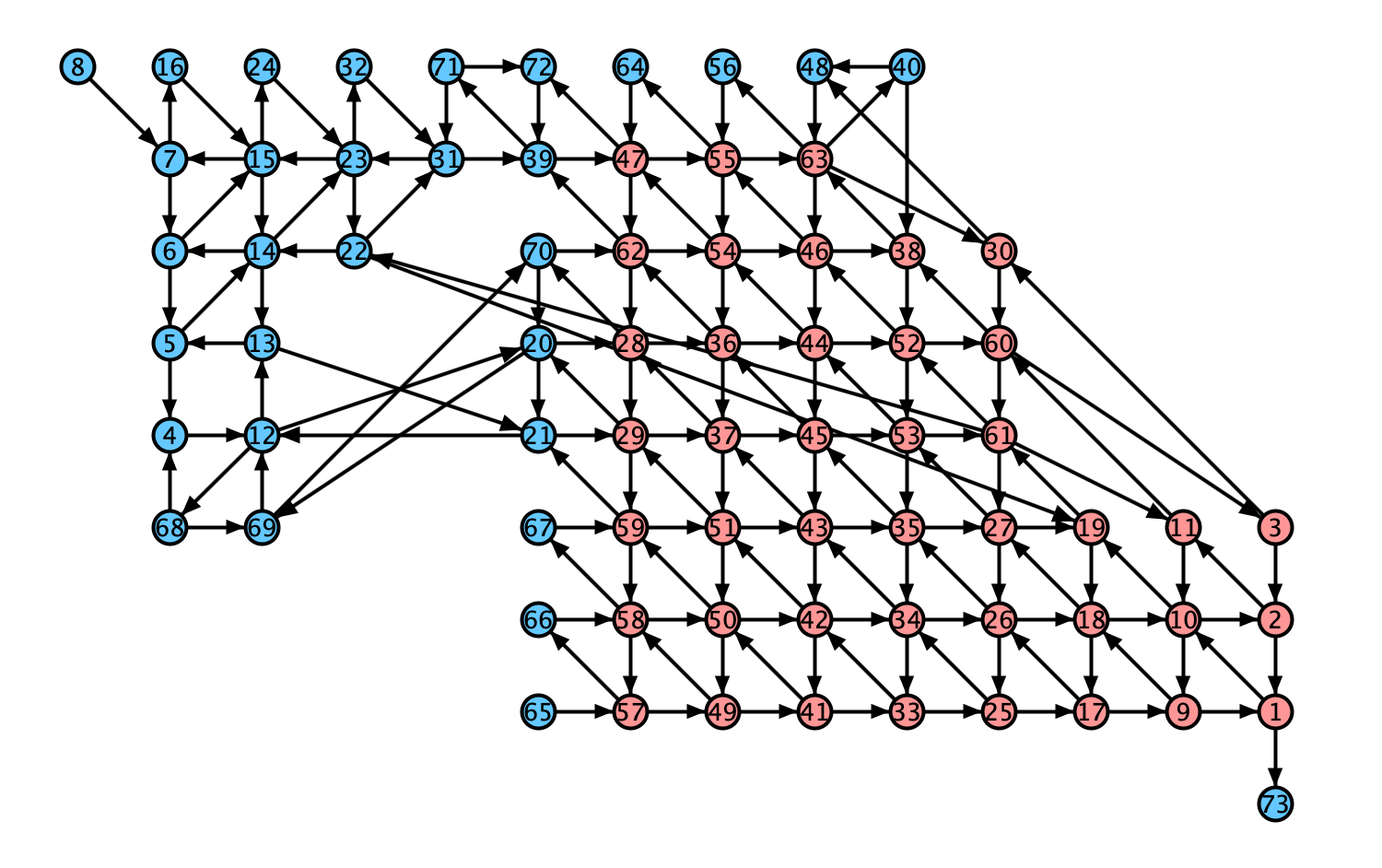} 
    \caption{The seed of $\Gr_{9;17}$ with labels after performing the mutation sequence in \S\ref{subsec:example flag 4 6 9 12}.}
    \label{fig:seed of Gr917 after mutation sequence with labels}
\end{figure}

\begin{figure}
\centering
\adjustbox{scale=0.55}{
 \begin{tikzcd}
	\boxed{[9,17]} & \boxed{[1]\cup [10,17]} & \boxed{[1,2]\cup [11,17]} & \boxed{[1,3]\cup [12,17]} & \boxed{[1,4]\cup [13,17]} & \boxed{[1,5]\cup [14,17]} & \boxed{[1,6]\cup [15,17]} & \boxed{[1,7]\cup [16,17]} & \boxed{[1,8]\cup \{17\}} \\
	\boxed{[8,16]} & \substack{ \textcolor{red}{[1]\cup [9,16]} \\ t_{(1,1)} } & \substack{ \textcolor{red}{[1,2]\cup [10,16]}  \\ t_{(1,2)} } & \substack{ \textcolor{red}{[1,3]\cup [11,16]} \\ t_{(1,3)} } & \substack{ \textcolor{red}{[1,4]\cup [12,16]} \\ t_{(1,4)} } & \substack{ \textcolor{blue}{[1,5]\cup [13,16]} \\ t_{(1,5)} } & \substack{ \textcolor{blue}{[1,6]\cup [14,16]} \\ t_{(1,6)} } & \substack{ \textcolor{blue}{[1,7]\cup [15,16]} \\ t_{(1,7)} } & \substack{  {[1,8]\cup \{16\}} \\ t_{(1,8)} } \\
	\boxed{[7,15]} & {[1]\cup [8,15]} & {[1,2]\cup [9,15]} & {[1,3]\cup [10,15]} & {[1,4]\cup [11,15]} & {[1,5]\cup [12,15]} & {[1,6]\cup [13,15]} & {[1,7]\cup [14,15]} & {[1,8]\cup \{15\}} \\
	\boxed{[6,14]} & \substack{ \textcolor{red}{[1]\cup [7,14]} \\ s_{(1,1)} } & \substack{ \textcolor{red}{[1,2]\cup [8,14]} \\ s_{(1,2)} } & \substack{ \textcolor{red}{[1,3]\cup [9,14]} \\ s_{(1,3)} } & \substack{ \textcolor{red}{[1,4]\cup [10,14]} \\ s_{(1,4)} } & \substack{ \textcolor{red}{[1,5]\cup [11,14]} \\ s_{(1,5)} } & \substack{ \textcolor{red}{[1,6]\cup [12,14]} \\ s_{(1,6)} } & \substack{ \textcolor{blue}{[1,7]\cup [13,14]} \\ s_{(1,7)} } & \substack{ {[1,8]\cup \{14\}} \\ s_{(1,8)} } \\
	\boxed{[5,13]} & \substack{ \textcolor{red}{[1]\cup [6,13]} \\ s_{(2,1)} } & \substack{ \textcolor{red}{[1,2]\cup [7,13]} \\ s_{(2,2)} } & \substack{ \textcolor{red}{[1,3]\cup [8,13]} \\ s_{(2,3)} } & \substack{ \textcolor{red}{[1,4]\cup [9,13]} \\ s_{(2,4)} } & \substack{ \textcolor{red}{[1,5]\cup [10,13]} \\ s_{(2,5)}} & \substack{ \textcolor{red}{[1,6]\cup [11,13]} \\ s_{(2,6)}} & \substack{ \textcolor{blue}{[1,7]\cup [12,13]} \\ s_{(2,7)}} & \substack{ {[1,8]\cup \{13\}} \\ s_{(2,8)} } \\
	\boxed{[4,12]} & {[1]\cup [5,12]} & {[1,2]\cup [6,12]} & {[1,3]\cup [7,12]} & {[1,4]\cup [8,12]} & {[1,5]\cup [9,12]} & {[1,6]\cup [10,12]} & {[1,7]\cup [11,12]} & {[1,8]\cup \{12\}} \\
	\boxed{[3,11]} & {[1]\cup [4,11]} & {[1,2]\cup [5,11]} & {[1,3]\cup [6,11]} & {[1,4]\cup [7,11]} & {[1,5]\cup [8,11]} & {[1,6]\cup [9,11]} & {[1,7]\cup [10,11]} & {[1,8]\cup \{11\}} \\
	\boxed{[2,10]} & {[1]\cup [3,10]} & {[1,2]\cup [4,10]} & {[1,3]\cup [5,10]} & {[1,4]\cup [6,10]} & {[1,5]\cup [7,10]} & {[1,6]\cup [8,10]} & {[1,7]\cup [9,10]} & {[1,8]\cup \{10\}} \\
    & & & & & & & & \boxed{[1,9]}
	\arrow[from=1-2, to=2-2]
	\arrow[from=1-3, to=2-3]
	\arrow[from=1-4, to=2-4]
	\arrow[from=1-5, to=2-5]
	\arrow[from=1-6, to=2-6]
	\arrow[from=1-7, to=2-7]
	\arrow[from=1-8, to=2-8]
	\arrow[from=1-9, to=2-9]
	\arrow[from=2-1, to=2-2]
	\arrow[from=2-2, to=1-1]
	\arrow[from=2-2, to=2-3]
	\arrow[from=2-2, to=3-2]
	\arrow[from=2-3, to=1-2]
	\arrow[from=2-3, to=2-4]
	\arrow[from=2-3, to=3-3]
	\arrow[from=2-4, to=1-3]
	\arrow[from=2-4, to=2-5]
	\arrow[from=2-4, to=3-4]
	\arrow[from=2-5, to=1-4]
	\arrow[from=2-5, to=2-6]
	\arrow[from=2-5, to=3-5]
	\arrow[from=2-6, to=1-5]
	\arrow[from=2-6, to=2-7]
	\arrow[from=2-6, to=3-6]
	\arrow[from=2-7, to=1-6]
	\arrow[from=2-7, to=2-8]
	\arrow[from=2-7, to=3-7]
	\arrow[from=2-8, to=1-7]
	\arrow[from=2-8, to=2-9]
	\arrow[from=2-8, to=3-8]
	\arrow[from=2-9, to=1-8]
	\arrow[from=2-9, to=3-9]
	\arrow[from=3-1, to=3-2]
	\arrow[from=3-2, to=2-1]
	\arrow[from=3-2, to=3-3]
	\arrow[from=3-2, to=4-2]
	\arrow[from=3-3, to=2-2]
	\arrow[from=3-3, to=3-4]
	\arrow[from=3-3, to=4-3]
	\arrow[from=3-4, to=2-3]
	\arrow[from=3-4, to=3-5]
	\arrow[from=3-4, to=4-4]
	\arrow[from=3-5, to=2-4]
	\arrow[from=3-5, to=3-6]
	\arrow[from=3-5, to=4-5]
	\arrow[from=3-6, to=2-5]
	\arrow[from=3-6, to=3-7]
	\arrow[from=3-6, to=4-6]
	\arrow[from=3-7, to=2-6]
	\arrow[from=3-7, to=3-8]
	\arrow[from=3-7, to=4-7]
	\arrow[from=3-8, to=2-7]
	\arrow[from=3-8, to=3-9]
	\arrow[from=3-8, to=4-8]
	\arrow[from=3-9, to=2-8]
	\arrow[from=3-9, to=4-9]
	\arrow[from=4-1, to=4-2]
	\arrow[from=4-2, to=3-1]
	\arrow[from=4-2, to=4-3]
	\arrow[from=4-2, to=5-2]
	\arrow[from=4-3, to=3-2]
	\arrow[from=4-3, to=4-4]
	\arrow[from=4-3, to=5-3]
	\arrow[from=4-4, to=3-3]
	\arrow[from=4-4, to=4-5]
	\arrow[from=4-4, to=5-4]
	\arrow[from=4-5, to=3-4]
	\arrow[from=4-5, to=4-6]
	\arrow[from=4-5, to=5-5]
	\arrow[from=4-6, to=3-5]
	\arrow[from=4-6, to=4-7]
	\arrow[from=4-6, to=5-6]
	\arrow[from=4-7, to=3-6]
	\arrow[from=4-7, to=4-8]
	\arrow[from=4-7, to=5-7]
	\arrow[from=4-8, to=3-7]
	\arrow[from=4-8, to=4-9]
	\arrow[from=4-8, to=5-8]
	\arrow[from=4-9, to=3-8]
	\arrow[from=4-9, to=5-9]
	\arrow[from=5-1, to=5-2]
	\arrow[from=5-2, to=5-3]
	\arrow[from=5-2, to=6-2]
	\arrow[from=5-3, to=4-2]
	\arrow[from=5-3, to=5-4]
	\arrow[from=5-3, to=6-3]
	\arrow[from=5-4, to=4-3]
	\arrow[from=5-4, to=5-5]
	\arrow[from=5-4, to=6-4]
	\arrow[from=5-5, to=4-4]
	\arrow[from=5-5, to=5-6]
	\arrow[from=5-5, to=6-5]
	\arrow[from=5-6, to=4-5]
	\arrow[from=5-6, to=5-7]
	\arrow[from=5-6, to=6-6]
	\arrow[from=5-7, to=4-6]
	\arrow[from=5-7, to=5-8]
	\arrow[from=5-7, to=6-7]
	\arrow[from=5-8, to=4-7]
	\arrow[from=5-8, to=5-9]
	\arrow[from=5-8, to=6-8]
	\arrow[from=5-9, to=4-8]
	\arrow[from=5-9, to=6-9]
	\arrow[from=6-1, to=6-2]
	\arrow[from=6-2, to=5-1]
	\arrow[from=6-2, to=6-3]
	\arrow[from=6-2, to=7-2]
	\arrow[from=6-3, to=5-2]
	\arrow[from=6-3, to=6-4]
	\arrow[from=6-3, to=7-3]
	\arrow[from=6-4, to=5-3]
	\arrow[from=6-4, to=6-5]
	\arrow[from=6-4, to=7-4]
	\arrow[from=6-5, to=5-4]
	\arrow[from=6-5, to=6-6]
	\arrow[from=6-5, to=7-5]
	\arrow[from=6-6, to=5-5]
	\arrow[from=6-6, to=6-7]
	\arrow[from=6-6, to=7-6]
	\arrow[from=6-7, to=5-6]
	\arrow[from=6-7, to=6-8]
	\arrow[from=6-7, to=7-7]
	\arrow[from=6-8, to=5-7]
	\arrow[from=6-8, to=6-9]
	\arrow[from=6-8, to=7-8]
	\arrow[from=6-9, to=5-8]
	\arrow[from=6-9, to=7-9]
	\arrow[from=7-1, to=7-2]
	\arrow[from=7-2, to=6-1]
	\arrow[from=7-2, to=7-3]
	\arrow[from=7-2, to=8-2]
	\arrow[from=7-3, to=6-2]
	\arrow[from=7-3, to=7-4]
	\arrow[from=7-3, to=8-3]
	\arrow[from=7-4, to=6-3]
	\arrow[from=7-4, to=7-5]
	\arrow[from=7-4, to=8-4]
	\arrow[from=7-5, to=6-4]
	\arrow[from=7-5, to=7-6]
	\arrow[from=7-5, to=8-5]
	\arrow[from=7-6, to=6-5]
	\arrow[from=7-6, to=7-7]
	\arrow[from=7-6, to=8-6]
	\arrow[from=7-7, to=6-6]
	\arrow[from=7-7, to=7-8]
	\arrow[from=7-7, to=8-7]
	\arrow[from=7-8, to=6-7]
	\arrow[from=7-8, to=7-9]
	\arrow[from=7-8, to=8-8]
	\arrow[from=7-9, to=6-8]
	\arrow[from=7-9, to=8-9]
	\arrow[from=8-1, to=8-2]
	\arrow[from=8-2, to=7-1]
	\arrow[from=8-2, to=8-3]
	\arrow[from=8-3, to=7-2]
	\arrow[from=8-3, to=8-4]
	\arrow[from=8-4, to=7-3]
	\arrow[from=8-4, to=8-5]
	\arrow[from=8-5, to=7-4]
	\arrow[from=8-5, to=8-6]
	\arrow[from=8-6, to=7-5]
	\arrow[from=8-6, to=8-7]
	\arrow[from=8-7, to=7-6]
	\arrow[from=8-7, to=8-8]
	\arrow[from=8-8, to=7-7]
	\arrow[from=8-8, to=8-9]
	\arrow[from=8-9, to=7-8]
        \arrow[from=8-9, to=9-9] 
\end{tikzcd}
}
\caption{The initial seed for $\kk[\Gr_{9;17}]$.}
\label{fig:initial seed for Gr917}
\end{figure}

\begin{figure}
\centering
\adjustbox{scale=0.38}{
\begin{tikzcd}
	12 &&&&&&& {} \\
	&&&&&& \begin{array}{c} \boxed{\begin{matrix} 1 \\ 2 \\ 3 \\12 \end{matrix}} \end{array} \\
	11 &&&&& {} \\
	&&&& \begin{array}{c} \boxed{\begin{matrix} 1 \\ 2 \\ 11 \\12 \end{matrix}} \end{array} \\
	10 &&& {} \\
	&& \begin{array}{c} \boxed{\begin{matrix} 1 \\ 10 \\ 11 \\12 \end{matrix}} \end{array} \\
	9 & {} \\
	\begin{array}{c} \boxed{\begin{matrix} 9 \\ 10 \\ 11 \\12 \end{matrix}} \end{array} &&&&&&&& \begin{array}{c} \boxed{\begin{matrix} 1 \\ 2 \\ 3 \\4 \end{matrix}} \end{array} \\
	8 &&&&&&&&&&& {} \\
	\begin{array}{c} \boxed{\textcolor{red}{\begin{matrix} 1&8 \\ 2&9 \\ 3&10 \\4&11 \\ 5 \\ 12 \end{matrix}}} \end{array} && \begin{array}{c} \textcolor{red}{\begin{matrix} 1&1 \\ 2&9 \\ 3&10 \\4&11 \\ 5 \\ 12 \end{matrix}} \end{array} && \begin{array}{c} \textcolor{red}{\begin{matrix} 1&1 \\ 2&2 \\ 3&10 \\4&11 \\ 5 \\ 12 \end{matrix}} \end{array} && \begin{array}{c} \textcolor{red}{\begin{matrix} 1&1 \\ 2&2 \\ 3&3 \\4&11 \\ 5 \\ 12 \end{matrix}} \end{array} &&&& \begin{array}{c} \begin{matrix} 1 \\ 2 \\ 3 \\4 \\ 5 \\ 12 \end{matrix} \end{array} \\
	7 &&&&&&&&& {} \\
	\begin{array}{c} \boxed{\begin{matrix} 7 \\ 8 \\ 9 \\10 \\ 11 \\ 12 \end{matrix}} \end{array} && \begin{array}{c} \begin{matrix} 1 \\ 8 \\ 9 \\10 \\ 11 \\ 12 \end{matrix} \end{array} && \begin{array}{c} \begin{matrix} 1 \\ 2 \\ 9 \\10 \\ 11 \\ 12 \end{matrix} \end{array} && \begin{array}{c} \begin{matrix} 1 \\ 2 \\ 3 \\10 \\ 11 \\ 12 \end{matrix} \end{array} && \begin{array}{c} \begin{matrix} 1 \\ 2 \\ 3 \\4 \\ 11 \\ 12 \end{matrix} \end{array} &&&& \begin{array}{c} \boxed{\begin{matrix} 1 \\ \vdots \\ 6 \end{matrix}} \end{array} \\
	6 &&&&&&&&&&&&&&&&& {} \\
	\begin{array}{c} \boxed{\textcolor{red}{\begin{matrix} 1&6 \\ \vdots & 7 \\ \vdots & 8 \\ \vdots &  \vdots \\ \vdots & 11 \\ 8 \\ 12 \end{matrix}}} \end{array} && \begin{array}{c} \textcolor{red}{\begin{matrix} 1&1 \\ \vdots & 7 \\ \vdots & 8 \\ \vdots &  \vdots \\ \vdots & 11 \\ 8 \\ 12 \end{matrix}} \end{array} && \begin{array}{c} \textcolor{red}{\begin{matrix} 1&1 \\ \vdots & 2 \\ \vdots & 8 \\ \vdots &  \vdots \\ \vdots & 11 \\ 8 \\ 12 \end{matrix}} \end{array} && \begin{array}{c} \textcolor{red}{\begin{matrix} 1&1 \\ \vdots & 2 \\ \vdots & 3 \\ \vdots &  9 \\ \vdots & 10 \\ 8 & 11 \\ 12 \end{matrix}} \end{array} && \begin{array}{c} \textcolor{red}{\begin{matrix} 1&1 \\ \vdots & \vdots \\ \vdots & 4 \\ \vdots & 10 \\ \vdots & 11 \\ 8 \\ 12 \end{matrix}} \end{array} && \begin{array}{c} \textcolor{red}{\begin{matrix} 1&1 \\ \vdots & \vdots \\ \vdots & 4 \\ \vdots & 5 \\ \vdots & 11 \\ 8 \\ 12 \end{matrix}} \end{array} &&&&&& \begin{array}{c} \begin{matrix} 1 \\ 2 \\ \vdots \\ 6 \\ 7 \\ 8 \\ 12 \end{matrix} \end{array} \\
	5 &&&&&&&&&&&&&&& {} \\
	\begin{array}{c} \boxed{\textcolor{red}{\begin{matrix} 1&5 \\ \vdots & 6 \\ \vdots & 7 \\ \vdots &  \vdots \\ 7 & 10 \\ 11 \\ 12 \end{matrix}}} \end{array} && \begin{array}{c} \textcolor{red}{\begin{matrix} 1&1 \\ \vdots & 6 \\ \vdots & 7 \\ \vdots &  \vdots \\ 7 & 10 \\ 11 \\ 12 \end{matrix}} \end{array} && \begin{array}{c} \textcolor{red}{\begin{matrix} 1&1 \\ \vdots & 2 \\ \vdots & 7 \\ \vdots &  \vdots \\ 7 & 10 \\ 11 \\ 12 \end{matrix}} \end{array} && \begin{array}{c} \textcolor{red}{\begin{matrix} 1&1 \\ \vdots & 2 \\ \vdots & 3 \\ \vdots &  8 \\ 7 & 9 \\ 11 & 10 \\ 12 \end{matrix}} \end{array} && \begin{array}{c} \textcolor{red}{\begin{matrix} 1&1 \\ \vdots & \vdots \\ \vdots & 4 \\ \vdots & 9 \\ 7 & 10 \\ 11 \\ 12 \end{matrix}} \end{array} && \begin{array}{c} \textcolor{red}{\begin{matrix} 1&1 \\ \vdots & \vdots \\ \vdots & 4 \\ \vdots & 5 \\ 7 & 10 \\ 11 \\ 12 \end{matrix}} \end{array} &&&& \begin{array}{c} \begin{matrix} 1 \\ 2 \\ \vdots \\ 7 \\ 11 \\ 12 \end{matrix} \end{array} \\
	4 &&&&&&&&&&&&& {} \\
	\begin{array}{c} \boxed{\begin{matrix} 4 \\ 5 \\ 6 \\ \vdots \\ 11 \\12 \end{matrix}} \end{array} && \begin{array}{c} \begin{matrix} 1 \\ 5 \\ 6 \\ 7 \\ \vdots \\ 12 \end{matrix} \end{array} && \begin{array}{c} \begin{matrix} 1 \\ 2 \\ 6 \\ 7 \\ \vdots \\ 12 \end{matrix} \end{array} && \begin{array}{c} \begin{matrix} 1 \\ \vdots \\ 3 \\ 7 \\ \vdots \\ 12 \end{matrix} \end{array} && \begin{array}{c} \begin{matrix} 1 \\ \vdots \\ 4 \\ 8 \\ \vdots \\ 12 \end{matrix} \end{array} && \begin{array}{c} \begin{matrix} 1 \\ \vdots \\ 5 \\ 9 \\ \vdots \\ 12 \end{matrix} \end{array} && \begin{array}{c} \begin{matrix} 1 \\ \vdots \\ 6 \\ 10 \\ 11 \\ 12 \end{matrix} \end{array} \\
	3 &&&&&&&&&&&&&&&&&&&&&&& {} \\
	\begin{array}{c} \boxed{\begin{matrix} 3 \\ 4 \\ 5 \\ \vdots \\ 10 \\11 \end{matrix}} \end{array} && \begin{array}{c} \begin{matrix} 1 \\ 4 \\ 5 \\ \vdots \\ 10 \\11 \end{matrix} \end{array} && \begin{array}{c} \begin{matrix} 1 \\ 2 \\ 5 \\ \vdots \\ 10 \\11 \end{matrix} \end{array} && \begin{array}{c} \begin{matrix} 1 \\ 2 \\ 3 \\ 6 \\ \vdots \\11 \end{matrix} \end{array} && \begin{array}{c} \begin{matrix} 1 \\ \vdots \\ 4 \\ 7 \\ \vdots \\11 \end{matrix} \end{array} && \begin{array}{c} \begin{matrix} 1 \\ \vdots \\ 5 \\ 8 \\ \vdots \\11 \end{matrix} \end{array} && \begin{array}{c} \begin{matrix} 1 \\ \vdots \\ 6 \\ 9 \\ 10 \\11 \end{matrix} \end{array} && \begin{array}{c} \begin{matrix} 1 \\ 2 \\ \vdots \\ 7 \\ 10 \\11 \end{matrix} \end{array} && \begin{array}{c} \begin{matrix} 1 \\ 2 \\ \vdots \\ 7 \\ 8 \\11 \end{matrix} \end{array} \\
	2 &&&&&&&&&&&&&&&&&&&&& {} \\
	\begin{array}{c} \boxed{\begin{matrix} 2 \\ 3 \\ 4 \\ \vdots \\ 9 \\10 \end{matrix}} \end{array} && \begin{array}{c} \begin{matrix} 1 \\ 3 \\ 4 \\ \vdots \\ 9 \\10 \end{matrix} \end{array} && \begin{array}{c} \begin{matrix} 1 \\ 2 \\ 4 \\ \vdots \\ 9 \\10 \end{matrix} \end{array} && \begin{array}{c} \begin{matrix} 1 \\ 2 \\ 3 \\ 5 \\ \vdots \\10 \end{matrix} \end{array} && \begin{array}{c} \begin{matrix} 1 \\ \vdots \\ 4 \\ 6 \\ \vdots \\10 \end{matrix} \end{array} && \begin{array}{c} \begin{matrix} 1 \\ \vdots \\ 5 \\ 7 \\ \vdots \\10 \end{matrix} \end{array} && \begin{array}{c} \begin{matrix} 1 \\ \vdots \\ 6 \\ 8 \\ 9 \\10 \end{matrix} \end{array} && \begin{array}{c} \begin{matrix} 1 \\ 2 \\ \vdots \\ 7 \\ 9 \\10 \end{matrix} \end{array} && \begin{array}{c} \begin{matrix} 1 \\ 2 \\ \vdots \\ 7 \\ 8 \\10 \end{matrix} \end{array} &&&& \begin{array}{c} \boxed{\begin{matrix} 1 \\ 2 \\ 3 \\ \vdots \\ 8 \\9 \end{matrix}} \end{array} \\
	1 &&&&&&&&&&&&&&&&&&& {} \\
	\\
	& 1 && 2 && 3 && 4 && 5 && 6 && 7 && 8 && 9 && 10 && 11 && 12
	\arrow[no head, from=1-1, to=1-8, teal, rounded corners, to path=-| (\tikztotarget)]
	\arrow[no head, from=1-8, to=25-8, teal, rounded corners, to path=-| (\tikztotarget)]
	\arrow[from=2-7, to=10-7]
	\arrow[no head, from=3-1, to=3-6, teal, rounded corners, to path=-| (\tikztotarget)]
	\arrow[no head, from=3-6, to=25-6, teal, rounded corners, to path=-| (\tikztotarget)]
	\arrow[from=4-5, to=10-5]
	\arrow[no head, from=5-1, to=5-4, teal, rounded corners, to path=-| (\tikztotarget)]
	\arrow[no head, from=5-4, to=25-4, teal, rounded corners, to path=-| (\tikztotarget)]
	\arrow[from=6-3, to=10-3]
	\arrow[no head, from=7-1, to=7-2, teal, rounded corners, to path=-| (\tikztotarget)]
	\arrow[no head, from=7-2, to=25-2, teal, rounded corners, to path=-| (\tikztotarget)]
	\arrow[from=8-9, to=12-9]
	\arrow[no head, from=9-1, to=9-12, teal, rounded corners, to path=-| (\tikztotarget)]
	\arrow[no head, from=9-12, to=25-12, teal, rounded corners, to path=-| (\tikztotarget)]
	\arrow[from=10-1, to=10-3]
	\arrow[from=10-3, to=8-1]
	\arrow[from=10-3, to=10-5]
	\arrow[from=10-3, to=12-3]
	\arrow[from=10-5, to=6-3]
	\arrow[from=10-5, to=10-7]
	\arrow[from=10-5, to=12-5]
	\arrow[from=10-7, to=4-5]
	\arrow[from=10-7, to=8-9]
	\arrow[from=10-7, to=10-11]
	\arrow[from=10-7, to=12-7]
	\arrow[from=10-11, to=2-7]
	\arrow[from=10-11, to=14-11]
	\arrow[no head, from=11-1, to=11-10, teal, rounded corners, to path=-| (\tikztotarget)]
	\arrow[no head, from=11-10, to=25-10, teal, rounded corners, to path=-| (\tikztotarget)]
	\arrow[from=12-1, to=12-3]
	\arrow[from=12-3, to=10-1]
	\arrow[from=12-3, to=12-5]
	\arrow[from=12-3, to=14-3]
	\arrow[from=12-5, to=10-3]
	\arrow[from=12-5, to=12-7]
	\arrow[from=12-5, to=14-5]
	\arrow[from=12-7, to=10-5]
	\arrow[from=12-7, to=12-9]
	\arrow[from=12-7, to=14-7]
	\arrow[from=12-9, to=10-7]
	\arrow[from=12-9, to=14-9]
	\arrow[from=12-13, to=18-13]
	\arrow[no head, from=13-1, to=13-18, teal, rounded corners, to path=-| (\tikztotarget)]
	\arrow[no head, from=13-18, to=25-18, teal, rounded corners, to path=-| (\tikztotarget)]
	\arrow[from=14-1, to=14-3]
	\arrow[from=14-3, to=12-1]
	\arrow[from=14-3, to=14-5]
	\arrow[from=14-3, to=16-3]
	\arrow[from=14-5, to=12-3]
	\arrow[from=14-5, to=14-7]
	\arrow[from=14-5, to=16-5]
	\arrow[from=14-7, to=12-5]
	\arrow[from=14-7, to=14-9]
	\arrow[from=14-7, to=16-7]
	\arrow[from=14-9, to=12-7]
	\arrow[from=14-9, to=14-11]
	\arrow[from=14-9, to=16-9]
	\arrow[from=14-11, to=12-9]
	\arrow[from=14-11, to=14-17]
	\arrow[from=14-11, to=16-11]
	\arrow[from=14-17, to=10-11]
	\arrow[from=14-17, to=20-17]
	\arrow[no head, from=15-1, to=15-16, teal, rounded corners, to path=-| (\tikztotarget)]
	\arrow[no head, from=15-16, to=25-16, teal, rounded corners, to path=-| (\tikztotarget)]
	\arrow[from=16-1, to=16-3]
	\arrow[from=16-3, to=16-5]
	\arrow[from=16-3, to=18-3]
	\arrow[from=16-5, to=14-3]
	\arrow[from=16-5, to=16-7]
	\arrow[from=16-5, to=18-5]
	\arrow[from=16-7, to=14-5]
	\arrow[from=16-7, to=16-9]
	\arrow[from=16-7, to=18-7]
	\arrow[from=16-9, to=14-7]
	\arrow[from=16-9, to=16-11]
	\arrow[from=16-9, to=18-9]
	\arrow[from=16-11, to=12-13]
	\arrow[from=16-11, to=14-9]
	\arrow[from=16-11, to=16-15]
	\arrow[from=16-11, to=18-11]
	\arrow[from=16-15, to=14-11]
	\arrow[from=16-15, to=20-15]
	\arrow[no head, from=17-1, to=17-14, teal, rounded corners, to path=-| (\tikztotarget)]
	\arrow[no head, from=17-14, to=25-14, teal, rounded corners, to path=-| (\tikztotarget)]
	\arrow[from=18-1, to=18-3]
	\arrow[from=18-3, to=16-1]
	\arrow[from=18-3, to=18-5]
	\arrow[from=18-3, to=20-3]
	\arrow[from=18-5, to=16-3]
	\arrow[from=18-5, to=18-7]
	\arrow[from=18-5, to=20-5]
	\arrow[from=18-7, to=16-5]
	\arrow[from=18-7, to=18-9]
	\arrow[from=18-7, to=20-7]
	\arrow[from=18-9, to=16-7]
	\arrow[from=18-9, to=18-11]
	\arrow[from=18-9, to=20-9]
	\arrow[from=18-11, to=16-9]
	\arrow[from=18-11, to=18-13]
	\arrow[from=18-11, to=20-11]
	\arrow[from=18-13, to=16-11]
	\arrow[from=18-13, to=20-13]
	\arrow[no head, from=19-1, to=19-24, teal, rounded corners, to path=-| (\tikztotarget)]
	\arrow[no head, from=19-24, to=25-24, teal, rounded corners, to path=-| (\tikztotarget)]
	\arrow[from=20-1, to=20-3]
	\arrow[from=20-3, to=18-1]
	\arrow[from=20-3, to=20-5]
	\arrow[from=20-3, to=22-3]
	\arrow[from=20-5, to=18-3]
	\arrow[from=20-5, to=20-7]
	\arrow[from=20-5, to=22-5]
	\arrow[from=20-7, to=18-5]
	\arrow[from=20-7, to=20-9]
	\arrow[from=20-7, to=22-7]
	\arrow[from=20-9, to=18-7]
	\arrow[from=20-9, to=20-11]
	\arrow[from=20-9, to=22-9]
	\arrow[from=20-11, to=18-9]
	\arrow[from=20-11, to=20-13]
	\arrow[from=20-11, to=22-11]
	\arrow[from=20-13, to=18-11]
	\arrow[from=20-13, to=20-15]
	\arrow[from=20-13, to=22-13]
	\arrow[from=20-15, to=18-13]
	\arrow[from=20-15, to=20-17]
	\arrow[from=20-15, to=22-15]
	\arrow[from=20-17, to=16-15]
	\arrow[from=20-17, to=22-17]
	\arrow[no head, from=21-1, to=21-22, teal, rounded corners, to path=-| (\tikztotarget)]
	\arrow[no head, from=21-22, to=25-22, teal, rounded corners, to path=-| (\tikztotarget)]
	\arrow[from=22-1, to=22-3]
	\arrow[from=22-3, to=20-1]
	\arrow[from=22-3, to=22-5]
	\arrow[from=22-5, to=20-3]
	\arrow[from=22-5, to=22-7]
	\arrow[from=22-7, to=20-5]
	\arrow[from=22-7, to=22-9]
	\arrow[from=22-9, to=20-7]
	\arrow[from=22-9, to=22-11]
	\arrow[from=22-11, to=20-9]
	\arrow[from=22-11, to=22-13]
	\arrow[from=22-13, to=20-11]
	\arrow[from=22-13, to=22-15]
	\arrow[from=22-15, to=20-13]
	\arrow[from=22-15, to=22-17]
	\arrow[from=22-17, to=20-15]
	\arrow[from=22-17, to=22-21]
	\arrow[no head, from=23-1, to=23-20, teal, rounded corners, to path=-| (\tikztotarget)]
	\arrow[no head, from=23-20, to=25-20, teal, rounded corners, to path=-| (\tikztotarget)]
\end{tikzcd}
}
\caption{The initial seed for $\flag_{4,6,9;12}$.}
\label{fig:initial seed for Fl4_6_9_12}
\end{figure}

\section{\texorpdfstring{Application: cluster structures on $\mathcal{MT}_n$ and $\mathcal{SH}_n$ from Grassmannians}{Application: cluster structures on MT\_n and SH\_n from Grassmannians}}
\label{sec:Cluster structures on MTn and SHn from Grassmannians}

In this section, as special cases of general partial flag varieties in \S\ref{sec:mutation sequence}, we focus on $\flag_{2,4;n}$, respectively $\flag_{2,n-2;n}$, which are of particular interest for the momentum twistor variety $\mathcal{MT}_n$, respectively for the spinor helicity variety $\mathcal{SH}_n$. We exhibit an explicit mutation sequence from the initial seed of the Grassmannian $\Gr_{4;n+2}$, respectively $\Gr_{n-2;2n-4}$, to a seed which contains $(Q_{2,4;n},\varphi^*({\bf x}_{2,4;n}))$, respectively $(Q_{2,n-2;n},\varphi^*({\bf x}_{2,n-2;n}))$, as a restricted seed. 
 
\subsection{Momentum twistor varieties}

\begin{figure}
    \centering
\adjustbox{scale=0.5}{
\begin{tikzcd}
	\begin{array}{c} \boxed{\begin{matrix} n-1 \\ n \\ n+1 \\ n+2 \end{matrix}} (16) \end{array} & \begin{array}{c} \boxed{\begin{matrix} 1 \\ n \\ n+1 \\ n+2 \end{matrix}} (12) \end{array} & \begin{array}{c} \boxed{\begin{matrix} 1 \\ 2 \\ n+1 \\ n+2 \end{matrix}}(8) \end{array} & \begin{array}{c} \boxed{\begin{matrix} 1 \\ 2 \\ 3 \\ n+2 \end{matrix}} (4) \end{array} \\
	\begin{array}{c} \boxed{\begin{matrix} 1 \\ n-1 \\ n \\ n+1 \end{matrix}} (15) \end{array} & \begin{array}{c} \textcolor{red}{\begin{matrix} 1 \\ n-1 \\ n \\ n+1 \end{matrix}} (11) \end{array} & \begin{array}{c} \textcolor{red}{\begin{matrix} 1 \\ 2 \\ n \\ n+1 \end{matrix}} (7) \end{array} & \begin{array}{c} \begin{matrix} 1 \\ 2 \\ 3 \\ n+1 \end{matrix} (3) \end{array} \\
	\begin{array}{c} \boxed{\begin{matrix} n-3 \\ n-2 \\ n-1 \\ n \end{matrix}} (14) \end{array} & \begin{array}{c} \begin{matrix} 1 \\ n-2 \\ n-1 \\ n \end{matrix} (10) \end{array} & \begin{array}{c} \begin{matrix} 1 \\ 2 \\ n-1 \\ n \end{matrix} (6) \end{array} & \begin{array}{c} \begin{matrix} 1 \\ 2 \\ 3 \\ n \end{matrix} (2) \end{array} \\
	\begin{array}{c} \boxed{\begin{matrix} n-4 \\ n-3 \\ n-2 \\ n-1 \end{matrix}} \end{array} & \begin{array}{c} \begin{matrix} 1 \\ n-3 \\ n-2 \\ n-1 \end{matrix} \end{array} & \begin{array}{c} \begin{matrix} 1 \\ 2 \\ n-2 \\ n-1 \end{matrix} \end{array} & \begin{array}{c} \begin{matrix} 1 \\ 2 \\ 3 \\ n-1 \end{matrix} \end{array} \\
	\vdots & \vdots & \vdots & \vdots \\
	\begin{array}{c} \boxed{\begin{matrix} 2 \\ 3 \\ 4 \\ 5 \end{matrix}} (13) \end{array} & \begin{array}{c} \begin{matrix} 1 \\ 3\\4\\5 \end{matrix} (9) \end{array} & \begin{array}{c} \begin{matrix} 1 \\ 2\\4\\5 \end{matrix} (5) \end{array} & \begin{array}{c} \begin{matrix} 1 \\ 2\\3\\5 \end{matrix} (1) \end{array} & \begin{array}{c} \boxed{\begin{matrix} 1 \\ 2 \\ 3 \\ 4 \end{matrix}} (17) \end{array}
	\arrow[from=1-2, to=2-2]
	\arrow[from=2-1, to=2-2]
	\arrow[from=2-2, to=1-1]
	\arrow[from=2-2, to=2-3]
	\arrow[from=2-2, to=3-2]
	\arrow[from=2-3, to=1-2]
	\arrow[from=2-3, to=2-4]
	\arrow[from=2-3, to=3-3]
	\arrow[from=2-4, to=1-3]
	\arrow[from=2-4, to=3-4]
	\arrow[from=3-1, to=3-2]
	\arrow[from=3-2, to=2-1]
	\arrow[from=3-2, to=3-3]
	\arrow[from=3-2, to=4-2]
	\arrow[from=3-3, to=2-2]
	\arrow[from=3-3, to=3-4]
	\arrow[from=3-3, to=4-3]
	\arrow[from=3-4, to=2-3]
	\arrow[from=3-4, to=4-4]
	\arrow[from=4-1, to=4-2]
	\arrow[from=4-1, to=5-1]
	\arrow[from=4-2, to=3-1]
	\arrow[from=4-2, to=4-3]
	\arrow[from=4-2, to=5-2]
	\arrow[from=4-3, to=3-2]
	\arrow[from=4-3, to=4-4]
	\arrow[from=4-3, to=5-3]
	\arrow[from=4-4, to=3-3]
	\arrow[from=4-4, to=5-4]
	\arrow[from=5-1, to=5-2]
	\arrow[from=5-2, to=4-1]
	\arrow[from=5-2, to=5-3]
	\arrow[from=5-2, to=6-2]
	\arrow[from=5-3, to=4-2]
	\arrow[from=5-3, to=5-4]
	\arrow[from=5-3, to=6-3]
	\arrow[from=5-4, to=4-3]
	\arrow[from=5-4, to=6-4]
	\arrow[from=6-1, to=6-2]
	\arrow[from=6-2, to=5-1]
	\arrow[from=6-2, to=6-3]
	\arrow[from=6-3, to=5-2]
	\arrow[from=6-3, to=6-4]
	\arrow[from=6-4, to=5-3]
	\arrow[from=6-4, to=6-5]
\end{tikzcd}
\begin{tikzcd}
	\begin{array}{c} \boxed{\begin{matrix} n-2 \\ n-1 \\ n \\ n+1 \end{matrix}} \end{array} (15) & \begin{array}{c} \boxed{\begin{matrix} n-1 \\ n \end{matrix}} \end{array} (16) & \begin{array}{c} \boxed{\begin{matrix} 1 \\ n \end{matrix}} \end{array} (17) & \begin{array}{c} \boxed{\begin{matrix} 1 \\ 2 \end{matrix}} \end{array} (18) \\
	\begin{array}{c} \boxed{\begin{matrix} 1 \\ 2 \\ 3 \\ n+1 \end{matrix}} (3) \end{array} & \begin{array}{c} \boxed{\textcolor{red}{\begin{matrix} 1 & n-2 \\ 2 & n-1 \\ 3 \\ n \end{matrix}}} \end{array} (7)  & \begin{array}{c} \textcolor{red}{\begin{matrix} 1 & 1 \\ 2 & n-1 \\ 3 \\ n \end{matrix}} \end{array} (11) && \begin{array}{c} \begin{matrix} 1  \\ 2 \\ 3 \\ n \end{matrix} \end{array} (2) \\
	\begin{array}{c} \boxed{\begin{matrix} 1 \\ 2 \\ 3 \\ n+2 \end{matrix}} (4) \end{array} & \begin{array}{c} \boxed{\begin{matrix} n-3 \\ n-2 \\ n-1 \\ n \end{matrix}} \end{array} (14) & \begin{array}{c} \begin{matrix} 1 \\ n-2 \\ n-1 \\ n \end{matrix} \end{array} (10) & \begin{array}{c} \begin{matrix} 1 \\ 2 \\ n-1 \\ n \end{matrix} \end{array} (6) \\
	& \begin{array}{c} \boxed{\begin{matrix} n-4 \\ n-3 \\ n-2 \\ n-1 \end{matrix}} \end{array} & \begin{array}{c} \begin{matrix} 1 \\ n-3 \\ n-2 \\ n-1 \end{matrix} \end{array} & \begin{array}{c} \begin{matrix} 1 \\ 2 \\ n-2 \\ n-1 \end{matrix} \end{array} & \begin{array}{c} \begin{matrix} 1 \\ 2 \\ 3 \\ n-1 \end{matrix} \end{array} \\
	& \vdots & \vdots & \vdots & \vdots \\
	& \begin{array}{c} \boxed{\begin{matrix} 2\\3\\4\\5 \end{matrix}} \end{array} (13) & \begin{array}{c} \begin{matrix} 1\\3\\4\\5 \end{matrix} \end{array} (9) & \begin{array}{c} \begin{matrix} 1\\2\\4\\5 \end{matrix} \end{array} (5) & \begin{array}{c} \begin{matrix} 1\\2\\3\\5 \end{matrix} \end{array} (1) & \begin{array}{c} \boxed{\begin{matrix} 1 \\ 2 \\ 3 \\ 4 \end{matrix}} \end{array} (17)
	\arrow[from=1-1, to=1-2]
	\arrow[from=1-1, to=2-1]
	\arrow[from=1-2, to=2-2]
	\arrow[from=1-3, to=2-3]
	\arrow[from=1-4, to=1-3]
	\arrow[from=1-4, to=3-4]
	\arrow[from=2-1, to=2-2]
	\arrow[from=2-2, to=1-1]
	\arrow[from=2-2, to=2-3]
	\arrow[from=2-3, to=1-2]
	\arrow[from=2-3, to=1-4]
	\arrow[from=2-3, to=2-5]
	\arrow[from=2-3, to=3-3]
	\arrow[from=2-5, to=1-3]
	\arrow[from=2-5, to=4-5]
	\arrow[from=3-1, to=2-1]
	\arrow[from=3-2, to=3-3]
	\arrow[from=3-2, to=4-2]
	\arrow[from=3-3, to=2-2]
	\arrow[from=3-3, to=3-4]
	\arrow[from=3-3, to=4-3]
	\arrow[from=3-4, to=2-3]
	\arrow[from=3-4, to=4-4]
	\arrow[from=4-2, to=4-3]
	\arrow[from=4-2, to=5-2]
	\arrow[from=4-3, to=3-2]
	\arrow[from=4-3, to=4-4]
	\arrow[from=4-3, to=5-3]
	\arrow[from=4-4, to=3-3]
	\arrow[from=4-4, to=4-5]
	\arrow[from=4-4, to=5-4]
	\arrow[from=4-5, to=3-4]
	\arrow[from=4-5, to=5-5]
	\arrow[from=5-2, to=5-3]
	\arrow[from=5-2, to=6-2]
	\arrow[from=5-3, to=4-2]
	\arrow[from=5-3, to=5-4]
	\arrow[from=5-3, to=6-3]
	\arrow[from=5-4, to=4-3]
	\arrow[from=5-4, to=5-5]
	\arrow[from=5-4, to=6-4]
	\arrow[from=5-5, to=4-4]
	\arrow[from=5-5, to=6-5]
	\arrow[from=6-2, to=6-3]
	\arrow[from=6-3, to=5-2]
	\arrow[from=6-3, to=6-4]
	\arrow[from=6-4, to=5-3]
	\arrow[from=6-4, to=6-5]
	\arrow[from=6-5, to=5-4]
	\arrow[from=6-5, to=6-6]
\end{tikzcd}
}
    \caption{LHS: The initial seed for Grassmannian $\Gr_{4;n+2}$.
    RHS: a seed obtained by mutating (11), (7), (11) of the seed on LHS and freeze (7). The full subquiver on RHS on all vertices but (3), (4), (15) coincides with the initial seed for $\flag_{2,4;n}$.}
    \label{fig:initial seed for F24n-G4n+2}
\end{figure}

The initial seeds for $\flag_{2,4;n}$ and $\Gr_{4;n+2}$ are depicted in Figure~\ref{fig:initial seed for F24n-G4n+2}. 
Notice that out of the $4n-7$ initial cluster variables of $\Gr_{4;n+2}$ all but five coincide with (images of) initial cluster variables of $\flag_{2,4;n}$.
Hence, we only need to recover the two missing initial cluster variables, respectively their images under \eqref{eq:embed}:
    \[
    \boxed{\begin{smallmatrix}
        1&n-2\\2&n-1\\3&n+1\\n&n+2
    \end{smallmatrix}} \quad \text{and}\quad
    \begin{smallmatrix}
        1&1\\2&n-1\\3&n+1\\n&n+2
    \end{smallmatrix}.
    \]
   The mutation sequence (11), (7), (11) yields the missing cluster variables, see Figure~\ref{fig:initial seed for F24n-G4n+2}. 
    This is followed by freezing 
        \[
    {\begin{smallmatrix}
        1&n-2\\2&n-1\\3&n+1\\n&n+2
    \end{smallmatrix}} 
    \]
    which is the mutable vertex (7) in the Grassmannian. Now all tableaux but
    \[
    \boxed{\begin{smallmatrix}
        n-2\\n-1\\n\\n+1
    \end{smallmatrix}}, \quad 
    \boxed{\begin{smallmatrix}
        1\\2\\3\\n+2
    \end{smallmatrix}}
    \quad \text{and}\quad
    \begin{smallmatrix}
        1\\2\\3\\n+1
    \end{smallmatrix}.
    \]
    coincide with the images of the initial tableaux for $\flag_{2,4;n}$. 
    Furthermore, the full subquiver on all vertices but the ones corresponding to the above mentioned three tableaux is the initial quiver for $\flag_{2,4;n}$.
    
\begin{figure}
     \centering
 \adjustbox{scale=0.5}{\begin{tikzcd}
	8 &&& {} \\
	\\
	7 & {} \\
	\boxed{\begin{array}{c} \begin{matrix} 7 \\ 8 \end{matrix} \end{array}} && \boxed{\begin{array}{c} \begin{matrix} 1 \\ 8 \end{matrix} \end{array}} && \boxed{\begin{array}{c} \begin{matrix} 1 \\ 2 \end{matrix} \end{array}} \\
	6 &&&&&&&&&&& {} \\
	\boxed{\begin{array}{c} \begin{matrix} 1 &6 \\ 2&7 \\ 3 \\4 \\5 \\8 \end{matrix} \end{array}} && \begin{array}{c} \begin{matrix} 1 &1 \\ 2&7 \\ 3 \\4 \\5 \\8 \end{matrix} \end{array} &&&&&&&& \begin{array}{c} \begin{matrix} 1 \\2\\3 \\4 \\5 \\8 \end{matrix} \end{array} \\
	5 &&&&&&&&& {} \\
	\boxed{\begin{array}{c} \begin{matrix} 1 &5 \\ 2&6 \\ 3 \\4 \\7 \\8 \end{matrix} \end{array}} && \begin{array}{c} \begin{matrix} 1 &1 \\ 2&6 \\ 3 \\4 \\7 \\8 \end{matrix} \end{array} &&&&&& \begin{array}{c} \begin{matrix} 1 \\2\\3 \\4 \\7 \\8 \end{matrix} \end{array} \\
	4 &&&&&&& {} \\
	\boxed{\begin{array}{c} \begin{matrix} 1 &4 \\ 2&5 \\ 3 \\6 \\7 \\8 \end{matrix} \end{array}} && \begin{array}{c} \begin{matrix} 1 &1 \\ 2&5 \\ 3 \\6 \\7 \\8 \end{matrix} \end{array} &&&& \begin{array}{c} \begin{matrix} 1 \\2\\3 \\6 \\7 \\8 \end{matrix} \end{array} \\
	3 &&&&& {} \\
	\boxed{\begin{array}{c} \begin{matrix} 3 \\4\\5 \\6 \\7 \\8 \end{matrix} \end{array}} && \begin{array}{c} \begin{matrix} 1 \\4\\5 \\6 \\7 \\8 \end{matrix} \end{array} && \begin{array}{c} \begin{matrix} 1 \\2\\5 \\6 \\7 \\8 \end{matrix} \end{array} \\
	2 &&&&&&&&&&&&&&& {} \\
	\boxed{\begin{array}{c} \begin{matrix} 2\\3 \\4\\5 \\6 \\7  \end{matrix} \end{array}} && \begin{array}{c} \begin{matrix} 1\\3 \\4\\5 \\6 \\7  \end{matrix} \end{array} && \begin{array}{c} \begin{matrix} 1\\2 \\4\\5 \\6 \\7  \end{matrix} \end{array} && \begin{array}{c} \begin{matrix} 1\\2 \\3\\5 \\6 \\7  \end{matrix} \end{array} && \begin{array}{c} \begin{matrix} 1\\2 \\3\\4 \\6 \\7  \end{matrix} \end{array} && \begin{array}{c} \begin{matrix} 1\\2 \\3\\4 \\5 \\7  \end{matrix} \end{array} &&& \boxed{\begin{array}{c} \begin{matrix} 1\\2\\3 \\4\\5 \\6  \end{matrix} \end{array}} \\
	1 &&&&&&&&&&&&& {} \\
	& 1 && 2 && 3 && 4 && 5 && 6 && 7 && 8
	%\arrow[no head, from=1-1, to=1-4]
	%\arrow[no head, from=1-4, to=16-4]
    \arrow[teal, no head, from=1-1, to=16-4, rounded corners, to path=-| (\tikztotarget)]
	%\arrow[no head, from=3-1, to=3-2]
	%\arrow[no head, from=3-2, to=16-2]
    \arrow[teal, no head, from=3-1, to=16-2, rounded corners, to path=-| (\tikztotarget)]
	\arrow[from=4-3, to=6-3]
	\arrow[from=4-5, to=12-5]
	%\arrow[no head, from=5-1, to=5-12]
	%\arrow[no head, from=5-12, to=16-12]
    \arrow[teal, no head, from=5-1, to=16-12, rounded corners, to path=-| (\tikztotarget)]
	\arrow[from=6-1, to=6-3]
	\arrow[from=6-3, to=4-1]
	\arrow[from=6-3, to=6-11]
	\arrow[from=6-3, to=8-3]
	\arrow[from=6-11, to=4-3]
	\arrow[from=6-11, to=14-11]
	%\arrow[no head, from=7-1, to=7-10]
	%\arrow[no head, from=7-10, to=16-10]
    \arrow[teal, no head, from=7-1, to=16-10, rounded corners, to path=-| (\tikztotarget)]
	\arrow[from=8-1, to=8-3]
	\arrow[from=8-3, to=6-1]
	\arrow[from=8-3, to=8-9]
	\arrow[from=8-3, to=10-3]
	\arrow[from=8-9, to=6-3]
	\arrow[from=8-9, to=14-9]
	%\arrow[no head, from=9-1, to=9-8]
	%\arrow[no head, from=9-8, to=16-8]
    \arrow[teal, no head, from=9-1, to=16-8, rounded corners, to path=-| (\tikztotarget)]
	\arrow[from=10-1, to=10-3]
	\arrow[from=10-3, to=4-5]
	\arrow[from=10-3, to=8-1]
	\arrow[from=10-3, to=10-7]
	\arrow[from=10-3, to=12-3]
	\arrow[from=10-7, to=8-3]
	\arrow[from=10-7, to=14-7]
	%\arrow[no head, from=11-1, to=11-6]
	%\arrow[no head, from=11-6, to=16-6]
    \arrow[teal, no head, from=11-1, to=16-6, rounded corners, to path=-| (\tikztotarget)]
	\arrow[from=12-1, to=12-3]
	\arrow[from=12-3, to=10-1]
	\arrow[from=12-3, to=12-5]
	\arrow[from=12-3, to=14-3]
	\arrow[from=12-5, to=10-3]
	\arrow[from=12-5, to=14-5]
	%\arrow[no head, from=13-1, to=13-16]
	%\arrow[no head, from=13-16, to=16-16]
    \arrow[teal, no head, from=13-1, to=16-16, rounded corners, to path=-| (\tikztotarget)]
	\arrow[from=14-1, to=14-3]
	\arrow[from=14-3, to=12-1]
	\arrow[from=14-3, to=14-5]
	\arrow[from=14-5, to=12-3]
	\arrow[from=14-5, to=14-7]
	\arrow[from=14-7, to=12-5]
	\arrow[from=14-7, to=14-9]
	\arrow[from=14-9, to=10-7]
	\arrow[from=14-9, to=14-11]
	\arrow[from=14-11, to=8-9]
	\arrow[from=14-11, to=14-14]
	%\arrow[no head, from=15-1, to=15-14]
	%\arrow[no head, from=15-14, to=16-14]
    \arrow[teal, no head, from=15-1, to=16-14, rounded corners, to path=-| (\tikztotarget)]
\end{tikzcd}
}
 \caption{The initial seed for the partial flag variety $\flag_{2,6;8}$.}
     \label{fig:initial seed for Fl268}
 \end{figure}

\subsection{Spinor helicity varieties} \label{subsec:Spinor helicity varieties}
As already mentioned in the introduction the spinor helicity variety $\mathcal{SH}_n$ is isomorphic to the partial flag variety $\flag_{2,n-2;n}$ with isomorphism given in \eqref{eq:SH to flag2,n-2,n}. The initial seed for $n=8$ and $\flag_{2,6;8}$ is displayed in Figure~\ref{fig:initial seed for Fl268}. 
%Recall, that the images of the initial tableaux under \eqref{eq:embed} are tableaux with six rows obtained by applying $\phi$ defined in \eqref{eq:map of tableaux}, which amounts to filling up columns with two rows by adding the numbers $n+1,\dots,2n-4$.

{\bf The initial quiver} for $\Gr_{n-2;2n-4}$ consists of a $(n-2)\times(n-2)$-grid with one additional vertex at the south-east corner. The first row and the first column of the grid consist of frozen vertices.  All mutations we perform take place in a $(n-5)\times (n-4)$ grid that is north-west bound inside the initial quiver. We call the latter the \emph{mutation grid}. An example for $n=8$ and $\Gr_{6;12}$ is given in Figure~\ref{fig:initial seed for Gr612}. 

\begin{figure}
    \centering
\adjustbox{scale=0.5}{
\begin{tikzcd}
	\begin{array}{c} {\color{teal}\boxed{\begin{matrix} 7 \\ 8 \\ 9 \\ 10 \\ 11 \\ 12 \end{matrix}}} (36) \end{array} & \begin{array}{c} {\color{teal}\boxed{\begin{matrix} 1 \\  8 \\ 9 \\ 10 \\ 11 \\ 12 \end{matrix}}} (30) \end{array} & \begin{array}{c} {\color{teal}\boxed{\begin{matrix} 1 \\ 2 \\ 9 \\ 10 \\ 11 \\ 12 \end{matrix}}} (24) \end{array} & \begin{array}{c} \boxed{\begin{matrix} 1 \\ 2 \\ 3 \\ 10 \\ 11 \\ 12 \end{matrix}} (18) \end{array} & \begin{array}{c} \boxed{\begin{matrix} 1 \\ 2 \\ 3 \\ 4 \\ 11 \\ 12 \end{matrix}} (12) \end{array} & \begin{array}{c} \boxed{\begin{matrix} 1 \\ 2 \\ 3 \\ 4 \\ 5 \\ 12 \end{matrix}} (6) \end{array} \\
	\begin{array}{c} \boxed{\begin{matrix} 6 \\ 7 \\ 8 \\ 9 \\ 10 \\ 11 \end{matrix}} (35) \end{array} & \begin{array}{c} {\color{red}\begin{matrix} 1 \\  7 \\ 8 \\ 9 \\ 10 \\ 11 \end{matrix}} (29) \end{array} & \begin{array}{c} {\color{red}\begin{matrix} 1 \\ 2 \\ 8 \\ 9 \\ 10 \\ 11 \end{matrix}} (23) \end{array} & \begin{array}{c} {\color{red}\begin{matrix} 1 \\ 2 \\ 3 \\ 9 \\ 10 \\ 11 \end{matrix}} (17) \end{array} & \begin{array}{c} {\color{red}\begin{matrix} 1 \\ 2 \\ 3 \\ 4 \\ 10 \\ 11 \end{matrix}} (11) \end{array} & \begin{array}{c} \begin{matrix} 1 \\ 2 \\ 3 \\ 4 \\ 5 \\ 11 \end{matrix} (5) \end{array} \\
	\begin{array}{c} \boxed{\begin{matrix} 5 \\ 6 \\ 7 \\ 8 \\ 9 \\ 10 \end{matrix}} (34) \end{array} & \begin{array}{c} {\color{red}\begin{matrix} 1 \\  6 \\ 7 \\ 8 \\ 9 \\ 10 \end{matrix}} (28) \end{array} & \begin{array}{c} {\color{red}\begin{matrix} 1 \\ 2 \\ 7 \\ 8 \\ 9 \\ 10 \end{matrix}} (22) \end{array} & \begin{array}{c} {\color{red}\begin{matrix} 1 \\ 2 \\ 3 \\ 8 \\ 9 \\ 10 \end{matrix}} (16) \end{array} & \begin{array}{c} {\color{red}\begin{matrix} 1 \\ 2 \\ 3 \\ 4 \\ 9 \\ 10 \end{matrix}} (10) \end{array} & \begin{array}{c} \begin{matrix} 1 \\ 2 \\ 3 \\ 4 \\ 5 \\ 10 \end{matrix} (4) \end{array} \\
	\begin{array}{c} \boxed{\begin{matrix} 4 \\ 5 \\ 6 \\ 7 \\ 8 \\ 9 \end{matrix}} (33) \end{array} & \begin{array}{c} {\color{red}\begin{matrix} 1 \\  5 \\ 6 \\ 7 \\ 8 \\ 9 \end{matrix}} (27) \end{array} & \begin{array}{c} {\color{red}\begin{matrix} 1 \\ 2 \\  6 \\ 7 \\ 8 \\ 9 \end{matrix}} (21) \end{array} & \begin{array}{c} {\color{red}\begin{matrix} 1 \\ 2 \\ 3 \\ 7 \\ 8 \\ 9 \end{matrix}} (15) \end{array} & \begin{array}{c} {\color{red}\begin{matrix} 1 \\ 2 \\ 3 \\ 4 \\ 8 \\ 9 \end{matrix}} (9) \end{array} & \begin{array}{c} \begin{matrix} 1 \\ 2 \\ 3 \\ 4 \\ 5 \\ 9 \end{matrix} (3) \end{array} \\
	\begin{array}{c} {\color{teal}\boxed{\begin{matrix} 3 \\ 4 \\ 5 \\ 6 \\ 7 \\ 8 \end{matrix}}} (32) \end{array} & \begin{array}{c} {\color{teal}\begin{matrix} 1 \\ 4 \\ 5 \\ 6 \\ 7 \\ 8 \end{matrix}} (26) \end{array} & \begin{array}{c} {\color{teal}\begin{matrix} 1 \\ 2 \\ 5 \\ 6 \\ 7 \\ 8 \end{matrix}} (20) \end{array} & \begin{array}{c} {\color{teal}\begin{matrix} 1 \\ 2 \\ 3 \\ 6 \\ 7 \\ 8 \end{matrix}} (14) \end{array} & \begin{array}{c} {\color{teal}\begin{matrix} 1 \\ 2 \\ 3 \\ 4 \\ 7 \\ 8 \end{matrix}} (8) \end{array} & \begin{array}{c} {\color{teal}\begin{matrix} 1 \\ 2 \\ 3 \\ 4 \\ 5 \\ 8 \end{matrix}} (2) \end{array} \\
	\begin{array}{c} {\color{teal}\boxed{\begin{matrix} 2 \\ 3 \\ 4 \\ 5 \\ 6 \\ 7 \end{matrix}}} (31) \end{array} & \begin{array}{c} {\color{teal}\begin{matrix} 1 \\ 3 \\ 4 \\ 5 \\ 6 \\ 7 \end{matrix}} (25) \end{array} & \begin{array}{c} {\color{teal}\begin{matrix} 1 \\ 2 \\ 4 \\ 5 \\ 6 \\ 7 \end{matrix}} (19) \end{array} & \begin{array}{c} {\color{teal}\begin{matrix} 1 \\ 2 \\ 3 \\ 5 \\ 6 \\ 7 \end{matrix}} (13) \end{array} & \begin{array}{c} {\color{teal}\begin{matrix} 1 \\ 2 \\ 3 \\ 4 \\ 6 \\ 7 \end{matrix}} (7) \end{array} & \begin{array}{c} {\color{teal}\begin{matrix} 1 \\ 2 \\ 3 \\ 4 \\ 5 \\ 7 \end{matrix}} (1) \end{array} & \begin{array}{c} {\color{teal}\boxed{\begin{matrix} 1 \\ 2 \\ 3 \\ 4 \\ 5 \\ 6 \end{matrix}}} (37) \end{array}
	\arrow[from=1-2, to=2-2]
	\arrow[from=1-3, to=2-3]
	\arrow[from=1-4, to=2-4]
	\arrow[from=1-5, to=2-5]
	\arrow[from=1-6, to=2-6]
	\arrow[from=2-1, to=2-2]
	\arrow[from=2-2, to=1-1]
	\arrow[from=2-2, to=2-3]
	\arrow[from=2-2, to=3-2]
	\arrow[from=2-3, to=1-2]
	\arrow[from=2-3, to=2-4]
	\arrow[from=2-3, to=3-3]
	\arrow[from=2-4, to=1-3]
	\arrow[from=2-4, to=2-5]
	\arrow[from=2-4, to=3-4]
	\arrow[from=2-5, to=1-4]
	\arrow[from=2-5, to=2-6]
	\arrow[from=2-5, to=3-5]
	\arrow[from=2-6, to=1-5]
	\arrow[from=2-6, to=3-6]
	\arrow[from=3-1, to=3-2]
	\arrow[from=3-2, to=2-1]
	\arrow[from=3-2, to=3-3]
	\arrow[from=3-2, to=4-2]
	\arrow[from=3-3, to=2-2]
	\arrow[from=3-3, to=3-4]
	\arrow[from=3-3, to=4-3]
	\arrow[from=3-4, to=2-3]
	\arrow[from=3-4, to=3-5]
	\arrow[from=3-4, to=4-4]
	\arrow[from=3-5, to=2-4]
	\arrow[from=3-5, to=3-6]
	\arrow[from=3-5, to=4-5]
	\arrow[from=3-6, to=2-5]
	\arrow[from=3-6, to=4-6]
	\arrow[from=4-1, to=4-2]
	\arrow[from=4-2, to=3-1]
	\arrow[from=4-2, to=4-3]
	\arrow[from=4-2, to=5-2]
	\arrow[from=4-3, to=3-2]
	\arrow[from=4-3, to=4-4]
	\arrow[from=4-3, to=5-3]
	\arrow[from=4-4, to=3-3]
	\arrow[from=4-4, to=4-5]
	\arrow[from=4-4, to=5-4]
	\arrow[from=4-5, to=3-4]
	\arrow[from=4-5, to=4-6]
	\arrow[from=4-5, to=5-5]
	\arrow[from=4-6, to=3-5]
	\arrow[from=4-6, to=5-6]
	\arrow[from=5-1, to=5-2]
	\arrow[from=5-2, to=4-1]
	\arrow[from=5-2, to=5-3]
	\arrow[from=5-2, to=6-2]
	\arrow[from=5-3, to=4-2]
	\arrow[from=5-3, to=5-4]
	\arrow[from=5-3, to=6-3]
	\arrow[from=5-4, to=4-3]
	\arrow[from=5-4, to=5-5]
	\arrow[from=5-4, to=6-4]
	\arrow[from=5-5, to=4-4]
	\arrow[from=5-5, to=5-6]
	\arrow[from=5-5, to=6-5]
	\arrow[from=5-6, to=4-5]
	\arrow[from=5-6, to=6-6]
	\arrow[from=6-1, to=6-2]
	\arrow[from=6-2, to=5-1]
	\arrow[from=6-2, to=6-3]
	\arrow[from=6-3, to=5-2]
	\arrow[from=6-3, to=6-4]
	\arrow[from=6-4, to=5-3]
	\arrow[from=6-4, to=6-5]
	\arrow[from=6-5, to=5-4]
	\arrow[from=6-5, to=6-6]
	\arrow[from=6-6, to=5-5]
	\arrow[from=6-6, to=6-7]
\end{tikzcd} }
    \caption{The initial seed for the Grassmannian $\Gr_{6,12}$. The central $3\times4$-grid from vertex (9) to (29) marked in \textcolor{red}{red} is the mutation grid. After the mutation sequence we freeze vertices (29), (16) and (10). 
    Vertices marked in \textcolor{teal}{teal} are those that  already correspond to initial cluster variables for $\mathcal F_{2,6;8}$.}
    \label{fig:initial seed for Gr612}
\end{figure}

%{\bf The mutation sequence} is similar to a maximal green sequence for the Grassmannian \cite[\S11]{Marsh-Scott} and we use similar notation: a \emph{page} in the sequence refers to a subsequence of consecutive mutations so that each vertex in the mutation grid is mutated at most once. 

Let $a=d_2-d_1-1=n-5$, $b=d_2-2=n-4$, $c=d_1=2$. The mutation sequence is as follows. 
\begin{itemize}
\item The first page: 
\begin{align*}
& (a,1), (a,2), \ldots, (a,b), \\
& (a-1,1), (a-1,2), \ldots, (a-1,b), \\
& \vdots \\
& (2,1), (2,2), \ldots, (2,b), \\
& (1,1), (1,2).
\end{align*}
The cluster variable $\begin{matrix}
1 & 4 \\ 2 & 5 \\ 3 \\ 6 \\ \vdots \\ n
\end{matrix}$ appears at position $(1,2)$. 

\item The second page: 
\begin{align*}
& (a,1), (a,2), \ldots, (a,b), \\
& (a-1,1), (a-1,2), \ldots, (a-1,b), \\
& \vdots \\
& (3,1), (3,2), \ldots, (3,b), \\
& (2,1), (2,2), \\
& (1,1).
\end{align*}
The two cluster variables $\begin{matrix}
1 & 5 \\ \vdots & 6 \\ 4 \\ 7 \\ \vdots \\ n
\end{matrix}$ and $\begin{matrix}
1 & 1 \\ 2 & 5 \\  3  \\ 6 \\ \vdots \\ n
\end{matrix}$ appears at positions $(2,2)$ and $(1,1)$ respectively.

\item Continue this procedure. At the $i$th ($i \in [1,n-4]$) page:
\begin{align*}
& (a,1), (a,2), \ldots, (a,b), \\
& (a-1,1), (a-1,2), \ldots, (a-1,b), \\
& \vdots \\
& (i+1,1), (i+1,2), \ldots, (i+1,b), \\
& (i,1), (i,2), \\
& (i-1,1).
\end{align*}
where the last line is empty if $i=1$. The cluster variables 
$$
\begin{matrix}
1 & 1 \\ \vdots & \vdots \\ \vdots & j_2 \\ j_1+3 & j_1+j_2+4 \\ j_1+6 & \vdots \\ \vdots & j_1+5 \\  n
\end{matrix}
$$ 
appears at positions $(j_1+1, 2-j_2)$, $j_1+j_2=i-1$, $j_1,j_2 \in \ZZ_{\ge 0}$, $ j_2 \le 1$, $i \in [1, n-4]$, respectively.
\end{itemize}

The number of mutation steps is $\frac{1}{2}\left(n - 5\right)\, \left(n^2 - 10\, n + 30\right)$, see Lemma \ref{lem:number of mutation steps for 2 step partial flag varieties}. We freeze the cluster variables at positions $(i,2)$, $i \in [a]$, $a=n-5$. Let $s'$ denote the seed of the Grassmannian obtained from the initial seed by applying the above mutation sequence and let $Q'$ be its quiver. The images of the initial cluster variables for $\flag_{2,n-2;n}$ are contained in the set of cluster variables of $s'$. The seed $(Q_{2,n-2;n},\varphi^*({\bf x}_{2,n-2;n}))$ is a restricted seed of $s'$. We can delete $(n-3)(n-5)$ vertices to obtain the initial seed for $\flag_{2,n-2;n}$. The final seed of the example $\flag_{2,6;8}$ is shown in Figure \ref{fig:mut seq G612}. 

\begin{figure}
     \centering
 \adjustbox{scale=0.5}{\begin{tikzcd}
	{(4)} && {\boxed{(34)}} && {\boxed{(35)}} & {\boxed{\textcolor{red}{(36)}}} & {\boxed{\textcolor{red}{(30)}}} & {\boxed{\textcolor{red}{(24)}}} \\
	&&& {(3)} \\
	&& {(9)} && {(15)} & {\boxed{\textcolor{red}{(21)}}} & {\textcolor{red}{(27)}} &&&& {\textcolor{red}{(2)}} \\
	& {\boxed{(33)}} & {(10)} && {(16)} & {\boxed{\textcolor{red}{(22)}}} & {\textcolor{red}{(28)}} &&& {\textcolor{red}{(8)}} \\
	{(5)} && {(11)} && {(17)} & {\boxed{\textcolor{red}{(23)}}} & {\textcolor{red}{(29)}} && {\textcolor{red}{(14)}} \\
	{\boxed{(6)}} && {\boxed{(12)}} && {\boxed{(18)}} & {\boxed{\textcolor{red}{(32)}}} & {\textcolor{red}{(26)}} & {\textcolor{red}{(20)}} \\
	&&&&& {\boxed{\textcolor{red}{(31)}}} & {\textcolor{red}{(25)}} & {\textcolor{red}{(19)}} & {\textcolor{red}{(13)}} & {\textcolor{red}{(7)}} & {\textcolor{red}{(1)}} & {\boxed{\textcolor{red}{(37)}}}
	\arrow[from=1-1, to=3-3]
	\arrow[from=1-3, to=1-1]
	\arrow[from=1-3, to=1-5]
	\arrow[from=1-5, to=1-6]
	\arrow[from=1-5, to=3-5]
	\arrow[from=1-6, to=3-6]
	\arrow[from=1-7, to=3-7]
	\arrow[from=1-8, to=1-7]
	\arrow[from=1-8, to=6-8]
	\arrow[from=2-4, to=3-5]
	\arrow[from=3-3, to=1-3]
	\arrow[from=3-3, to=2-4]
	\arrow[from=3-3, to=4-3]
	\arrow[from=3-5, to=3-3]
	\arrow[from=3-5, to=3-6]
	\arrow[from=3-5, to=4-5]
	\arrow[from=3-6, to=1-5]
	\arrow[from=3-6, to=3-7]
	\arrow[from=3-6, to=4-6]
	\arrow[from=3-7, to=1-6]
	\arrow[from=3-7, to=3-11]
	\arrow[from=3-7, to=4-7]
	\arrow[from=3-11, to=1-7]
	\arrow[from=3-11, to=7-11]
	\arrow[from=4-2, to=3-3]
	\arrow[from=4-3, to=3-5]
	\arrow[from=4-3, to=4-2]
	\arrow[from=4-3, to=5-3]
	\arrow[from=4-5, to=4-3]
	\arrow[from=4-5, to=4-6]
	\arrow[from=4-5, to=5-5]
	\arrow[from=4-6, to=3-5]
	\arrow[from=4-6, to=4-7]
	\arrow[from=4-6, to=5-6]
	\arrow[from=4-7, to=3-6]
	\arrow[from=4-7, to=4-10]
	\arrow[from=4-7, to=5-7]
	\arrow[from=4-10, to=3-7]
	\arrow[from=4-10, to=7-10]
	\arrow[from=5-1, to=1-1]
	\arrow[from=5-1, to=6-3]
	\arrow[from=5-3, to=4-5]
	\arrow[from=5-3, to=5-1]
	\arrow[from=5-3, to=6-5]
	\arrow[from=5-5, to=5-3]
	\arrow[from=5-5, to=5-6]
	\arrow[from=5-6, to=4-5]
	\arrow[from=5-6, to=5-7]
	\arrow[from=5-7, to=1-8]
	\arrow[from=5-7, to=4-6]
	\arrow[from=5-7, to=5-9]
	\arrow[from=5-7, to=6-7]
	\arrow[from=5-9, to=4-7]
	\arrow[from=5-9, to=7-9]
	\arrow[from=6-1, to=5-1]
	\arrow[from=6-3, to=5-3]
	\arrow[from=6-5, to=5-5]
	\arrow[from=6-6, to=6-7]
	\arrow[from=6-7, to=5-6]
	\arrow[from=6-7, to=6-8]
	\arrow[from=6-7, to=7-7]
	\arrow[from=6-8, to=5-7]
	\arrow[from=6-8, to=7-8]
	\arrow[from=7-6, to=7-7]
	\arrow[from=7-7, to=6-6]
	\arrow[from=7-7, to=7-8]
	\arrow[from=7-8, to=6-7]
	\arrow[from=7-8, to=7-9]
	\arrow[from=7-9, to=6-8]
	\arrow[from=7-9, to=7-10]
	\arrow[from=7-10, to=5-9]
	\arrow[from=7-10, to=7-11]
	\arrow[from=7-11, to=4-10]
	\arrow[from=7-12, to=7-11]
\end{tikzcd}
}
 \caption{The quiver $Q'$ obtained from the initial seed in Figure \ref{fig:initial seed for Gr612} after the mutation sequence (27), (21), (15), (9), (28), (22), (16), (10), (29), (23), (27), (21), (15), (9), (28), (22), (29), (27), (21), (28), (27) and freezing vertices (21), (22), (23). All colored vertices on the RHS yield the restricted seed corresponding to the initial seed for $\flag_{2,6;8}$ (compare to Figure~\ref{fig:initial seed for Fl268}). %\textcolor{red}{Check Lauren's book for the technical term of a quiver that splits into two parts separated by frozens.}
 }
     \label{fig:mut seq G612}
 \end{figure}

\footnotesize{
\bibliographystyle{alpha}
\bibliography{Bib}
}

\end{document}